\documentclass[10pt,reqno]{amsart}

\usepackage[utf8x]{inputenc}
\usepackage[english]{babel}
\usepackage{amsmath,amsfonts,amssymb,amsthm,shuffle}
\usepackage[T1]{fontenc}
\usepackage{lmodern}

\usepackage[top=3.5cm,bottom=3.5cm,left=3.4cm,right=3.4cm]{geometry}

\usepackage[dvipsnames]{xcolor}
\usepackage[hyperindex=true,frenchlinks=true,colorlinks=true,
citecolor=NavyBlue,linkcolor=Mulberry,urlcolor=LimeGreen,linktocpage]{hyperref}

\usepackage{tikz}

\usepackage{cite}
\usepackage{enumitem}
\usepackage{dsfont}

\numberwithin{equation}{subsection}

\newtheorem{Theoreme}{Theorem}[section]
\newtheorem{Proposition}[Theoreme]{Proposition}
\newtheorem{Lemme}[Theoreme]{Lemma}

\title{Combinatorial Hopf algebras from PROs}
\date{\today}
\author{Jean-Paul Bultel}
\address{Laboratoire d'Informatique, du Traitement de
    l'Information et des Systèmes, Université de Rouen, Avenue de
    l'Université, 76801 Saint-Étienne-du-Rouvray cedex, France.}
\email{jean-paul.bultel@univ-rouen.fr}
\author{Samuele Giraudo}
\thanks{Phone number and email address of the corresponding author:
+33160957558, {\tt samuele.giraudo@u-pem.fr}}
\address{Laboratoire d'Informatique Gaspard-Monge, Université
    Paris-Est Marne-la-Vallée, 5 boulevard Descartes, Champs-sur-Marne,
    77454 Marne-la-Vallée cedex 2, France.}
\email{samuele.giraudo@u-pem.fr}

\renewcommand{\leq}{\leqslant}
\renewcommand{\geq}{\geqslant}

\newcommand{\Pca}{\mathcal{P}}

\newcommand{\Hca}{\mathcal{H}}
\newcommand{\Oca}{\mathcal{O}}
\newcommand{\Mca}{\mathcal{M}}
\newcommand{\Unite}{\mathds{1}}
\newcommand{\In}{\operatorname{i}}
\newcommand{\Out}{\operatorname{o}}
\newcommand{\Sbf}{\mathbf{S}}

\newcommand{\Tbf}{\mathbf{T}}
\newcommand{\Rit}{\mathit{R}}
\newcommand{\Tit}{\mathit{T}}
\newcommand{\PvH}{\mathsf{H}}
\newcommand{\OvH}{\mathit{H}}
\newcommand{\OvP}{\mathsf{R}}
\newcommand{\MvP}{\mathsf{B}}
\newcommand{\Vect}{\operatorname{Vect}}
\newcommand{\Reduit}{\operatorname{red}}

\newcommand{\C}{\mathbb{C}}
\newcommand{\R}{\mathbb{R}}
\newcommand{\N}{\mathbb{N}}
\newcommand{\EnsNat}{\mathbb{N}}

\newcommand{\PRT}{\mathsf{PRT}}
\newcommand{\PRF}{\mathsf{PRF}}
\newcommand{\Dec}{\operatorname{dec}}
\newcommand{\La}{\mathtt{a}}
\newcommand{\Lb}{\mathtt{b}}
\newcommand{\Adm}{\mathrm{Adm}}
\newcommand{\Sat}{\mathsf{S}}
\newcommand{\Op}{\mathrm{op}}
\newcommand{\Free}{\mathrm{Free}}
\newcommand{\Rev}{\mathrm{rev}}

\newcommand{\CK}{\mathbf{CK}}
\newcommand{\Sym}{\mathit{Sym}}
\newcommand{\SymNC}{\mathbf{Sym}}
\newcommand{\FdB}{\mathit{FdB}}
\newcommand{\FdBNC}{\mathbf{FdB}}

\newcommand{\AB}{\mathsf{AB}}
\newcommand{\As}{\mathsf{As}}
\newcommand{\Heap}{\mathsf{Heap}}
\newcommand{\FHeap}{\mathsf{FHeap}}
\newcommand{\FBT}{\mathsf{FBT}}
\newcommand{\BAs}{\mathsf{BAs}}

\newcommand{\Sloane}[1]{\href{http://oeis.org/#1}{{\bf #1}}}

\definecolor{Noir}{RGB}{0,0,0}
\definecolor{Blanc}{RGB}{255,255,255}
\definecolor{Rouge}{RGB}{205,35,38}
\definecolor{Bleu}{RGB}{2,60,195}
\definecolor{Vert}{RGB}{23,163,1}
\definecolor{Violet}{RGB}{181,18,225}
\definecolor{Orange}{RGB}{255,113,15}
\definecolor{Marron}{RGB}{52,46,0}
\definecolor{Bordeaux}{RGB}{149,14,0}

\tikzstyle{Noeud} = [circle,draw=Bleu!80,fill=Bleu!20,thick,inner sep=0pt,
minimum size=10mm,line width=4pt,font=\Huge]
\tikzstyle{Arete}=[Rouge!80,cap=round,line width=3pt]
\tikzstyle{Feuille}=[rectangle,draw=Noir!70,fill=Noir!20,minimum size=3mm,
line width=2pt]
\tikzstyle{Operateur} = [rectangle,rounded corners,draw=Bleu!100,fill=Bleu!10,
minimum size=15mm,line width=3pt,font=\Huge]
\tikzstyle{Marque1} = [draw=Vert!100,fill=Vert!30]
\tikzstyle{Marque2} = [draw=Orange!100,fill=Orange!40]
\tikzstyle{Marque3} = [draw=Marron!100,fill=Marron!20]
\tikzstyle{Marque4} = [draw=Bordeaux!100,fill=Bordeaux!30]
\tikzstyle{Boite} = [rectangle,rounded corners,draw=Bleu!100,fill=Bleu!10,
minimum width=1cm,minimum height=1cm,line width=3pt,font=\Huge,text centered]
\tikzstyle{Domino3} = [rectangle,draw=Bordeaux!100,fill=Bordeaux!30,
minimum width=3cm,minimum height=.5cm,line width=2pt,font=\Huge]
\tikzstyle{Domino2} = [rectangle,draw=Bordeaux!100,fill=Bordeaux!30,
minimum width=2cm,minimum height=.5cm,line width=2pt,font=\Huge]
\tikzstyle{Domino1} = [rectangle,draw=Vert!100,fill=Vert!30,
minimum width=1cm,minimum height=.5cm,line width=2pt,font=\Huge]
\tikzstyle{Injection} = [Noir!100,draw,>->]
\tikzstyle{Surjection} = [Noir!100,draw,->>]

\begin{document}
\maketitle

\begin{abstract}
    We introduce a general construction that takes as input a so-called
    stiff PRO and that outputs a Hopf algebra. Stiff PROs are particular
    PROs that can be described by generators and relations with precise
    conditions. Our construction generalizes the classical construction
    from operads to Hopf algebras of van der Laan. We study some of its
    properties and review some examples of application. We get in particular
    Hopf algebras on heaps of pieces and retrieve some deformed versions
    of the noncommutative Fa\`a di Bruno algebra introduced by Foissy.
\end{abstract}

\tableofcontents

\section*{Introduction}
Operads are algebraic structures introduced in the 70s  by Boardman and
Vogt \cite{BV73} and by May \cite{May72} in the context of algebraic
topology to offer a formalization of the notion of operators and their
composition (see \cite{Mar08} and \cite{LV12} for a very complete
presentation of the theory of operads). Operads provide therefore a unified
framework to study some sorts of {\em a priori} very different algebras,
such as associative algebras, Lie algebras, and commutative algebras. Besides,
the theory of operads is also beneficial in combinatorics \cite{Cha08}
since it provides some ways to decompose combinatorial objects into
elementary pieces. On the other hand, the theory of Hopf algebras holds a
special place in algebraic combinatorics \cite{JR79}. In recent years, many
Hopf algebras were defined and studied, and most of these involve very famous
combinatorial objects such as permutations \cite{MR95,DHT02}, standard
Young tableaux \cite{PR95,DHT02}, or binary trees \cite{LR98,HNT05}.
\smallskip

These two theories ---operads and Hopf algebras--- have several
interactions. One of these is a construction \cite{Vdl04} taking an operad
$\Oca$ as input and producing a bialgebra $\OvH(\Oca)$ as output, which is
called the {\em natural bialgebra} of $\Oca$. This construction has been
studied in some recent works: in \cite{CL07}, it is shown that $\OvH$ can
be rephrased in terms of an incidence Hopf algebra of a certain family of
posets, and in \cite{ML13}, a general formula for its antipode is
established. Let us also cite \cite{Fra08} in which this construction is
considered to study series of trees.
\smallskip

The initial motivation of our work was to generalize this $\OvH$
construction with the aim of constructing some new and interesting Hopf
algebras. The direction we have chosen is to start with {\em PROs},
algebraic structures which generalize operads in the sense that PROs deal
with operators with possibly several outputs. Surprisingly, these structures
appeared earlier than operads in the work of Mac Lane \cite{McL65}.
Intuitively, a PRO $\Pca$ is a set of operators together with two
operations: an horizontal composition and a vertical composition. The first
operation takes two operators $x$ and $y$ of $\Pca$ and builds a new one
whose inputs (resp. outputs) are, from left to right, those of $x$ and
then those of $y$. The second operation takes two operators $x$ and $y$
of $\Pca$ and produces a new one obtained by plugging the outputs of $y$
onto the inputs of $x$. Basic and modern references about PROs
are \cite{Lei04} and \cite{Mar08}.
\smallskip

Our main contribution consists in the definition of a new construction
$\PvH$ from PROs to bialgebras. Roughly speaking, the construction $\PvH$
can be described as follows. Given a PRO $\Pca$ satisfying some mild
properties, the bialgebra $\PvH \Pca$ has bases indexed by a particular
subset of elements of $\Pca$. The product of $\PvH \Pca$ is the horizontal
composition of $\Pca$ and the coproduct of $\PvH \Pca$ is defined from
the vertical composition of $\Pca$, enabling to separate a basis element
into two smaller parts. The properties satisfied by $\Pca$ imply, in a
nontrivial way, that the product and the coproduct of $\PvH \Pca$ satisfy
the required axioms to be a bialgebra. This construction generalizes $\OvH$
and establishes a new connection between the theory of PROs and the theory
of Hopf algebras.
\smallskip

Our results are organized as follows. In
Section \ref{sec:background}, we recall some general background about
Hopf algebras, operads, and PROs. In particular, we give a description
of free PROs in terms of {\em prographs}, similar to the one of
Lafont \cite{Laf11}. We also recall the natural bialgebra construction
of an operad. We provide in Section \ref{sec:construction_H} the
description of our new construction $\PvH$. A first version of this
construction is presented, associating a bialgebra $\PvH(\Pca)$ with
a free PRO $\Pca$. We then present an extended version of the construction,
taking as input non-necessarily free PROs satisfying some properties,
called {\em stiff PROs}. Next, we consider two well-known constructions
of PROs \cite{Mar08}, one, $\OvP$, taking as input operads and the other,
$\MvP$, taking as input monoids. We prove that under some mild conditions,
these constructions produce stiff PROs. We establish that the natural
bialgebra of an operad can be reformulated as a particular case of our
construction $\PvH$. We conclude by giving some examples of application
of $\PvH$ in Section \ref{sec:exemples}. The Hopf algebras that we obtain
are very similar to the Connes-Kreimer Hopf algebra \cite{CK98} in the
sense that their coproduct can be computed by means of admissible cuts
in various combinatorial objects. From very simple stiff PROs, we
reconstruct the Hopf algebra of noncommutative symmetric functions
$\SymNC$ \cite{GKLLRT95,KLT97} and the noncommutative Fàa di Bruno
algebra $\FdBNC_1$ \cite{BFK06}. Besides, we present a way of using
$\PvH$ to reconstruct some of the Hopf algebras $\FdBNC_\gamma$, a
$\gamma$-deformation of $\FdBNC_1$ introduced by Foissy \cite{Foi08}. We
also obtain several other Hopf algebras, which, respectively, involve forests
of planar rooted trees, some kind of graphs consisting of nodes with one
parent and several children or several parents and one child that we call
{\em forests of bitrees}, heaps of pieces (see \cite{Vie86} for a general
presentation of these combinatorial objects), and a particular class of
heaps of pieces that we call {\em heaps of friable pieces}. All these
Hopf algebras depend on a nonnegative integer as parameter.
\smallskip

{\it Acknowledgments.} The authors would like to thank Jean-Christophe
Novelli for his suggestions during the preparation of this paper.
Moreover, the authors would like to thank the anonymous referee for its
useful suggestions, improving the paper. The computations of this work
have been done with the open-source mathematical software Sage~\cite{Sage}
and one of its extensions, Sage-Combinat~\cite{SageC}.
\medskip

{\it Notations.} For any integer $n \geq 0$, $[n]$ denotes the set
$\{1, \dots, n\}$. If $u$ is a word and $i$ is a positive integer no
greater than the length of $u$, $u_i$ denotes the $i$-th letter of $u$.
The empty word is denoted by $\epsilon$.
\medskip

\section{Algebraic structures and background}
\label{sec:background}
We recall in this preliminary section some basics about the algebraic
structures in play in all this work, {\em i.e.}, Hopf algebras, operads,
and PROs. We also present some well-known Hopf algebras and recall the
construction associating a combinatorial Hopf algebra with an operad.
\medskip

\subsection{Combinatorial Hopf algebras}
In the sequel, all vector spaces have $\C$ as ground field. By {\em algebra}
we mean a unitary associative algebra and by {\em coalgebra} a
counitary coassociative coalgebra. We call {\em combinatorial Hopf algebra}
any graded bialgebra $\Hca = \bigoplus_{n \geq 0} \Hca_n$ such that for
any $n \geq 1$, the $n$-th homogeneous component $\Hca_n$ of $\Hca$ has
finite dimension and the dimension of $\Hca_0$ is $1$. The {\em degree}
of any element $x \in \Hca_n$ is $n$ and is denoted by $\deg(x)$.
Combinatorial Hopf algebras are Hopf algebras because the antipode can
be defined recursively degree by degree. Let us now review some classical
combinatorial Hopf algebras which play an important role in this work.
\medskip

\subsubsection{Faà di Bruno algebra and its deformations}
Let $\FdB$ be the free commutative algebra generated by elements $h_n$,
$n \geq 1$ with $\deg(h_n) = n$. The bases of $\FdB$ are thus
indexed by integer partitions, and the unit is denoted by $h_0$. This is the
{\em algebra of symmetric functions} \cite{Mcd95}. There are several ways
to endow $\FdB$ with a coproduct to turn it into a Hopf algebra. In \cite{Foi08},
Foissy obtains, as a byproduct of his investigation of combinatorial
Dyson-Schwinger equations in the Connes-Kreimer algebra, a one-parameter family $\Delta_\gamma$,
$\gamma \in \R$, of coproducts on $\FdB$, defined by using
alphabet transformations (see \cite{Mcd95}), by
\begin{equation} \label{equ:coproduit_delta_gamma}
    \Delta_\gamma(h_n) :=
    \sum_{k=0}^n h_k \otimes h_{n-k}((k\gamma +1)X),
\end{equation}
where, for any $\alpha\in\R$ and $n\in\N$, $h_n(\alpha X)$
is the coefficient of $t^n$ in $\left(\sum_{k\geq 0} h_kt^k\right)^\alpha$.
In particular,
\begin{equation}
    \Delta_0(h_n) =
    \sum_{k=0}^n h_k \otimes h_{n-k}.
\end{equation}
The algebra $\FdB$ with the coproduct $\Delta_0$ is the classical
{\em Hopf algebra of symmetric functions} $\Sym$ \cite{Mcd95}. Moreover,
for all $\gamma \neq 0$, all $\FdB_\gamma$ are isomorphic to $\FdB_1$,
which is known as the {\em Faà di Bruno algebra} \cite{JR79}. The coproduct
$\Delta_0$ comes from the interpretation of $\FdB$ as the algebra of
polynomial functions on the multiplicative group
$(G(t) := \{1+\sum_{k \geq 1}a_k t^k\},\cdot)$ of formal power series
of constant term $1$, and $\Delta_1$ comes from its interpretation as
the algebra of polynomial functions on the group $(tG(t),\circ)$ of
formal diffeomorphisms of the real line.
\medskip

\subsubsection{Noncommutative analogs}
Formal power series in one variable with coefficients in a noncommutative
algebra can be composed (by substitution of the variable). This operation
is not associative, so that they do not form a group. For example, when
$a$ and $b$ belong to a noncommutative algebra, one has
\begin{equation}
    (t^2\circ at)\circ bt = a^2t^2 \circ bt = a^2b^2t^2
\end{equation}
but
\begin{equation}
    t^2\circ (at\circ bt) = t^2\circ abt = ababt^2.
\end{equation}
However, the analogue of the Fa\`a di Bruno algebra still exists in this
context and is known as the \emph{noncommutative Fa\`a di Bruno algebra}.
It is investigated in \cite{BFK06} in view of applications in quantum
field theory. In \cite{Foi08}, Foissy also obtains an analogue of the
family $\FdB_\gamma$ in this context. Indeed, considering noncommutative
generators $\Sbf_n$ (with $\deg(\Sbf_n) = n$) instead of the $h_n$, for
all $n \geq 1$, leads to a free noncommutative algebra $\FdBNC$ whose
bases are indexed by integer compositions. This is the \emph{algebra of
noncommutative symmetric functions} \cite{GKLLRT95}. The addition of the
coproduct $\Delta_\gamma$ defined by
\begin{equation}\label{equ:coproduit_delta_gamma_non_commutatif}
\Delta_\gamma(\Sbf_n) :=
    \sum_{k=0}^n \Sbf_k \otimes \Sbf_{n-k}((k\gamma +1)A),
\end{equation}
where, for any $\alpha\in\R$ and $n\in\N$, $\Sbf_n(\alpha A)$
is the coefficient of $t^n$ in
$\left(\sum_{k\geq 0} \Sbf_kt^k\right)^\alpha$, forms a noncommutative
Hopf algebra $\FdBNC_\gamma$. In particular,
\begin{equation}
     \Delta_0(\Sbf_n)=\sum_{k=0}^n \Sbf_k\otimes\Sbf_{n-k},
\end{equation}
where $\Sbf_0$ is the unit. In this way, $\FdBNC$ with the coproduct
$\Delta_0$ is the {\em Hopf algebra of noncommutative symmetric functions}
$\SymNC$ \cite{GKLLRT95,KLT97}, and for all $\gamma \neq 0$, all the
$\FdBNC_\gamma$ are isomorphic to $\FdBNC_1$, which is the
{\em noncommutative Faà di Bruno algebra}.
\medskip

\subsection{The natural Hopf algebra of an operad}
We shall consider in this work only nonsymmetric operads in the category
of sets. For this reason, we shall call these simply {\em operads}.
\medskip

\subsubsection{Operads}
In our context, an operad is a triple $(\Oca, \circ_i, \Unite)$ where
\begin{equation}
    \Oca := \bigsqcup_{n \geq 1} \Oca(n)
\end{equation}
is a graded set,
\begin{equation}
    \circ_i : \Oca(n) \times \Oca(m) \to \Oca(n + m - 1),
    \qquad n, m \geq 1, i \in [n],
\end{equation}
is a composition map, called {\em partial composition}, and
$\Unite \in \Oca(1)$ is a unit. These data have to satisfy the relations
\begin{equation} \label{eq:AssocSerie}
    (x \circ_i y) \circ_{i + j - 1} z = x \circ_i (y \circ_j z),
    \qquad x \in \Oca(n), y \in \Oca(m),
    z \in \Oca(k), i \in [n], j \in [m],
\end{equation}
\begin{equation} \label{eq:AssocParallele}
    (x \circ_i y) \circ_{j + m - 1} z = (x \circ_j z) \circ_i y,
    \qquad x \in \Oca(n), y \in \Oca(m),
    z \in \Oca(k), 1 \leq i < j \leq n,
\end{equation}
\begin{equation} \label{eq:Unite}
    \Unite \circ_1 x = x = x \circ_i \Unite,
    \qquad x \in \Oca(n), i \in [n].
\end{equation}
\medskip

Besides, we shall denote by
\begin{equation}
    \circ : \Oca(n) \times
    \Oca(m_1) \times \dots \times \Oca(m_n) \to
    \Oca(m_1 + \dots + m_n),
    \qquad n, m_1, \dots, m_n \geq 1,
\end{equation}
the {\em total composition map} of $\Oca$. It is defined for any
$x \in \Oca(n)$ and $y_1, \dots, y_n \in \Oca$ by
\begin{equation}
    x \circ [y_1, \dots, y_n] :=
    (\dots ((x \circ_n y_n) \circ_{n - 1} y_{n - 1}) \dots) \circ_1 y_1.
\end{equation}
\medskip

If $x$ is an element of $\Oca(n)$, we say that the {\em arity} $|x|$ of
$x$ is $n$. An {\em operad morphism} is a map $\phi : \Oca_1 \to \Oca_2$
between two operads $\Oca_1$ and $\Oca_2$ such that $\phi$ commutes
with the partial composition maps and preserves the arities. A subset $S$
of an operad $\Oca$ is a {\em suboperad} of $\Oca$ if $\Unite \in S$ and
the composition of $\Oca$ is stable in $S$. The
{\em suboperad of $\Oca$ generated by} a subset $G$ of $\Oca$ is the
smallest suboperad of $\Oca$ containing $G$.
\medskip

\subsubsection{The natural bialgebra of an operad}
Let us recall a very simple construction associating a Hopf algebra with
an operad. A slightly different version of this construction is considered
in \cite{Vdl04,CL07,ML13}. Let $\Oca$ be an operad and denote by
$\Oca^+$ the set $\Oca \setminus \{\Unite\}$. The {\em natural bialgebra}
of $\Oca$ is the free commutative algebra  $\OvH(\Oca)$ spanned by the
$\Tit_x$, where the $x$ are elements of $\Oca^+$. The bases of
$\OvH(\Oca)$ are thus indexed by finite multisets of elements of
$\Oca^+$. The unit of $\OvH(\Oca)$ is denoted by $\Tit_{\Unite}$ and
the coproduct of $\OvH(\Oca)$ is the unique algebra morphism satisfying,
for any element $x$ of $\Oca^+$,
\begin{equation}
    \Delta(\Tit_x) =
    \sum_{\substack{y, z_1, \dots, z_\ell \in \Oca \\
    y \circ [z_1, \dots, z_\ell] = x}}
    \Tit_y \otimes \Tit_{z_1} \dots \Tit_{z_\ell}.
\end{equation}
\medskip

The bialgebra $\OvH(\Oca)$ can be graded by $\deg(\Tit_x) := |x| - 1$.
Note that with this grading, when $\Oca(1) = \{\Unite\}$ and the $\Oca(n)$
are finite for all $n \geq 1$, $\OvH(\Oca)$ becomes a combinatorial
Hopf algebra.
\medskip

\subsection{PROs and free PROs}
We recall here the definitions of PROs and free PROs in terms of prographs
and introduce the notions of reduced and indecomposable elements, which
will be used in the following sections.
\medskip

\subsubsection{PROs}
A {\em PRO} is a quadruple $(\Pca, *, \circ, \Unite_p)$ where $\Pca$ is
a bigraded set of the form
\begin{equation}
    \Pca := \bigsqcup_{p \geq 0} \bigsqcup_{q \geq 0} \Pca(p, q),
\end{equation}
such that for any $p, q \geq 0$, $\Pca(p, q)$ contains elements $x$ with
$\In(x) := p$ as {\em input arity} and $\Out(x) := q$ as {\em output arity},
$*$ is a map of the form
\begin{equation}
    * : \Pca(p, q) \times \Pca(p', q') \to \Pca(p + p', q + q'),
    \qquad p, p', q, q' \geq 0,
\end{equation}
called {\em horizontal composition}, $\circ$ is a map of the form
\begin{equation}
    \circ : \Pca(q, r) \times \Pca(p, q) \to \Pca(p, r),
    \qquad p, q, r \geq 0,
\end{equation}
called {\em vertical composition}, and for any $p \geq 0$, $\Unite_p$ is
an element of $\Pca(p, p)$ called {\em unit of arity $p$}.
\medskip

These data have to satisfy for all $x, y, z \in \Pca$ the six relations
\begin{equation} \label{equ:assoc_compo_h}
    (x * y) * z = x * (y * z),
    \qquad x, y, z \in \Pca,
\end{equation}
\begin{equation} \label{equ:assoc_compo_v}
    (x \circ y) \circ z = x \circ (y \circ z),
    \qquad x, y, z \in \Pca, \In(x) = \Out(y), \In(y) = \Out(z),
\end{equation}
\begin{equation} \label{equ:compo_h_v}
    (x \circ y) * (x' \circ y') = (x * x') \circ (y * y'),
    \qquad x, x', y, y' \in \Pca, \In(x) = \Out(y), \In(x') = \Out(y'),
\end{equation}
\begin{equation} \label{equ:relation_unite_h}
    \Unite_p * \Unite_q = \Unite_{p + q},
    \qquad p, q \geq 0,
\end{equation}
\begin{equation} \label{equ:relation_unite_zero}
    x * \Unite_0 = x = \Unite_0 * x,
    \qquad x \in \Pca,
\end{equation}
\begin{equation} \label{equ:relation_unite_v}
    x \circ \Unite_p = x = \Unite_q \circ x,
    \qquad x \in \Pca, p, q \geq 0, \In(x) = p, \Out(x) = q.
\end{equation}
\medskip

A {\em PRO morphism} is a map $\phi : \Pca_1 \to \Pca_2$ between two
PROs $\Pca_1$ and $\Pca_2$ such that $\phi$ commutes with the
horizontal and vertical compositions and preserves the input and
output arities. A subset $S$ of a PRO $\Pca$ is a {\em sub-PRO} of $\Pca$
if $\Unite_p \in S$ for any $p \geq 0$ and the horizontal and vertical
compositions of $\Pca$ are stable in $S$. The {\em sub-PRO of $\Pca$ generated by}
a subset $G$ of $\Pca$ is the smallest sub-PRO of $\Pca$ containing $G$.
An equivalence relation $\equiv$ on $\Pca$ is a {\em congruence of PROs}
if all the elements of a same $\equiv$-equivalence class have the same
input arity and the same output arity, and $\equiv$ is compatible with the
horizontal and the vertical composition. Any congruence $\equiv$ of
$\Pca$ gives rise to a {\em PRO quotient} of $\Pca$ denoted by $\Pca/_\equiv$
and defined in the expected way.
\medskip

\subsubsection{Free PROs}
Let us now set our terminology about free PROs and its elements in terms
of {\em prographs}. From now,
\begin{equation} \label{equ:ensemble_bigradue}
    G := \bigsqcup_{p \geq 1} \bigsqcup_{q \geq 1} G(p, q)
\end{equation}
is a bigraded set. An {\em elementary prograph} $e$ over $G$ is a formal
operator labeled by an element $\La$ of $G(p, q)$. The {\em input}
(resp. {\em output}) {\em arity} of $e$ is $p$ (resp. $q$). We represent
$e$ as a rectangle labeled by $g$ with $p$ incoming edges (below the
rectangle) and $q$ outgoing edges (above the rectangle). For instance,
if $\La \in G(2, 3)$, the elementary prograph labeled by $\La$ is
depicted by
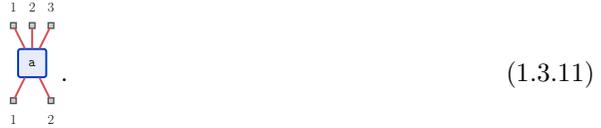
\begin{equation}
    \begin{split}
    \scalebox{.25}{\begin{tikzpicture}
        \node[Operateur](0)at(0,0){\begin{math}\La\end{math}};
        \node[Feuille](1)at(-1,-2){};
        \node[Feuille](2)at(1,-2){};
        \node[Feuille](3)at(-1,2){};
        \node[Feuille](33)at(0,2){};
        \node[Feuille](4)at(1,2){};
        \draw[Arete](0)--(1);
        \draw[Arete](0)--(2);
        \draw[Arete](0)--(3);
        \draw[Arete](0)--(33);
        \draw[Arete](0)--(4);
        \node[below of=1,font=\Huge]{\begin{math}1\end{math}};
        \node[below of=2,font=\Huge]{\begin{math}2\end{math}};
        \node[above of=3,font=\Huge]{\begin{math}1\end{math}};
        \node[above of=33,font=\Huge]{\begin{math}2\end{math}};
        \node[above of=4,font=\Huge]{\begin{math}3\end{math}};
    \end{tikzpicture}}
    \end{split}\,.
\end{equation}
\medskip

A {\em prograph} over $G$ is a formal operator defined recursively
as follows. A prograph over $G$ can be either an elementary prograph
over $G$, or a special element, the {\em wire} depicted by
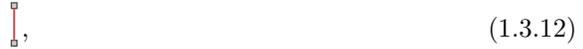
\begin{equation}
    \begin{split}
    \scalebox{.25}{\begin{tikzpicture}
        \node[Feuille](S1)at(0,0){};
        \node[Feuille](E1)at(0,-2){};
        \draw[Arete](E1)--(S1);
    \end{tikzpicture}}
    \end{split}\,,
\end{equation}
or a combination of two prographs over $G$ through the following two
operations. The first one, denoted by $*$, consists in placing two
prographs side by side. For instance, if $x$ is a prograph with $p$
inputs (resp. $q$ outputs) and $y$ is a prograph with $p'$ inputs
(resp. $q'$ outputs),
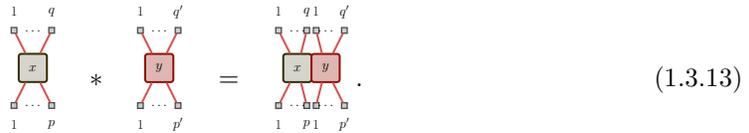
\begin{equation}
    \begin{split}
    \scalebox{.25}{\begin{tikzpicture}
        \node[Operateur,Marque3](0)at(0,0){\begin{math}x\end{math}};
        \node[Feuille](1)at(-1,-2){};
        \node[Feuille](2)at(1,-2){};
        \node[Feuille](3)at(-1,2){};
        \node[Feuille](4)at(1,2){};
        \draw[Arete](0)--(1);
        \draw[Arete](0)--(2);
        \draw[Arete](0)--(3);
        \draw[Arete](0)--(4);
        \node[below of=1,font=\Huge]{\begin{math}1\end{math}};
        \node[below of=2,font=\Huge]{\begin{math}p\end{math}};
        \node[above of=3,font=\Huge]{\begin{math}1\end{math}};
        \node[above of=4,font=\Huge]{\begin{math}q\end{math}};
        \node[below of=0,node distance=2cm,font=\Huge]
                {\begin{math}\dots\end{math}};
        \node[above of=0,node distance=2cm,font=\Huge]
                {\begin{math}\dots\end{math}};
    \end{tikzpicture}}
    \end{split}
    \quad * \quad
    \begin{split}
    \scalebox{.25}{\begin{tikzpicture}
        \node[Operateur,Marque4](0)at(0,0){\begin{math}y\end{math}};
        \node[Feuille](1)at(-1,-2){};
        \node[Feuille](2)at(1,-2){};
        \node[Feuille](3)at(-1,2){};
        \node[Feuille](4)at(1,2){};
        \draw[Arete](0)--(1);
        \draw[Arete](0)--(2);
        \draw[Arete](0)--(3);
        \draw[Arete](0)--(4);
        \node[below of=1,font=\Huge]{\begin{math}1\end{math}};
        \node[below of=2,font=\Huge]{\begin{math}p'\end{math}};
        \node[above of=3,font=\Huge]{\begin{math}1\end{math}};
        \node[above of=4,font=\Huge]{\begin{math}q'\end{math}};
        \node[below of=0,node distance=2cm,font=\Huge]
                {\begin{math}\dots\end{math}};
        \node[above of=0,node distance=2cm,font=\Huge]
                {\begin{math}\dots\end{math}};
    \end{tikzpicture}}
    \end{split}
    \quad = \quad
    \begin{split}
    \scalebox{.25}{\begin{tikzpicture}
        \node[Operateur,Marque3](0)at(0,0){\begin{math}x\end{math}};
        \node[Operateur,Marque4](00)at(1.5,0){\begin{math}y\end{math}};
        \node[Feuille](1)at(-1,-2){};
        \node[Feuille](2)at(.5,-2){};
        \node[Feuille](3)at(-1,2){};
        \node[Feuille](4)at(.5,2){};
        \draw[Arete](0)--(1);
        \draw[Arete](0)--(2);
        \draw[Arete](0)--(3);
        \draw[Arete](0)--(4);
        \node[below of=1,font=\Huge]{\begin{math}1\end{math}};
        \node[below of=2,font=\Huge]{\begin{math}p\end{math}};
        \node[above of=3,font=\Huge]{\begin{math}1\end{math}};
        \node[above of=4,font=\Huge]{\begin{math}q\end{math}};
        \node[below of=0,node distance=2cm,font=\Huge]
                {\begin{math}\dots\end{math}};
        \node[above of=0,node distance=2cm,font=\Huge]
                {\begin{math}\dots\end{math}};
        \node[Feuille](10)at(1,-2){};
        \node[Feuille](20)at(2.5,-2){};
        \node[Feuille](30)at(1,2){};
        \node[Feuille](40)at(2.5,2){};
        \draw[Arete](00)--(10);
        \draw[Arete](00)--(20);
        \draw[Arete](00)--(30);
        \draw[Arete](00)--(40);
        \node[below of=10,font=\Huge]{\begin{math}1\end{math}};
        \node[below of=20,font=\Huge]{\begin{math}p'\end{math}};
        \node[above of=30,font=\Huge]{\begin{math}1\end{math}};
        \node[above of=40,font=\Huge]{\begin{math}q'\end{math}};
        \node[below of=00,node distance=2cm,font=\Huge]
                {\begin{math}\dots\end{math}};
        \node[above of=00,node distance=2cm,font=\Huge]
                {\begin{math}\dots\end{math}};
    \end{tikzpicture}}
    \end{split}\,.
\end{equation}
The second one, denoted by $\circ$, consists in connecting the
inputs of a first prograph over the outputs of a second.
For instance, if $x$ is a prograph with $p$ inputs (resp. $q$ outputs)
and $y$ is a prograph with $r$ inputs (resp. $p$ outputs),
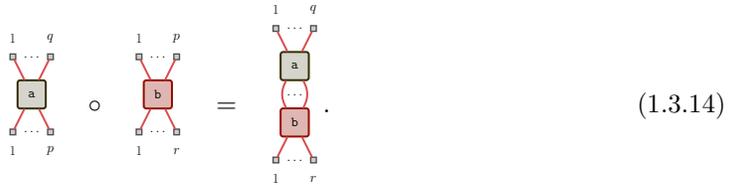
\begin{equation}
    \begin{split}
    \scalebox{.25}{\begin{tikzpicture}[yscale=1]
        \node[Operateur,Marque3](0)at(0,0){\begin{math}\La\end{math}};
        \node[Feuille](1)at(-1,-2){};
        \node[Feuille](2)at(1,-2){};
        \node[Feuille](3)at(-1,2){};
        \node[Feuille](4)at(1,2){};
        \draw[Arete](0)--(1);
        \draw[Arete](0)--(2);
        \draw[Arete](0)--(3);
        \draw[Arete](0)--(4);
        \node[below of=1,font=\Huge]{\begin{math}1\end{math}};
        \node[below of=2,font=\Huge]{\begin{math}p\end{math}};
        \node[above of=3,font=\Huge]{\begin{math}1\end{math}};
        \node[above of=4,font=\Huge]{\begin{math}q\end{math}};
        \node[below of=0,node distance=2cm,font=\Huge]
                {\begin{math}\dots\end{math}};
        \node[above of=0,node distance=2cm,font=\Huge]
                {\begin{math}\dots\end{math}};
    \end{tikzpicture}}
    \end{split}
    \quad \circ \quad
    \begin{split}
    \scalebox{.25}{\begin{tikzpicture}[yscale=1]
        \node[Operateur,Marque4](0)at(0,0){\begin{math}\Lb\end{math}};
        \node[Feuille](1)at(-1,-2){};
        \node[Feuille](2)at(1,-2){};
        \node[Feuille](3)at(-1,2){};
        \node[Feuille](4)at(1,2){};
        \draw[Arete](0)--(1);
        \draw[Arete](0)--(2);
        \draw[Arete](0)--(3);
        \draw[Arete](0)--(4);
        \node[below of=1,font=\Huge]{\begin{math}1\end{math}};
        \node[below of=2,font=\Huge]{\begin{math}r\end{math}};
        \node[above of=3,font=\Huge]{\begin{math}1\end{math}};
        \node[above of=4,font=\Huge]{\begin{math}p\end{math}};
        \node[below of=0,node distance=2cm,font=\Huge]
                {\begin{math}\dots\end{math}};
        \node[above of=0,node distance=2cm,font=\Huge]
                {\begin{math}\dots\end{math}};
    \end{tikzpicture}}
    \end{split}
    \quad = \quad
    \begin{split}
    \scalebox{.25}{\begin{tikzpicture}
        \node[Operateur,Marque3](0)at(0,0){\begin{math}\La\end{math}};
        \node[Operateur,Marque4](00)at(0,-3){\begin{math}\Lb\end{math}};
        \node[Feuille](10)at(-1,-5){};
        \node[Feuille](20)at(1,-5){};
        \node[Feuille](3)at(-1,2){};
        \node[Feuille](4)at(1,2){};
        \draw[Arete](00)--(10);
        \draw[Arete](00)--(20);
        \draw[Arete](0)--(3);
        \draw[Arete](0)--(4);
        \node[below of=10,font=\Huge]{\begin{math}1\end{math}};
        \node[below of=20,font=\Huge]{\begin{math}r\end{math}};
        \node[above of=3,font=\Huge]{\begin{math}1\end{math}};
        \node[above of=4,font=\Huge]{\begin{math}q\end{math}};
        \node[below of=00,node distance=2cm,font=\Huge]
                {\begin{math}\dots\end{math}};
        \node[above of=0,node distance=2cm,font=\Huge]
                {\begin{math}\dots\end{math}};
        \draw[Arete](0)edge[bend right] node[]{}(00);
        \draw[Arete](0)edge[bend left] node[]{}(00);
        \node[below of=0, node distance=1.5cm,font=\Huge]
                {\begin{math}\dots\end{math}};
    \end{tikzpicture}}
    \end{split}\,.
\end{equation}
By definition, connecting the input (resp. output) of a wire to the
output (resp. input) of a prograph $x$ does not change $x$.
\medskip

The {\em input} (resp. {\em output}) {\em arity} of a prograph $x$
is its number of inputs $\In(x)$ (resp. outputs $\Out(x)$). The inputs
(resp. outputs) of a prograph are numbered from left to right from $1$
to $\In(x)$ (resp. $\Out(x)$), possibly implicitly in the drawings.
The {\em degree} $\deg(x)$ of a prograph $x$ is the number
of elementary prographs required to build it. For instance,
\begin{equation}
    \begin{split}\scalebox{.25}{\begin{tikzpicture}[yscale=.65]
        \node[Feuille](S1)at(0,0){};
        \node[Feuille](S2)at(2,0){};
        \node[Feuille](S3)at(4.5,0){};
        \node[Feuille](S4)at(8,0){};
        \node[Feuille](S5)at(10,0){};
        \node[Operateur](N1)at(1,-3.5){\begin{math}\La\end{math}};
        \node[Operateur,Marque1](N2)at(4.5,-2){\begin{math}\Lb\end{math}};
        \node[Operateur](N3)at(7,-5){\begin{math}\La\end{math}};
        \node[Feuille](E1)at(0,-7){};
        \node[Feuille](E2)at(2,-7){};
        \node[Feuille](E3)at(3,-7){};
        \node[Feuille](E4)at(4.5,-7){};
        \node[Feuille](E5)at(6,-7){};
        \node[Feuille](E6)at(8,-7){};
        \node[Feuille](E7)at(10,-7){};
        \draw[Arete](N1)--(S1);
        \draw[Arete](N1)--(S2);
        \draw[Arete](N1)--(E1);
        \draw[Arete](N1)--(E2);
        \draw[Arete](N2)--(S3);
        \draw[Arete](N2)--(E3);
        \draw[Arete](N2)--(E4);
        \draw[Arete](N2)--(N3);
        \draw[Arete](N3)--(S4);
        \draw[Arete](N3)--(E5);
        \draw[Arete](N3)--(E6);
        \draw[Arete](S5)--(E7);
    \end{tikzpicture}}\end{split}
\end{equation}
is a prograph over $G := G(2, 2) \sqcup G(3, 1)$ where $G(2, 2) := \{\La\}$
and $G(3, 1) := \{\Lb\}$. Its input arity is $7$, its output arity is
$5$, and its degree is $3$.
\medskip

The {\em free PRO generated by $G$} is the PRO $\Free(G)$ whose elements
are all the prographs on~$G$, the horizontal composition being the operation
$*$ on prographs, and the vertical composition being the operation $\circ$.
Its unit $\Unite_1$ is the wire, and for any $p \geq 0$, $\Unite_p$ is
the horizontal composition of $p$ occurrences of the wire. Notice
that by~\eqref{equ:ensemble_bigradue}, there is no elementary prograph in
$\Free(G)$ with a null input or output arity. Therefore, $\Unite_0$ is
the only element of $\Free(G)$ with a null input (resp. output) arity.
In this work, we consider only free PROs satisfying this property.
\medskip

\begin{Lemme} \label{lem:regle_du_carre}
    Let $\Pca$ be a free PRO and $x, y, z, t \in \Pca$ such that
    $x * y = z \circ t$. Then, there exist four unique elements
    $x_1$, $x_2$, $y_1$, $y_2$ of $\Pca$ such that $x = x_1 \circ x_2$,
    $y = y_1 \circ y_2$, $z = x_1 * y_1$, and $t = x_2 * y_2$.
\end{Lemme}
\begin{proof}
    Let us prove the uniqueness. Assume that there are
    $x_1, x_2, y_1, y_2 \in \Pca$ and $x'_1, x'_2, y'_1, y'_2 \in \Pca$
    such that $x = x_1 \circ x_2 = x'_1 \circ x'_2$,
    $y = y_1 \circ y_2  = y'_1 \circ y'_2$, $z = x_1 * y_1 = x'_1 * y'_1$,
    and $t = x_2 * y_2 = x'_2 * y'_2$. Then, we have in particular
    $\Out(x_1) = \Out(x'_1)$ and $\Out(y_1) = \Out(y'_1)$. This, together
    with the relation $x_1 * y_1 = x'_1 * y'_1$ and the fact that $\Pca$
    is free, implies $x_1 = x'_1$ and $y_1 = y'_1$. In the same way,
    we obtain $x_2 = x'_2$ and $y_2 = y'_2$.
    \smallskip

    Let us now give a geometrical proof for the existence based upon the
    fact that prographs are planar objects. Since
    $u := x * y = z \circ t$, $u$ is a prograph obtained by an horizontal
    composition of two prographs. Then, $u = x * y$ depicted in a plane
    $\mathfrak{P}$ can be split into two regions $\mathfrak{X}$ and
    $\mathfrak{Y}$ such that $\mathfrak{X}$ contains the prograph $x$,
    $\mathfrak{Y}$ contains the prograph $y$, and $y$ is at the right of
    $x$. On the other hand, $u = z \circ t$ depicted in the same plane
    $\mathfrak{P}$ can be split into two regions $\mathfrak{Z}$ and
    $\mathfrak{T}$ such that $\mathfrak{Z}$ contains the prograph $z$,
    $\mathfrak{T}$ contains the prograph $t$, and $t$ is below $z$, the
    inputs of $z$ being connected to the outputs of $t$. We then obtain
    a division of $\mathfrak{P}$ into four regions
    $\mathfrak{X}\cap \mathfrak{Z}$, $\mathfrak{X}\cap \mathfrak{T}$,
    $\mathfrak{Y}\cap \mathfrak{Z}$, and $\mathfrak{Y}\cap \mathfrak{T}$,
    respectively containing prographs $x_1$, $x_2$, $y_1$, and $y_2$,
    and such that $x_1 * y_1 = z$, $x_2 * y_2 = t$, $x_1 \circ x_2 = x$,
    and $y_1 \circ y_2 = y$.
\end{proof}
\medskip

\subsubsection{Reduced and indecomposable elements}
Let $\Pca$ be a free PRO. Since $\Pca$ is free, any element $x$ of $\Pca$
can be uniquely written as $x = x_1 * \dots * x_\ell$
where the $x_i$ are elements of $\Pca$ different from $\Unite_0$, and
$\ell \geq 0$ is maximal. We call the word $\Dec(x) := (x_1, \dots, x_\ell)$
the {\em maximal decomposition} of $x$ and the $x_i$ the {\em factors}
of $x$. Notice that the maximal decomposition of $\Unite_0$ is the empty
word. We have, for instance,
\begin{equation}
    \begin{split} \Dec \end{split} \left(\;
    \begin{split}\scalebox{.25}{\begin{tikzpicture}[yscale=.65]
        \node[Feuille](S1)at(0,0){};
        \node[Feuille](S2)at(2,0){};
        \node[Feuille](S3)at(4.5,0){};
        \node[Feuille](S4)at(8,0){};
        \node[Feuille](S5)at(10,0){};
        \node[Operateur](N1)at(1,-3.5){\begin{math}\La\end{math}};
        \node[Operateur,Marque1](N2)at(4.5,-2){\begin{math}\Lb\end{math}};
        \node[Operateur](N3)at(7,-5){\begin{math}\La\end{math}};
        \node[Feuille](E1)at(0,-7){};
        \node[Feuille](E2)at(2,-7){};
        \node[Feuille](E3)at(3,-7){};
        \node[Feuille](E4)at(4.5,-7){};
        \node[Feuille](E5)at(6,-7){};
        \node[Feuille](E6)at(8,-7){};
        \node[Feuille](E7)at(10,-7){};
        \draw[Arete](N1)--(S1);
        \draw[Arete](N1)--(S2);
        \draw[Arete](N1)--(E1);
        \draw[Arete](N1)--(E2);
        \draw[Arete](N2)--(S3);
        \draw[Arete](N2)--(E3);
        \draw[Arete](N2)--(E4);
        \draw[Arete](N2)--(N3);
        \draw[Arete](N3)--(S4);
        \draw[Arete](N3)--(E5);
        \draw[Arete](N3)--(E6);
        \draw[Arete](S5)--(E7);
    \end{tikzpicture}}\end{split}\;\right)
    \begin{split} \quad = \quad \end{split}
    \left(\;
    \begin{split}\scalebox{.25}{\begin{tikzpicture}[yscale=.65]
        \node[Feuille](S1)at(0,0){};
        \node[Feuille](S2)at(2,0){};
        \node[Operateur](N1)at(1,-3.5){\begin{math}\La\end{math}};
        \node[Feuille](E1)at(0,-7){};
        \node[Feuille](E2)at(2,-7){};
        \draw[Arete](N1)--(S1);
        \draw[Arete](N1)--(S2);
        \draw[Arete](N1)--(E1);
        \draw[Arete](N1)--(E2);
    \end{tikzpicture}}\end{split}
    \begin{split}, \quad \end{split}
    \begin{split}\scalebox{.25}{\begin{tikzpicture}[yscale=.65]
        \node[Feuille](S3)at(4.5,0){};
        \node[Feuille](S4)at(8,0){};
        \node[Operateur,Marque1](N2)at(4.5,-2){\begin{math}\Lb\end{math}};
        \node[Operateur](N3)at(7,-5){\begin{math}\La\end{math}};
        \node[Feuille](E3)at(3,-7){};
        \node[Feuille](E4)at(4.5,-7){};
        \node[Feuille](E5)at(6,-7){};
        \node[Feuille](E6)at(8,-7){};
        \draw[Arete](N2)--(S3);
        \draw[Arete](N2)--(E3);
        \draw[Arete](N2)--(E4);
        \draw[Arete](N2)--(N3);
        \draw[Arete](N3)--(S4);
        \draw[Arete](N3)--(E5);
        \draw[Arete](N3)--(E6);
    \end{tikzpicture}}\end{split}
    \begin{split}, \quad \end{split}
    \begin{split}\scalebox{.25}{\begin{tikzpicture}[yscale=.65]
        \node[Feuille](S5)at(10,0){};
        \node[Feuille](E7)at(10,-7){};
        \draw[Arete](S5)--(E7);
    \end{tikzpicture}}\end{split}
    \;\right).
\end{equation}
\medskip

An element $x$ of $\Pca$ is {\em reduced} if all its factors are
different from $\Unite_1$. For any element $x$ of $\Pca$, we denote by
$\Reduit(x)$ the reduced element of $\Pca$ admitting as maximal decomposition
the longest subword of $\Dec(x)$ consisting in factors different from
$\Unite_1$. We have, for instance,
\begin{equation}
    \begin{split} \Reduit \end{split} \left(\;
    \begin{split}\scalebox{.25}{\begin{tikzpicture}[yscale=.65]
        \node[Feuille](S0)at(-1,0){};
        \node[Feuille](S1)at(0,0){};
        \node[Feuille](S2)at(2,0){};
        \node[Feuille](S22)at(3,0){};
        \node[Feuille](S222)at(4,0){};
        \node[Feuille](S3)at(6.5,0){};
        \node[Feuille](S4)at(10,0){};
        \node[Feuille](S5)at(12,0){};
        \node[Operateur](N1)at(1,-3.5){\begin{math}\La\end{math}};
        \node[Operateur,Marque1](N2)at(6.5,-2){\begin{math}\Lb\end{math}};
        \node[Operateur](N3)at(9,-5){\begin{math}\La\end{math}};
        \node[Feuille](E0)at(-1,-7){};
        \node[Feuille](E1)at(0,-7){};
        \node[Feuille](E2)at(2,-7){};
        \node[Feuille](E22)at(3,-7){};
        \node[Feuille](E222)at(4,-7){};
        \node[Feuille](E3)at(5,-7){};
        \node[Feuille](E4)at(6.5,-7){};
        \node[Feuille](E5)at(8,-7){};
        \node[Feuille](E6)at(10,-7){};
        \node[Feuille](E7)at(12,-7){};
        \draw[Arete](S0)--(E0);
        \draw[Arete](N1)--(S1);
        \draw[Arete](N1)--(S2);
        \draw[Arete](N1)--(E1);
        \draw[Arete](N1)--(E2);
        \draw[Arete](N2)--(S3);
        \draw[Arete](N2)--(E3);
        \draw[Arete](N2)--(E4);
        \draw[Arete](N2)--(N3);
        \draw[Arete](N3)--(S4);
        \draw[Arete](N3)--(E5);
        \draw[Arete](N3)--(E6);
        \draw[Arete](S5)--(E7);
        \draw[Arete](S22)--(E22);
        \draw[Arete](S222)--(E222);
    \end{tikzpicture}}\end{split}\;\right)
    \begin{split} \quad = \quad \end{split}
    \begin{split}\scalebox{.25}{\begin{tikzpicture}[yscale=.65]
        \node[Feuille](S1)at(0,0){};
        \node[Feuille](S2)at(2,0){};
        \node[Feuille](S3)at(4.5,0){};
        \node[Feuille](S4)at(8,0){};
        \node[Operateur](N1)at(1,-3.5){\begin{math}\La\end{math}};
        \node[Operateur,Marque1](N2)at(4.5,-2){\begin{math}\Lb\end{math}};
        \node[Operateur](N3)at(7,-5){\begin{math}\La\end{math}};
        \node[Feuille](E1)at(0,-7){};
        \node[Feuille](E2)at(2,-7){};
        \node[Feuille](E3)at(3,-7){};
        \node[Feuille](E4)at(4.5,-7){};
        \node[Feuille](E5)at(6,-7){};
        \node[Feuille](E6)at(8,-7){};
        \draw[Arete](N1)--(S1);
        \draw[Arete](N1)--(S2);
        \draw[Arete](N1)--(E1);
        \draw[Arete](N1)--(E2);
        \draw[Arete](N2)--(S3);
        \draw[Arete](N2)--(E3);
        \draw[Arete](N2)--(E4);
        \draw[Arete](N2)--(N3);
        \draw[Arete](N3)--(S4);
        \draw[Arete](N3)--(E5);
        \draw[Arete](N3)--(E6);
    \end{tikzpicture}}\end{split}\,.
\end{equation}
By extension, we denote by $\Reduit(\Pca)$ the set of the reduced elements
of $\Pca$. Note that $\Unite_0$ belongs to~$\Reduit(\Pca)$.
\medskip

Besides, we say that an element $x$ of $\Pca$ is {\em indecomposable}
if its maximal decomposition consists in exactly one factor. Note that
$\Unite_0$ is not indecomposable while $\Unite_1$ is.
\medskip

\begin{Lemme} \label{lem:relation_element_et_son_reduit}
    Let $\Pca$ be a free PRO and $x, y \in \Pca$, such that $x = \Reduit(y)$.
    Then, by denoting by $(x_1, \dots, x_\ell)$ the maximal decomposition
    of $x$, there exists a unique sequence of nonnegative integers
    $p_1, \dots, p_\ell, p_{\ell + 1}$ such that
    \begin{equation}
        y = \Unite_{p_1} * x_1 * \Unite_{p_2} * x_2 *
            \dots * x_\ell * \Unite_{p_{\ell + 1}}.
    \end{equation}
\end{Lemme}
\begin{proof}
    The existence comes from the fact that, since $x = \Reduit(y)$, the
    maximal decomposition of $x$ is obtained from the one of $y$ by
    suppressing the factors equal to the wire. The uniqueness comes from
    the fact that $\Pca$ is a free monoid for the horizontal composition $*$.
\end{proof}
\medskip

\section{From PROs to combinatorial Hopf algebras}
\label{sec:construction_H}
We introduce in this section the main construction of this work and review
some of its properties. In all this section, $\Pca$ is a free PRO
generated by a bigraded set $G$. Starting with $\Pca$, our
construction produces a bialgebra $\PvH(\Pca)$ whose bases are indexed
by the reduced elements of $\Pca$. We shall also extend this construction
over a class of non necessarily free PROs.
\medskip

\subsection{The Hopf algebra of a free PRO}
\label{subsec:PRO_libre_vers_AHC}
The bases of the vector space
\begin{equation}
    \PvH(\Pca) := \Vect(\Reduit(\Pca))
\end{equation}
are indexed by the reduced elements of $\Pca$. The elements $\Sbf_x$,
$x \in \Reduit(\Pca)$, form thus a basis of $\PvH(\Pca)$, called
{\em fundamental basis}. We endow $\PvH(\Pca)$ with a product
$\cdot : \PvH(\Pca) \otimes \PvH(\Pca) \to \PvH(\Pca)$ linearly defined,
for any reduced elements $x$ and $y$ of $\Pca$, by
\begin{equation}
    \Sbf_x \cdot \Sbf_y := \Sbf_{x * y},
\end{equation}
and with a coproduct
$\Delta : \PvH(\Pca) \to \PvH(\Pca) \otimes \PvH(\Pca)$ linearly defined,
for any reduced elements $x$ of $\Pca$, by
\begin{equation}
    \Delta\left(\Sbf_x\right) :=
    \sum_{\substack{y, z \in \Pca \\ y \circ z = x}}
    \Sbf_{\Reduit(y)} \otimes \Sbf_{\Reduit(z)}.
\end{equation}
\medskip

Throughout this section, we shall consider some examples involving
the free PRO generated by $G := G(2, 2) \sqcup G(3, 1)$ where
$G(2, 2) := \{\La\}$ and $G(3, 1) := \{\Lb\}$, denoted by $\AB$.
For instance, we have in $\PvH(\AB)$
\begin{equation}
    \Sbf_{
    \begin{split}\scalebox{.25}{\begin{tikzpicture}[yscale=.7]
        \node[Feuille](S1)at(0,0){};
        \node[Feuille](S2)at(2,0){};
        \node[Feuille](S3)at(4,0){};
        \node[Operateur](N1)at(1,-3.5){\begin{math}\La\end{math}};
        \node[Operateur,Marque1](N2)at(4,-2){\begin{math}\Lb\end{math}};
        \node[Operateur](N3)at(5,-5){\begin{math}\La\end{math}};
        \node[Feuille](E1)at(0,-7){};
        \node[Feuille](E2)at(2,-7){};
        \node[Feuille](E3)at(3,-7){};
        \node[Feuille](E4)at(4,-7){};
        \node[Feuille](E5)at(6,-7){};
        \draw[Arete](N1)--(S1);
        \draw[Arete](N1)--(S2);
        \draw[Arete](N1)--(E1);
        \draw[Arete](N1)--(E2);
        \draw[Arete](N2)--(S3);
        \draw[Arete](N2)--(E3);
        \draw[Arete](N3)--(E4);
        \draw[Arete](N3)--(E5);
        \draw[Arete](N2)edge[bend right=20] node[]{}(N3);
        \draw[Arete](N2)edge[bend left=20] node[]{}(N3);
    \end{tikzpicture}}\end{split}}
    \cdot
    \Sbf_{
    \begin{split}\scalebox{.25}{\begin{tikzpicture}[yscale=.7]
        \node[Feuille](S1)at(0,0){};
        \node[Feuille](S2)at(2,0){};
        \node[Feuille](S3)at(3,0){};
        \node[Feuille](S4)at(5,0){};
        \node[Operateur](N1)at(1,-2){\begin{math}\La\end{math}};
        \node[Operateur](N2)at(4,-2){\begin{math}\La\end{math}};
        \node[Feuille](E1)at(0,-4){};
        \node[Feuille](E2)at(2,-4){};
        \node[Feuille](E3)at(3,-4){};
        \node[Feuille](E4)at(5,-4){};
        \draw[Arete](N1)--(S1);
        \draw[Arete](N1)--(S2);
        \draw[Arete](N1)--(E1);
        \draw[Arete](N1)--(E2);
        \draw[Arete](N2)--(S3);
        \draw[Arete](N2)--(S4);
        \draw[Arete](N2)--(E3);
        \draw[Arete](N2)--(E4);
    \end{tikzpicture}}\end{split}}
    \enspace = \enspace
    \Sbf_{
    \begin{split}\scalebox{.25}{\begin{tikzpicture}[yscale=.7]
        \node[Feuille](S1)at(0,0){};
        \node[Feuille](S2)at(2,0){};
        \node[Feuille](S3)at(4,0){};
        \node[Feuille](S4)at(7,0){};
        \node[Feuille](S5)at(9,0){};
        \node[Feuille](S6)at(10,0){};
        \node[Feuille](S7)at(12,0){};
        \node[Operateur](N1)at(1,-3.5){\begin{math}\La\end{math}};
        \node[Operateur,Marque1](N2)at(4,-2){\begin{math}\Lb\end{math}};
        \node[Operateur](N3)at(5,-5){\begin{math}\La\end{math}};
        \node[Operateur](N4)at(8,-3.5){\begin{math}\La\end{math}};
        \node[Operateur](N5)at(11,-3.5){\begin{math}\La\end{math}};
        \node[Feuille](E1)at(0,-7){};
        \node[Feuille](E2)at(2,-7){};
        \node[Feuille](E3)at(3,-7){};
        \node[Feuille](E4)at(4,-7){};
        \node[Feuille](E5)at(6,-7){};
        \node[Feuille](E6)at(7,-7){};
        \node[Feuille](E7)at(9,-7){};
        \node[Feuille](E8)at(10,-7){};
        \node[Feuille](E9)at(12,-7){};
        \draw[Arete](N1)--(S1);
        \draw[Arete](N1)--(S2);
        \draw[Arete](N1)--(E1);
        \draw[Arete](N1)--(E2);
        \draw[Arete](N2)--(S3);
        \draw[Arete](N2)--(E3);
        \draw[Arete](N3)--(E4);
        \draw[Arete](N3)--(E5);
        \draw[Arete](N2)edge[bend right=20] node[]{}(N3);
        \draw[Arete](N2)edge[bend left=20] node[]{}(N3);
        \draw[Arete](N4)--(S4);
        \draw[Arete](N4)--(S5);
        \draw[Arete](N4)--(E6);
        \draw[Arete](N4)--(E7);
        \draw[Arete](N5)--(S6);
        \draw[Arete](N5)--(S7);
        \draw[Arete](N5)--(E8);
        \draw[Arete](N5)--(E9);
    \end{tikzpicture}}\end{split}}
\end{equation}
and
\begin{multline}
    \Delta \Sbf_{
    \begin{split}\scalebox{.25}{\begin{tikzpicture}[yscale=.7]
        \node[Feuille](S1)at(0,0){};
        \node[Feuille](S2)at(2,0){};
        \node[Feuille](S3)at(4,0){};
        \node[Operateur](N1)at(1,-3.5){\begin{math}\La\end{math}};
        \node[Operateur,Marque1](N2)at(4,-2){\begin{math}\Lb\end{math}};
        \node[Operateur](N3)at(5,-5){\begin{math}\La\end{math}};
        \node[Feuille](E1)at(0,-7){};
        \node[Feuille](E2)at(2,-7){};
        \node[Feuille](E3)at(3,-7){};
        \node[Feuille](E4)at(4,-7){};
        \node[Feuille](E5)at(6,-7){};
        \draw[Arete](N1)--(S1);
        \draw[Arete](N1)--(S2);
        \draw[Arete](N1)--(E1);
        \draw[Arete](N1)--(E2);
        \draw[Arete](N2)--(S3);
        \draw[Arete](N2)--(E3);
        \draw[Arete](N3)--(E4);
        \draw[Arete](N3)--(E5);
        \draw[Arete](N2)edge[bend right=20] node[]{}(N3);
        \draw[Arete](N2)edge[bend left=20] node[]{}(N3);
    \end{tikzpicture}}\end{split}}
    \enspace = \enspace
    \Sbf_{\Unite_0} \otimes
    \Sbf_{
    \begin{split}\scalebox{.25}{\begin{tikzpicture}[yscale=.7]
        \node[Feuille](S1)at(0,0){};
        \node[Feuille](S2)at(2,0){};
        \node[Feuille](S3)at(4,0){};
        \node[Operateur](N1)at(1,-3.5){\begin{math}\La\end{math}};
        \node[Operateur,Marque1](N2)at(4,-2){\begin{math}\Lb\end{math}};
        \node[Operateur](N3)at(5,-5){\begin{math}\La\end{math}};
        \node[Feuille](E1)at(0,-7){};
        \node[Feuille](E2)at(2,-7){};
        \node[Feuille](E3)at(3,-7){};
        \node[Feuille](E4)at(4,-7){};
        \node[Feuille](E5)at(6,-7){};
        \draw[Arete](N1)--(S1);
        \draw[Arete](N1)--(S2);
        \draw[Arete](N1)--(E1);
        \draw[Arete](N1)--(E2);
        \draw[Arete](N2)--(S3);
        \draw[Arete](N2)--(E3);
        \draw[Arete](N3)--(E4);
        \draw[Arete](N3)--(E5);
        \draw[Arete](N2)edge[bend right=20] node[]{}(N3);
        \draw[Arete](N2)edge[bend left=20] node[]{}(N3);
    \end{tikzpicture}}\end{split}}
    \enspace + \enspace
    \Sbf_{
    \begin{split}\scalebox{.25}{\begin{tikzpicture}[yscale=.7]
        \node[Feuille](S1)at(0,0){};
        \node[Feuille](S2)at(2,0){};
        \node[Operateur](N1)at(1,-2){\begin{math}\La\end{math}};
        \node[Feuille](E1)at(0,-4){};
        \node[Feuille](E2)at(2,-4){};
        \draw[Arete](N1)--(S1);
        \draw[Arete](N1)--(S2);
        \draw[Arete](N1)--(E1);
        \draw[Arete](N1)--(E2);
    \end{tikzpicture}}\end{split}}
    \otimes
    \Sbf_{
    \begin{split}\scalebox{.25}{\begin{tikzpicture}[yscale=.7]
        \node[Feuille](S3)at(4,0){};
        \node[Operateur,Marque1](N2)at(4,-2){\begin{math}\Lb\end{math}};
        \node[Operateur](N3)at(5,-5){\begin{math}\La\end{math}};
        \node[Feuille](E3)at(3,-7){};
        \node[Feuille](E4)at(4,-7){};
        \node[Feuille](E5)at(6,-7){};
        \draw[Arete](N2)--(S3);
        \draw[Arete](N2)--(E3);
        \draw[Arete](N3)--(E4);
        \draw[Arete](N3)--(E5);
        \draw[Arete](N2)edge[bend right=20] node[]{}(N3);
        \draw[Arete](N2)edge[bend left=20] node[]{}(N3);
    \end{tikzpicture}}\end{split}}
    \enspace + \enspace
    \Sbf_{
    \begin{split}\scalebox{.25}{\begin{tikzpicture}[yscale=.7]
        \node[Feuille](S1)at(1,0){};
        \node[Operateur,Marque1](N1)at(1,-2){\begin{math}\Lb\end{math}};
        \node[Feuille](E1)at(0,-4){};
        \node[Feuille](E2)at(1,-4){};
        \node[Feuille](E3)at(2,-4){};
        \draw[Arete](N1)--(S1);
        \draw[Arete](N1)--(E1);
        \draw[Arete](N1)--(E2);
        \draw[Arete](N1)--(E3);
    \end{tikzpicture}}\end{split}}
    \otimes
    \Sbf_{
    \begin{split}\scalebox{.25}{\begin{tikzpicture}[yscale=.7]
        \node[Feuille](S1)at(0,0){};
        \node[Feuille](S2)at(2,0){};
        \node[Feuille](S3)at(3,0){};
        \node[Feuille](S4)at(5,0){};
        \node[Operateur](N1)at(1,-2){\begin{math}\La\end{math}};
        \node[Operateur](N2)at(4,-2){\begin{math}\La\end{math}};
        \node[Feuille](E1)at(0,-4){};
        \node[Feuille](E2)at(2,-4){};
        \node[Feuille](E3)at(3,-4){};
        \node[Feuille](E4)at(5,-4){};
        \draw[Arete](N1)--(S1);
        \draw[Arete](N1)--(S2);
        \draw[Arete](N1)--(E1);
        \draw[Arete](N1)--(E2);
        \draw[Arete](N2)--(S3);
        \draw[Arete](N2)--(S4);
        \draw[Arete](N2)--(E3);
        \draw[Arete](N2)--(E4);
    \end{tikzpicture}}\end{split}} \\
    \enspace + \enspace
    \Sbf_{
    \begin{split}\scalebox{.25}{\begin{tikzpicture}[yscale=.7]
        \node[Feuille](S1)at(0,0){};
        \node[Feuille](S2)at(2,0){};
        \node[Feuille](S3)at(4,0){};
        \node[Operateur](N1)at(1,-2){\begin{math}\La\end{math}};
        \node[Operateur,Marque1](N2)at(4,-2){\begin{math}\Lb\end{math}};
        \node[Feuille](E1)at(0,-4){};
        \node[Feuille](E2)at(2,-4){};
        \node[Feuille](E3)at(3,-4){};
        \node[Feuille](E4)at(4,-4){};
        \node[Feuille](E5)at(5,-4){};
        \draw[Arete](N1)--(S1);
        \draw[Arete](N1)--(S2);
        \draw[Arete](N1)--(E1);
        \draw[Arete](N1)--(E2);
        \draw[Arete](N2)--(S3);
        \draw[Arete](N2)--(E3);
        \draw[Arete](N2)--(E4);
        \draw[Arete](N2)--(E5);
    \end{tikzpicture}}\end{split}}
    \otimes
    \Sbf_{
    \begin{split}\scalebox{.25}{\begin{tikzpicture}[yscale=.7]
        \node[Feuille](S1)at(0,0){};
        \node[Feuille](S2)at(2,0){};
        \node[Operateur](N1)at(1,-2){\begin{math}\La\end{math}};
        \node[Feuille](E1)at(0,-4){};
        \node[Feuille](E2)at(2,-4){};
        \draw[Arete](N1)--(S1);
        \draw[Arete](N1)--(S2);
        \draw[Arete](N1)--(E1);
        \draw[Arete](N1)--(E2);
    \end{tikzpicture}}\end{split}}
    \enspace + \enspace
    \Sbf_{
    \begin{split}\scalebox{.25}{\begin{tikzpicture}[yscale=.7]
        \node[Feuille](S3)at(4,0){};
        \node[Operateur,Marque1](N2)at(4,-2){\begin{math}\Lb\end{math}};
        \node[Operateur](N3)at(5,-5){\begin{math}\La\end{math}};
        \node[Feuille](E3)at(3,-7){};
        \node[Feuille](E4)at(4,-7){};
        \node[Feuille](E5)at(6,-7){};
        \draw[Arete](N2)--(S3);
        \draw[Arete](N2)--(E3);
        \draw[Arete](N3)--(E4);
        \draw[Arete](N3)--(E5);
        \draw[Arete](N2)edge[bend right=20] node[]{}(N3);
        \draw[Arete](N2)edge[bend left=20] node[]{}(N3);
    \end{tikzpicture}}\end{split}}
    \otimes
    \Sbf_{
    \begin{split}\scalebox{.25}{\begin{tikzpicture}[yscale=.7]
        \node[Feuille](S1)at(0,0){};
        \node[Feuille](S2)at(2,0){};
        \node[Operateur](N1)at(1,-2){\begin{math}\La\end{math}};
        \node[Feuille](E1)at(0,-4){};
        \node[Feuille](E2)at(2,-4){};
        \draw[Arete](N1)--(S1);
        \draw[Arete](N1)--(S2);
        \draw[Arete](N1)--(E1);
        \draw[Arete](N1)--(E2);
    \end{tikzpicture}}\end{split}}
    \enspace + \enspace
    \Sbf_{
    \begin{split}\scalebox{.25}{\begin{tikzpicture}[yscale=.7]
        \node[Feuille](S1)at(0,0){};
        \node[Feuille](S2)at(2,0){};
        \node[Feuille](S3)at(4,0){};
        \node[Operateur](N1)at(1,-3.5){\begin{math}\La\end{math}};
        \node[Operateur,Marque1](N2)at(4,-2){\begin{math}\Lb\end{math}};
        \node[Operateur](N3)at(5,-5){\begin{math}\La\end{math}};
        \node[Feuille](E1)at(0,-7){};
        \node[Feuille](E2)at(2,-7){};
        \node[Feuille](E3)at(3,-7){};
        \node[Feuille](E4)at(4,-7){};
        \node[Feuille](E5)at(6,-7){};
        \draw[Arete](N1)--(S1);
        \draw[Arete](N1)--(S2);
        \draw[Arete](N1)--(E1);
        \draw[Arete](N1)--(E2);
        \draw[Arete](N2)--(S3);
        \draw[Arete](N2)--(E3);
        \draw[Arete](N3)--(E4);
        \draw[Arete](N3)--(E5);
        \draw[Arete](N2)edge[bend right=20] node[]{}(N3);
        \draw[Arete](N2)edge[bend left=20] node[]{}(N3);
    \end{tikzpicture}}\end{split}}
    \otimes
    \Sbf_{\Unite_0}\,.
\end{multline}
\medskip

\begin{Lemme} \label{lem:PRO_vers_AHC_coassociativite}
    Let $\Pca$ be a free PRO. Then, the coproduct $\Delta$ of $\PvH(\Pca)$
    is coassociative.
\end{Lemme}
\begin{proof}
    The bases of the vector space $\Vect(\Pca)$ are indexed by the
    (non-necessarily reduced) elements of $\Pca$. Then, the elements
    $\Rit_x$, $x \in \Pca$, form a basis of $\Vect(\Pca)$. Let us
    consider the coproduct $\Delta'$ defined, for any $x \in \Pca$, by
    \begin{equation}
        \Delta'\left(\Rit_x\right) :=
        \sum_{\substack{y, z \in \Pca \\ y \circ z = x}}
        \Rit_y \otimes \Rit_z.
     \end{equation}
    The associativity of the vertical composition $\circ$ of $\Pca$
    (see \eqref{equ:assoc_compo_v}) implies that $\Delta'$ is
    coassociative and hence, that $\Vect(\Pca)$ together with $\Delta'$
    form a coalgebra.
    \smallskip

    Consider now the map $\phi : \Vect(\Pca) \to \PvH(\Pca)$ defined, for
    any $x \in \Pca$, by $\phi(\Rit_x) := \Sbf_{\Reduit(x)}$. Let us show
    that $\phi$ commutes with the coproducts $\Delta$ and $\Delta'$,
    that is, $(\phi \otimes \phi) \Delta' = \Delta \phi$.
    Let $x \in \Pca$. By Lemma \ref{lem:relation_element_et_son_reduit},
    by denoting by $(x_1, \dots, x_\ell)$ the maximal decomposition of
    $\Reduit(x)$, there is a unique way to write $x$ as
    $x = \Unite_{p_1} * x_1 * \dots * x_\ell * \Unite_{\ell + 1}$
    where the $p_i$ are some integers. Then, thanks to the associativity
    of $*$ (see \eqref{equ:assoc_compo_h}), by iteratively applying
    Lemma \ref{lem:regle_du_carre}, we have
    \begin{align}
        (\phi \otimes \phi) \Delta'(\Rit_x)
        & =
        \sum_{\substack{y, z \in \Pca \\ y \circ z = x}}
        \Sbf_{\Reduit(y)} \otimes \Sbf_{\Reduit(z)} \\
        & =
        \sum_{\substack{y_1, \dots, y_\ell \in \Pca \\
            z_1, \dots, z_\ell \in \Pca \\
            (\Unite_{p_1} * y_1 * \dots * y_\ell * \Unite_{\ell + 1}) \\
            \qquad \circ
            (\Unite_{p_1} * z_1 * \dots * z_\ell * \Unite_{\ell + 1}) = \\
            \qquad \Unite_{p_1} * x_1 * \dots * x_\ell * \Unite_{\ell + 1}}}
        \Sbf_{\Reduit(\Unite_{p_1} * y_1 * \dots *
            y_\ell * \Unite_{\ell + 1})} \otimes
            \Sbf_{\Reduit(\Unite_{p_1} * z_1 * \dots *
                z_\ell * \Unite_{\ell + 1})} \\
        & =
        \sum_{\substack{y_1, \dots, y_\ell \in \Pca \\
            z_1, \dots, z_\ell \in \Pca \\
            (y_1 * \dots * y_\ell) \circ (z_1 * \dots * z_\ell) =
            x_1 * \dots * x_\ell}}
        \Sbf_{\Reduit(y_1 * \dots * y_\ell)} \otimes
            \Sbf_{\Reduit(z_1 * \dots * z_\ell)} \\
        & =
        \sum_{\substack{y, z \in \Pca \\ y \circ z = \Reduit(x)}}
            \Sbf_{\Reduit(y)} \otimes \Sbf_{\Reduit(z)} \\
        & =
        \Delta\left(\phi\left(\Rit_x\right)\right).
    \end{align}
    \smallskip

    Now, the coassociativity of $\Delta$ comes from the fact that $\phi$
    is a surjective map commuting with $\Delta$ and $\Delta'$. In more
    details, if $x$ is a reduced element of $\Pca$ and $I$ is the
    identity map on $\PvH(\Pca)$, we have
    \begin{align}
        (\Delta \otimes I) \Delta\left(\Sbf_x\right)
        & =
        (\Delta \otimes I) \Delta(\phi\left(\Rit_x\right)) \\
        & =
        (\phi \otimes \phi \otimes \phi)(\Delta' \otimes I)
            \Delta'\left(\Rit_x\right) \\
        & =
        (\phi \otimes \phi \otimes \phi)(I \otimes \Delta')
            \Delta'\left(\Rit_x\right) \\
        & =
        (I \otimes \Delta) \Delta(\phi\left(\Rit_x\right)) \\
        & =
        (I \otimes \Delta) \Delta\left(\Sbf_x\right).
    \end{align}
\end{proof}
\medskip

\begin{Lemme} \label{lem:PRO_vers_AHC_compatibilite}
    Let $\Pca$ be a free PRO. Then, the coproduct $\Delta$ of $\PvH(\Pca)$
    is a morphism of algebras.
\end{Lemme}
\begin{proof}
    Let $x$ and $y$ be two reduced elements of $\Pca$. We have
    \begin{equation} \label{equ:PRO_vers_AHC_coproduit_morphisme_1}
        \Delta(\Sbf_x \cdot \Sbf_y) =
        \sum_{\substack{z, t \in \Pca \\ z \circ t = x * y}}
        \Sbf_{\Reduit(z)} \otimes \Sbf_{\Reduit(t)}
    \end{equation}
    and
    \begin{equation} \label{equ:PRO_vers_AHC_coproduit_morphisme_2}
        \Delta(\Sbf_x) \Delta(\Sbf_y) =
        \sum_{\substack{x_1, x_2, y_1, y_2 \in \Pca \\ x_1 \circ x_2 = x \\
            y_1 \circ y_2 = y}}
        \Sbf_{\Reduit(x_1) * \Reduit(y_1)} \otimes
        \Sbf_{\Reduit(x_2) * \Reduit(y_2)}.
    \end{equation}
    The coproduct $\Delta$ is a morphism of algebras if
    \eqref{equ:PRO_vers_AHC_coproduit_morphisme_1} and
    \eqref{equ:PRO_vers_AHC_coproduit_morphisme_2} are equal. Let us
    show that it is the case.
    \smallskip

    Assume that there are two elements $z$ and $t$ of $\Pca$ such that
    $z \circ t = x * y$. Then, the pair $(z, t)$ contributes to the
    coefficient of the tensor $\Sbf_{\Reduit(z)} \otimes \Sbf_{\Reduit(t)}$
    in \eqref{equ:PRO_vers_AHC_coproduit_morphisme_1}. Moreover, since
    $\Pca$ is free, by Lemma \ref{lem:regle_du_carre}, there exist four
    unique elements $x_1$, $x_2$, $y_1$, and $y_2$ of $\Pca$ such that
    $x = x_1 \circ x_2$, $y = y_1 \circ y_2$, $z = x_1 * y_1$, and
    $t = x_2 * y_2$. Then, since
    $\Reduit(x_1) * \Reduit(y_1) = \Reduit(x_1 * y_1) = \Reduit(z)$
    and $\Reduit(x_2) * \Reduit(y_2) = \Reduit(x_2 * y_2) = \Reduit(t)$,
    the quadruple $(x_1, x_2, y_1, y_2)$, wholly and uniquely determined
    by the pair $(z, t)$, contributes to the coefficient of the tensor
    $\Sbf_{\Reduit(z)} \otimes \Sbf_{\Reduit(t)}$ in
    \eqref{equ:PRO_vers_AHC_coproduit_morphisme_2}.
    \smallskip

    Conversely, assume that there are four elements $x_1$, $x_2$, $y_1$,
    and $y_2$ in $\Pca$ such that $x_1 \circ x_2 = x$ and $y_1 \circ y_2 = y$.
    Then, the quadruple $(x_1, x_2, y_1, y_2)$ contributes to the
    coefficient of the tensor
    $\Sbf_{\Reduit(x_1) * \Reduit(y_1)} \otimes \Sbf_{\Reduit(x_2) * \Reduit(y_2)}$
    in \eqref{equ:PRO_vers_AHC_coproduit_morphisme_2}. Now, by
    \eqref{equ:compo_h_v}, we have
    \begin{equation}
        x * y = (x_1 \circ x_2) * (y_1 \circ y_2) = (x_1 * y_1) \circ (x_2 * y_2).
    \end{equation}
    Then, since $\Reduit(x_1) * \Reduit(y_1) = \Reduit(x_1 * y_1)$
    and $\Reduit(x_2) * \Reduit(y_2) = \Reduit(x_2 * y_2)$, the pair
    $(x_1 * y_1, x_2 * y_2)$, wholly and uniquely determined by the
    quadruple $(x_1, x_2, y_1, y_2)$, contributes to the coefficient of
    the tensor
    $\Sbf_{\Reduit(x_1) * \Reduit(y_1)} \otimes \Sbf_{\Reduit(x_2) * \Reduit(y_2)}$
    in \eqref{equ:PRO_vers_AHC_coproduit_morphisme_1}.
    \smallskip

    Hence, the coefficient of any tensor is the same in
    \eqref{equ:PRO_vers_AHC_coproduit_morphisme_1} and in
    \eqref{equ:PRO_vers_AHC_coproduit_morphisme_2}. Then, these
    expressions are equal.
\end{proof}
\medskip

\begin{Theoreme} \label{thm:PRO_vers_AHC_bigebre}
    Let $\Pca$ be a free PRO. Then, $\PvH(\Pca)$ is a bialgebra.
\end{Theoreme}
\begin{proof}
    The associativity of $\cdot$ comes directly from the associativity of
    the horizontal composition $*$ of $\Pca$ (see \eqref{equ:assoc_compo_h}).
    Moreover, by Lemmas \ref{lem:PRO_vers_AHC_coassociativite} and
    \ref{lem:PRO_vers_AHC_compatibilite}, the coproduct of $\PvH(\Pca)$
    is coassociative and is a morphism of algebras. Thus, $\PvH(\Pca)$ is
    a bialgebra.
\end{proof}
\medskip

\subsection{Properties of the construction}
Let us now study the general properties of the bialgebras obtained
by the construction $\PvH$.
\medskip

\subsubsection{Algebraic generators and freeness}
\begin{Proposition} \label{prop:PRO_vers_AHC_generation_liberte}
    Let $\Pca$ be a free PRO. Then, $\PvH(\Pca)$ is freely generated
    as an algebra by the set of all $\Sbf_g$, where the $g$ are
    indecomposable and reduced elements of $\Pca$.
\end{Proposition}
\begin{proof}
    Any reduced element $x$ of $\Pca$ can be written as
    $x = x_1 * \dots * x_\ell$ where $(x_1, \dots, x_\ell)$ is the
    maximal decomposition of $x$. This implies that in $\PvH(\Pca)$,
    we have
    $\Sbf_x = \Sbf_{x_1} \cdot \, \dots \, \cdot \Sbf_{x_\ell}$. Since
    for all $i \in [\ell]$, the $\Sbf_{x_i}$ are indecomposable and
    reduced elements of $\Pca$, the set of all $\Sbf_g$ generates
    $\PvH(\Pca)$. The uniqueness of the maximal decomposition of $x$
    implies that $\PvH(\Pca)$ is free on the $\Sbf_g$.
\end{proof}
\medskip

\subsubsection{Gradings}
There are several ways to define gradings for $\PvH(\Pca)$ to turn it
into a combinatorial Hopf algebra. For this purpose, we say that a map
$\omega : \Reduit(\Pca) \to \EnsNat$ is a {\em grading} of $\Pca$
if it satisfies the following four properties:
\begin{enumerate}[label = ({\it G{\arabic*})}]
    \item \label{item:propriete_bonne_graduation_1}
    for any reduced elements $x$ and $y$ of $\Pca$,
    $\omega(x * y) = \omega(x) + \omega(y)$;
    \item \label{item:propriete_bonne_graduation_2}
    for any reduced elements $x$ of $\Pca$ satisfying $x = y \circ z$ where
    $y, z \in \Pca$,
    $\omega(x) = \omega(\Reduit(y)) + \omega(\Reduit(z))$;
    \item \label{item:propriete_bonne_graduation_3}
    for any $n \geq 0$, the fiber $\omega^{-1}(n)$ is finite;
    \item \label{item:propriete_bonne_graduation_4}
    $\omega^{-1}(0) = \{\Unite_0\}$.
\end{enumerate}
\medskip

A very generic way to endow $\Pca$ with a grading consists in
providing a map $\omega : G \to \EnsNat \setminus \{0\}$
associating a positive integer with any generator of $\Pca$, namely its
{\em weight}; the degree $\omega(x)$ of any element $x$ of $\Pca$ being
the sum of the weights of the occurrences of the generators used to build $x$.
For instance, the map $\omega$ defined by $\omega(\La) := 3$ and
$\omega(\Lb) := 2$ is a grading of $\AB$ and we have
\begin{equation}
    \begin{split} \omega \end{split}\left(
    \begin{split}\scalebox{.25}{\begin{tikzpicture}[yscale=.65]
        \node[Feuille](S1)at(0,0){};
        \node[Feuille](S2)at(2,0){};
        \node[Feuille](S3)at(5,0){};
        \node[Feuille](S4)at(7,0){};
        \node[Operateur](N1)at(1,-2){\begin{math}\La\end{math}};
        \node[Operateur,Marque1](N2)at(3,-4){\begin{math}\Lb\end{math}};
        \node[Operateur](N3)at(6,-3){\begin{math}\La\end{math}};
        \node[Feuille](E1)at(0,-6){};
        \node[Feuille](E2)at(2,-6){};
        \node[Feuille](E3)at(3,-6){};
        \node[Feuille](E4)at(4,-6){};
        \node[Feuille](E5)at(5,-6){};
        \node[Feuille](E6)at(7,-6){};
        \draw[Arete](N1)--(S1);
        \draw[Arete](N1)--(S2);
        \draw[Arete](N3)--(S3);
        \draw[Arete](N3)--(S4);
        \draw[Arete](N1)--(E1);
        \draw[Arete](N1)--(N2);
        \draw[Arete](N2)--(E2);
        \draw[Arete](N2)--(E3);
        \draw[Arete](N2)--(E4);
        \draw[Arete](N3)--(E5);
        \draw[Arete](N3)--(E6);
    \end{tikzpicture}}\end{split}\right)
    \begin{split} = 8\end{split}.
\end{equation}
\medskip

\begin{Proposition} \label{prop:PRO_vers_AHC_graduation}
    Let $\Pca$ be a free PRO and $\omega$ be a grading of $\Pca$.
    Then, with the grading
    \begin{equation} \label{equ:graduation}
        \PvH(\Pca) =
        \bigoplus_{n \geq 0}
        \Vect\left(\Sbf_x : x \in \Reduit(\Pca) \mbox{ and }
        \omega(x) = n \right),
    \end{equation}
    $\PvH(\Pca)$ is a combinatorial Hopf algebra.
\end{Proposition}
\begin{proof}
    Notice first that $\Sbf_{\Unite_0}$ is the neutral element of the
    product of $\PvH(\Pca)$, and, for any reduced element $x$ of $\Pca$
    of degree different from $0$, the coproduct $\Delta(\Sbf_x)$ contains
    the tensors $\Sbf_{\Unite_0} \otimes \Sbf_x$ and
    $\Sbf_x \otimes \Sbf_{\Unite_0}$. Then, together with the fact that
    by \ref{item:propriete_bonne_graduation_4}, $\Sbf_{\Unite_0}$ is an
    element of the homogeneous component of degree $0$ of $\PvH(\Pca)$,
    $\PvH(\Pca)$ admits a unit and a counit. Moreover,
    \ref{item:propriete_bonne_graduation_3} implies that for all $n \geq 0$,
    the homogeneous components of degree $n$ of $\PvH(\Pca)$ are finite,
    and \ref{item:propriete_bonne_graduation_4} implies that $\PvH(\Pca)$
    is connected.
    \smallskip

    Besides, respectively by \ref{item:propriete_bonne_graduation_1} and
    by \ref{item:propriete_bonne_graduation_2}, the grading provided by
    $\omega$ is compatible with the product and the coproduct of
    $\PvH(\Pca)$.
    \smallskip

    Hence, together with the fact that, by Theorem \ref{thm:PRO_vers_AHC_bigebre},
    $\PvH(\Pca)$ is a bialgebra, it is also a combinatorial Hopf algebra.
\end{proof}
\medskip

\subsubsection{Antipode}
Since the antipode of a combinatorial Hopf algebra can be computed by
induction on the degree, we obtain an expression for the one of $\PvH(\Pca)$
when $\Pca$ admits a grading. This expression is an instance of the
Takeuchi formula \cite{Tak71} and is particularly simple since the product
of $\PvH(\Pca)$ is multiplicative.
\medskip

\begin{Proposition} \label{prop:PRO_vers_AHC_antipode}
    Let $\Pca$ be a free PRO admitting a grading. For any reduced element
    $x$ of $\Pca$ different from $\Unite_0$, the antipode $\nu$ of $\PvH(\Pca)$
    satisfies
    \begin{equation} \label{equ:PRO_vers_AHC_antipode}
        \nu(\Sbf_x) =
        \sum_{\substack{x_1, \dots, x_\ell \in \Pca, \ell \geq 1 \\
        x_1 \circ \dots \circ x_\ell = x \\
        \Reduit(x_i) \ne \Unite_0, i \in [\ell]}}
        (-1)^\ell \;
        \Sbf_{\Reduit(x_1 * \dots * x_\ell)}.
    \end{equation}
\end{Proposition}
\begin{proof}
    By Proposition \ref{prop:PRO_vers_AHC_graduation}, $\PvH(\Pca)$
    is a combinatorial Hopf algebra, and hence, the antipode $\nu$, which
    is the inverse of the identity morphism for the convolution product,
    exists and is unique. Then, for any reduced element $x$ of $\Pca$
    different from $\Unite_0$,
    \begin{equation}
        \nu(\Sbf_x) = - \Sbf_x -
        \sum_{\substack{y, z \in \Pca \setminus \{x\} \\ y \circ z = x}}
        \nu\left(\Sbf_{\Reduit(y)}\right) \cdot \Sbf_{\Reduit(z)}.
    \end{equation}
    Expression \eqref{equ:PRO_vers_AHC_antipode} for $\nu$ follows now
    by induction on the degree of $x$ in $\Pca$.
\end{proof}
\medskip

We have for instance in $\PvH(\AB)$,
\begin{equation}
    \nu
    \Sbf_{
    \begin{split}\scalebox{.25}{\begin{tikzpicture}[yscale=.65]
        \node[Feuille](S1)at(0,0){};
        \node[Feuille](S2)at(2,0){};
        \node[Operateur](N1)at(1,-2){\begin{math}\La\end{math}};
        \node[Operateur,Marque1](N2)at(4,-4){\begin{math}\Lb\end{math}};
        \node[Operateur](N3)at(1,-6){\begin{math}\La\end{math}};
        \node[Feuille](E1)at(0,-8){};
        \node[Feuille](E2)at(2,-8){};
        \node[Feuille](E3)at(4,-8){};
        \node[Feuille](E4)at(6,-8){};
        \draw[Arete](N1)--(S1);
        \draw[Arete](N1)--(S2);
        \draw[Arete](N3)--(E1);
        \draw[Arete](N3)--(E2);
        \draw[Arete](N2)--(E3);
        \draw[Arete](N2)--(E4);
        \draw[Arete](N1)--(N3);
        \draw[Arete](N1)--(N2);
        \draw[Arete](N2)--(N3);
    \end{tikzpicture}}\end{split}}
    \enspace = \enspace
    - \Sbf_{
    \begin{split}\scalebox{.25}{\begin{tikzpicture}[yscale=.65]
        \node[Feuille](S1)at(0,0){};
        \node[Feuille](S2)at(2,0){};
        \node[Operateur](N1)at(1,-2){\begin{math}\La\end{math}};
        \node[Operateur,Marque1](N2)at(4,-4){\begin{math}\Lb\end{math}};
        \node[Operateur](N3)at(1,-6){\begin{math}\La\end{math}};
        \node[Feuille](E1)at(0,-8){};
        \node[Feuille](E2)at(2,-8){};
        \node[Feuille](E3)at(4,-8){};
        \node[Feuille](E4)at(6,-8){};
        \draw[Arete](N1)--(S1);
        \draw[Arete](N1)--(S2);
        \draw[Arete](N3)--(E1);
        \draw[Arete](N3)--(E2);
        \draw[Arete](N2)--(E3);
        \draw[Arete](N2)--(E4);
        \draw[Arete](N1)--(N3);
        \draw[Arete](N1)--(N2);
        \draw[Arete](N2)--(N3);
    \end{tikzpicture}}\end{split}}
    \enspace + \enspace
    \Sbf_{
    \begin{split}\scalebox{.25}{\begin{tikzpicture}[yscale=.65]
        \node[Feuille](S1)at(0,0){};
        \node[Feuille](S2)at(2,0){};
        \node[Feuille](S3)at(3,0){};
        \node[Feuille](S4)at(6,0){};
        \node[Operateur](N1)at(1,-3){\begin{math}\La\end{math}};
        \node[Operateur,Marque1](N2)at(6,-2){\begin{math}\Lb\end{math}};
        \node[Operateur](N3)at(4,-4){\begin{math}\La\end{math}};
        \node[Feuille](E1)at(0,-6) {};
        \node[Feuille](E2)at(2,-6){};
        \node[Feuille](E3)at(3,-6){};
        \node[Feuille](E4)at(5,-6){};
        \node[Feuille](E5)at(6,-6){};
        \node[Feuille](E6)at(8,-6){};
        \draw[Arete](N1)--(S1);
        \draw[Arete](N1)--(S2);
        \draw[Arete](N1)--(E1);
        \draw[Arete](N1)--(E2);
        \draw[Arete](N3)--(S3);
        \draw[Arete](N2)--(S4);
        \draw[Arete](N2)--(N3);
        \draw[Arete](N3)--(E3);
        \draw[Arete](N3)--(E4);
        \draw[Arete](N2)--(E5);
        \draw[Arete](N2)--(E6);
    \end{tikzpicture}}\end{split}}
    \enspace + \enspace
    \Sbf_{
    \begin{split}\scalebox{.25}{\begin{tikzpicture}[yscale=.65]
        \node[Feuille](S1)at(0,0){};
        \node[Feuille](S2)at(2,0){};
        \node[Feuille](S3)at(5,0){};
        \node[Feuille](S4)at(7,0){};
        \node[Operateur](N1)at(1,-2){\begin{math}\La\end{math}};
        \node[Operateur,Marque1](N2)at(3,-4){\begin{math}\Lb\end{math}};
        \node[Operateur](N3)at(6,-3){\begin{math}\La\end{math}};
        \node[Feuille](E1)at(0,-6){};
        \node[Feuille](E2)at(2,-6){};
        \node[Feuille](E3)at(3,-6){};
        \node[Feuille](E4)at(4,-6){};
        \node[Feuille](E5)at(5,-6){};
        \node[Feuille](E6)at(7,-6){};
        \draw[Arete](N1)--(S1);
        \draw[Arete](N1)--(S2);
        \draw[Arete](N3)--(S3);
        \draw[Arete](N3)--(S4);
        \draw[Arete](N1)--(E1);
        \draw[Arete](N1)--(N2);
        \draw[Arete](N2)--(E2);
        \draw[Arete](N2)--(E3);
        \draw[Arete](N2)--(E4);
        \draw[Arete](N3)--(E5);
        \draw[Arete](N3)--(E6);
    \end{tikzpicture}}\end{split}}
    \enspace - \enspace
    \Sbf_{
    \begin{split}\scalebox{.25}{\begin{tikzpicture}[yscale=.65]
        \node[Feuille](S1)at(0,0){};
        \node[Feuille](S2)at(2,0){};
        \node[Feuille](S3)at(4,0){};
        \node[Feuille](S4)at(6,0){};
        \node[Feuille](S5)at(8,0){};
        \node[Operateur](N1)at(1,-2){\begin{math}\La\end{math}};
        \node[Operateur,Marque1](N2)at(4,-2){\begin{math}\Lb\end{math}};
        \node[Operateur](N3)at(7,-2){\begin{math}\La\end{math}};
        \node[Feuille](E1)at(0,-4){};
        \node[Feuille](E2)at(2,-4){};
        \node[Feuille](E3)at(3,-4){};
        \node[Feuille](E4)at(4,-4){};
        \node[Feuille](E5)at(5,-4){};
        \node[Feuille](E6)at(6,-4){};
        \node[Feuille](E7)at(8,-4){};
        \draw[Arete](N1)--(S1);
        \draw[Arete](N1)--(S2);
        \draw[Arete](N1)--(E1);
        \draw[Arete](N1)--(E2);
        \draw[Arete](N2)--(S3);
        \draw[Arete](N2)--(E3);
        \draw[Arete](N2)--(E4);
        \draw[Arete](N2)--(E5);
        \draw[Arete](N3)--(S4);
        \draw[Arete](N3)--(S5);
        \draw[Arete](N3)--(E6);
        \draw[Arete](N3)--(E7);
    \end{tikzpicture}}\end{split}}\,.
\end{equation}
\medskip

\subsubsection{Duality}
When $\Pca$ admits a grading, let us denote by $\PvH(\Pca)^\star$
the graded dual of $\PvH(\Pca)$. By definition, the adjoint basis of the
fundamental basis of $\PvH(\Pca)$ consists in the elements $\Sbf^\star_x$,
$x \in \Reduit(\Pca)$.
\medskip

\begin{Proposition} \label{prop:PRO_vers_AHC_dual}
    Let $\Pca$ be a free PRO admitting a grading. Then, for any
    reduced elements $x$ and $y$ of $\Pca$, the product and the
    coproduct of $\PvH(\Pca)^\star$ satisfy
    \begin{equation} \label{equ:PRO_vers_AHC_produit_dual}
        \Sbf^\star_x \cdot \Sbf^\star_y =
        \sum_{\substack{x', y' \in \Pca \\
        x' \circ y' \in \Reduit(\Pca) \\
        \Reduit(x') = x, \Reduit(y') = y}}
        \Sbf^\star_{x' \circ y'}
    \end{equation}
    and
    \begin{equation} \label{equ:PRO_vers_AHC_coproduit_dual}
        \Delta\left(\Sbf^\star_x\right) =
        \sum_{\substack{y, z \in \Pca \\ y * z = x}}
        \Sbf^\star_y \otimes \Sbf^\star_z.
    \end{equation}
\end{Proposition}
\begin{proof}
    Let us denote by
    $\langle -, -\rangle : \PvH(\Pca) \otimes \PvH(\Pca)^\star \to \C$
    the duality bracket between $\PvH(\Pca)$ and its graded dual.
    \smallskip

    By duality, we have
    \begin{equation}
        \Sbf^\star_x \cdot \Sbf^\star_y =
        \sum_{z \in \Reduit(\Pca)}
        \left\langle \Delta(\Sbf_z),
        \Sbf^\star_x \otimes \Sbf^\star_y \right\rangle \;
        \Sbf^\star_z.
    \end{equation}
    Expression \eqref{equ:PRO_vers_AHC_produit_dual} follows from the
    fact that for any reduced element $z$ of $\Pca$, $\Sbf_x \otimes \Sbf_y$
    appears in $\Delta(\Sbf_z)$ if and only if there exist $x', y' \in \Pca$
    such that $x' \circ y' = z$, $\Reduit(x') = x$ and $\Reduit(y') = y$.
    \smallskip

    Besides, again by duality, we have
    \begin{equation}
        \Delta(\Sbf^\star_x) =
        \sum_{y, z \in \Reduit(\Pca)}
        \left\langle \Sbf_y \cdot \Sbf_z, \Sbf^\star_x \right\rangle \;
        \Sbf^\star_y \otimes \Sbf^\star_z.
    \end{equation}
    Expression \eqref{equ:PRO_vers_AHC_coproduit_dual} follows from the
    fact that for any reduced elements $y$ and $z$ of $\Pca$, $\Sbf_x$
    appears in $\Sbf_y \cdot \Sbf_z$ if and only if $y * z = x$.
\end{proof}
\medskip

For instance, we have in $\PvH(\AB)$
\begin{multline}
    \Sbf^\star_{
    \begin{split}\scalebox{.25}{\begin{tikzpicture}[yscale=.65]
        \node[Feuille](S1)at(0,0){};
        \node[Feuille](S2)at(2,0){};
        \node[Operateur](N1)at(1,-2){\begin{math}\La\end{math}};
        \node[Feuille](E1)at(0,-4){};
        \node[Feuille](E2)at(2,-4){};
        \draw[Arete](N1)--(S1);
        \draw[Arete](N1)--(S2);
        \draw[Arete](N1)--(E1);
        \draw[Arete](N1)--(E2);
    \end{tikzpicture}}\end{split}}
    \cdot
    \Sbf^\star_{
    \begin{split}\scalebox{.25}{\begin{tikzpicture}[yscale=.65]
        \node[Feuille](S1)at(1,0){};
        \node[Feuille](S2)at(3,0){};
        \node[Feuille](S3)at(5,0){};
        \node[Operateur,Marque1](N1)at(1,-2){\begin{math}\Lb\end{math}};
        \node[Operateur](N2)at(4,-2){\begin{math}\La\end{math}};
        \node[Feuille](E1)at(0,-4){};
        \node[Feuille](E2)at(1,-4){};
        \node[Feuille](E3)at(2,-4){};
        \node[Feuille](E4)at(3,-4){};
        \node[Feuille](E5)at(5,-4){};
        \draw[Arete](N1)--(S1);
        \draw[Arete](N1)--(E1);
        \draw[Arete](N1)--(E2);
        \draw[Arete](N1)--(E3);
        \draw[Arete](N2)--(S2);
        \draw[Arete](N2)--(S3);
        \draw[Arete](N2)--(E4);
        \draw[Arete](N2)--(E5);
    \end{tikzpicture}}\end{split}}
    \enspace = \enspace
    \Sbf^\star_{
    \begin{split}\scalebox{.25}{\begin{tikzpicture}[yscale=.65]
        \node[Feuille](SS1)at(-3,0){};
        \node[Feuille](SS2)at(-1,0){};
        \node[Operateur](NN1)at(-2,-2){\begin{math}\La\end{math}};
        \node[Feuille](EE1)at(-3,-4){};
        \node[Feuille](EE2)at(-1,-4){};
        \draw[Arete](NN1)--(SS1);
        \draw[Arete](NN1)--(SS2);
        \draw[Arete](NN1)--(EE1);
        \draw[Arete](NN1)--(EE2);
        \node[Feuille](S1)at(1,0){};
        \node[Feuille](S2)at(3,0){};
        \node[Feuille](S3)at(5,0){};
        \node[Operateur,Marque1](N1)at(1,-2){\begin{math}\Lb\end{math}};
        \node[Operateur](N2)at(4,-2){\begin{math}\La\end{math}};
        \node[Feuille](E1)at(0,-4){};
        \node[Feuille](E2)at(1,-4){};
        \node[Feuille](E3)at(2,-4){};
        \node[Feuille](E4)at(3,-4){};
        \node[Feuille](E5)at(5,-4){};
        \draw[Arete](N1)--(S1);
        \draw[Arete](N1)--(E1);
        \draw[Arete](N1)--(E2);
        \draw[Arete](N1)--(E3);
        \draw[Arete](N2)--(S2);
        \draw[Arete](N2)--(S3);
        \draw[Arete](N2)--(E4);
        \draw[Arete](N2)--(E5);
    \end{tikzpicture}}\end{split}}
    \enspace + \enspace
    \Sbf^\star_{
    \begin{split}\scalebox{.25}{\begin{tikzpicture}[yscale=.65]
        \node[Operateur](N1)at(1,-2){\begin{math}\La\end{math}};
        \node[Operateur,Marque1](N2)at(2,-5){\begin{math}\Lb\end{math}};
        \node[Operateur](N3)at(5,-3.5){\begin{math}\La\end{math}};
        \node[Feuille](S1)at(0,0){};
        \node[Feuille](S2)at(2,0){};
        \node[Feuille](S3)at(4,0){};
        \node[Feuille](S4)at(6,0){};
        \node[Feuille](E1)at(-1,-7){};
        \node[Feuille](E2)at(1,-7){};
        \node[Feuille](E3)at(2,-7){};
        \node[Feuille](E4)at(3,-7){};
        \node[Feuille](E5)at(4,-7){};
        \node[Feuille](E6)at(6,-7){};
        \draw[Arete](N1)--(N2);
        \draw[Arete](N1)--(S1);
        \draw[Arete](N1)--(S2);
        \draw[Arete](N3)--(S3);
        \draw[Arete](N3)--(S4);
        \draw[Arete](N1)--(E1);
        \draw[Arete](N2)--(E2);
        \draw[Arete](N2)--(E3);
        \draw[Arete](N2)--(E4);
        \draw[Arete](N3)--(E5);
        \draw[Arete](N3)--(E6);
    \end{tikzpicture}}\end{split}}
    \enspace + \enspace
    \Sbf^\star_{
    \begin{split}\scalebox{.25}{\begin{tikzpicture}[yscale=.65]
        \node[Operateur](N1)at(1,-2){\begin{math}\La\end{math}};
        \node[Operateur,Marque1](N2)at(0,-5){\begin{math}\Lb\end{math}};
        \node[Operateur](N3)at(5,-3.5){\begin{math}\La\end{math}};
        \node[Feuille](S1)at(0,0){};
        \node[Feuille](S2)at(2,0){};
        \node[Feuille](S3)at(4,0){};
        \node[Feuille](S4)at(6,0){};
        \node[Feuille](E1)at(-1,-7){};
        \node[Feuille](E2)at(0,-7){};
        \node[Feuille](E3)at(1,-7){};
        \node[Feuille](E4)at(3,-7){};
        \node[Feuille](E5)at(4,-7){};
        \node[Feuille](E6)at(6,-7){};
        \draw[Arete](N1)--(N2);
        \draw[Arete](N1)--(S1);
        \draw[Arete](N1)--(S2);
        \draw[Arete](N3)--(S3);
        \draw[Arete](N3)--(S4);
        \draw[Arete](N2)--(E1);
        \draw[Arete](N2)--(E2);
        \draw[Arete](N2)--(E3);
        \draw[Arete](N1)--(E4);
        \draw[Arete](N3)--(E5);
        \draw[Arete](N3)--(E6);
    \end{tikzpicture}}\end{split}}
    \enspace + \enspace
    \Sbf^\star_{
    \begin{split}\scalebox{.25}{\begin{tikzpicture}[yscale=.65]
        \node[Operateur](N1)at(1,-2){\begin{math}\La\end{math}};
        \node[Operateur,Marque1](N2)at(-.5,-5){\begin{math}\Lb\end{math}};
        \node[Operateur](N3)at(2.5,-5){\begin{math}\La\end{math}};
        \node[Feuille](S1)at(0,0){};
        \node[Feuille](S2)at(2,0){};
        \node[Feuille](S3)at(4,0){};
        \node[Feuille](E1)at(-1.5,-7){};
        \node[Feuille](E2)at(-.5,-7){};
        \node[Feuille](E3)at(.5,-7){};
        \node[Feuille](E4)at(1.5,-7){};
        \node[Feuille](E5)at(3.5,-7){};
        \draw[Arete](N1)--(N2);
        \draw[Arete](N1)--(N3);
        \draw[Arete](N1)--(S1);
        \draw[Arete](N1)--(S2);
        \draw[Arete](N3)--(S3);
        \draw[Arete](N2)--(E1);
        \draw[Arete](N2)--(E2);
        \draw[Arete](N2)--(E3);
        \draw[Arete](N3)--(E4);
        \draw[Arete](N3)--(E5);
    \end{tikzpicture}}\end{split}} \\
    \enspace + \enspace
    \Sbf^\star_{
    \begin{split}\scalebox{.25}{\begin{tikzpicture}[yscale=.65]
        \node[Operateur,Marque1](N1)at(1,-3.5){\begin{math}\Lb\end{math}};
        \node[Operateur](N2)at(4,-2){\begin{math}\La\end{math}};
        \node[Operateur](N3)at(5.5,-5){\begin{math}\La\end{math}};
        \node[Feuille](S1)at(1,0){};
        \node[Feuille](S2)at(3,0){};
        \node[Feuille](S3)at(5,0){};
        \node[Feuille](S4)at(7,0){};
        \node[Feuille](E1)at(0,-7){};
        \node[Feuille](E2)at(1,-7){};
        \node[Feuille](E3)at(2,-7){};
        \node[Feuille](E4)at(3,-7){};
        \node[Feuille](E5)at(4.5,-7){};
        \node[Feuille](E6)at(6.5,-7){};
        \draw[Arete](N2)--(N3);
        \draw[Arete](N1)--(S1);
        \draw[Arete](N2)--(S2);
        \draw[Arete](N2)--(S3);
        \draw[Arete](N3)--(S4);
        \draw[Arete](N1)--(E1);
        \draw[Arete](N1)--(E2);
        \draw[Arete](N1)--(E3);
        \draw[Arete](N2)--(E4);
        \draw[Arete](N3)--(E5);
        \draw[Arete](N3)--(E6);
    \end{tikzpicture}}\end{split}}
    \enspace + \enspace
    \Sbf^\star_{
    \begin{split}\scalebox{.25}{\begin{tikzpicture}[yscale=.65]
        \node[Operateur,Marque1](N1)at(1,-3.5){\begin{math}\Lb\end{math}};
        \node[Operateur](N2)at(4,-2){\begin{math}\La\end{math}};
        \node[Operateur](N3)at(4,-5){\begin{math}\La\end{math}};
        \node[Feuille](S1)at(1,0){};
        \node[Feuille](S2)at(3,0){};
        \node[Feuille](S3)at(5,0){};
        \node[Feuille](E1)at(0,-7){};
        \node[Feuille](E2)at(1,-7){};
        \node[Feuille](E3)at(2,-7){};
        \node[Feuille](E4)at(3,-7){};
        \node[Feuille](E5)at(5,-7){};
        \draw[Arete](N2)edge[bend right] node[]{}(N3);
        \draw[Arete](N2)edge[bend left] node[]{}(N3);
        \draw[Arete](N1)--(S1);
        \draw[Arete](N2)--(S2);
        \draw[Arete](N2)--(S3);
        \draw[Arete](N1)--(E1);
        \draw[Arete](N1)--(E2);
        \draw[Arete](N1)--(E3);
        \draw[Arete](N3)--(E4);
        \draw[Arete](N3)--(E5);
    \end{tikzpicture}}\end{split}}
    \enspace + \enspace
    \Sbf^\star_{
    \begin{split}\scalebox{.25}{\begin{tikzpicture}[yscale=.65]
        \node[Operateur,Marque1](N1)at(1,-3.5){\begin{math}\Lb\end{math}};
        \node[Operateur](N2)at(5.5,-2){\begin{math}\La\end{math}};
        \node[Operateur](N3)at(4,-5){\begin{math}\La\end{math}};
        \node[Feuille](S1)at(1,0){};
        \node[Feuille](S2)at(2.5,0){};
        \node[Feuille](S3)at(4.5,0){};
        \node[Feuille](S4)at(6.5,0){};
        \node[Feuille](E1)at(0,-7){};
        \node[Feuille](E2)at(1,-7){};
        \node[Feuille](E3)at(2,-7){};
        \node[Feuille](E4)at(3,-7){};
        \node[Feuille](E5)at(5,-7){};
        \node[Feuille](E6)at(7,-7){};
        \draw[Arete](N2)--(N3);
        \draw[Arete](N1)--(S1);
        \draw[Arete](N3)--(S2);
        \draw[Arete](N2)--(S3);
        \draw[Arete](N2)--(S4);
        \draw[Arete](N1)--(E1);
        \draw[Arete](N1)--(E2);
        \draw[Arete](N1)--(E3);
        \draw[Arete](N3)--(E4);
        \draw[Arete](N3)--(E5);
        \draw[Arete](N2)--(E6);
    \end{tikzpicture}}\end{split}}
    \enspace + \enspace
    2 \,
    \Sbf^\star_{
    \begin{split}\scalebox{.25}{\begin{tikzpicture}[yscale=.65]
        \node[Feuille](SS1)at(-2,0){};
        \node[Operateur,Marque1](NN1)at(-2,-2){\begin{math}\Lb\end{math}};
        \node[Feuille](EE1)at(-3,-4){};
        \node[Feuille](EE2)at(-2,-4){};
        \node[Feuille](EE3)at(-1,-4){};
        \draw[Arete](NN1)--(SS1);
        \draw[Arete](NN1)--(EE1);
        \draw[Arete](NN1)--(EE2);
        \draw[Arete](NN1)--(EE3);
        \node[Feuille](S1)at(0,0){};
        \node[Feuille](S2)at(2,0){};
        \node[Operateur](N1)at(1,-2){\begin{math}\La\end{math}};
        \node[Feuille](E1)at(0,-4){};
        \node[Feuille](E2)at(2,-4){};
        \draw[Arete](N1)--(S1);
        \draw[Arete](N1)--(S2);
        \draw[Arete](N1)--(E1);
        \draw[Arete](N1)--(E2);
        \node[Feuille](SSS1)at(3,0){};
        \node[Feuille](SSS2)at(5,0){};
        \node[Operateur](N2)at(4,-2){\begin{math}\La\end{math}};
        \node[Feuille](EEE1)at(3,-4){};
        \node[Feuille](EEE2)at(5,-4){};
        \draw[Arete](N2)--(SSS1);
        \draw[Arete](N2)--(SSS2);
        \draw[Arete](N2)--(EEE1);
        \draw[Arete](N2)--(EEE2);
    \end{tikzpicture}}\end{split}}
\end{multline}
and
\begin{multline}
    \Delta
    \Sbf^\star_{
    \begin{split}\scalebox{.25}{\begin{tikzpicture}[yscale=.65]
        \node[Feuille](S1)at(0,0){};
        \node[Feuille](S2)at(2,0){};
        \node[Feuille](S3)at(4.5,0){};
        \node[Feuille](S4)at(8,0){};
        \node[Feuille](S5)at(10,0){};
        \node[Operateur](N1)at(1,-3.5){\begin{math}\La\end{math}};
        \node[Operateur,Marque1](N2)at(4.5,-2){\begin{math}\Lb\end{math}};
        \node[Operateur](N3)at(7,-5){\begin{math}\La\end{math}};
        \node[Operateur,Marque1](N4)at(10,-3.5){\begin{math}\Lb\end{math}};
        \node[Feuille](E1)at(0,-7){};
        \node[Feuille](E2)at(2,-7){};
        \node[Feuille](E3)at(3,-7){};
        \node[Feuille](E4)at(4.5,-7){};
        \node[Feuille](E5)at(6,-7){};
        \node[Feuille](E6)at(8,-7){};
        \node[Feuille](E7)at(9,-7){};
        \node[Feuille](E8)at(10,-7){};
        \node[Feuille](E9)at(11,-7){};
        \draw[Arete](N1)--(S1);
        \draw[Arete](N1)--(S2);
        \draw[Arete](N1)--(E1);
        \draw[Arete](N1)--(E2);
        \draw[Arete](N2)--(S3);
        \draw[Arete](N2)--(E3);
        \draw[Arete](N2)--(E4);
        \draw[Arete](N2)--(N3);
        \draw[Arete](N3)--(S4);
        \draw[Arete](N3)--(E5);
        \draw[Arete](N3)--(E6);
        \draw[Arete](N4)--(S5);
        \draw[Arete](N4)--(E7);
        \draw[Arete](N4)--(E8);
        \draw[Arete](N4)--(E9);
    \end{tikzpicture}}\end{split}}
    \enspace = \enspace
    \Sbf^\star_{\Unite_0} \otimes
    \Sbf^\star_{
    \begin{split}\scalebox{.25}{\begin{tikzpicture}[yscale=.65]
        \node[Feuille](S1)at(0,0){};
        \node[Feuille](S2)at(2,0){};
        \node[Feuille](S3)at(4.5,0){};
        \node[Feuille](S4)at(8,0){};
        \node[Feuille](S5)at(10,0){};
        \node[Operateur](N1)at(1,-3.5){\begin{math}\La\end{math}};
        \node[Operateur,Marque1](N2)at(4.5,-2){\begin{math}\Lb\end{math}};
        \node[Operateur](N3)at(7,-5){\begin{math}\La\end{math}};
        \node[Operateur,Marque1](N4)at(10,-3.5){\begin{math}\Lb\end{math}};
        \node[Feuille](E1)at(0,-7){};
        \node[Feuille](E2)at(2,-7){};
        \node[Feuille](E3)at(3,-7){};
        \node[Feuille](E4)at(4.5,-7){};
        \node[Feuille](E5)at(6,-7){};
        \node[Feuille](E6)at(8,-7){};
        \node[Feuille](E7)at(9,-7){};
        \node[Feuille](E8)at(10,-7){};
        \node[Feuille](E9)at(11,-7){};
        \draw[Arete](N1)--(S1);
        \draw[Arete](N1)--(S2);
        \draw[Arete](N1)--(E1);
        \draw[Arete](N1)--(E2);
        \draw[Arete](N2)--(S3);
        \draw[Arete](N2)--(E3);
        \draw[Arete](N2)--(E4);
        \draw[Arete](N2)--(N3);
        \draw[Arete](N3)--(S4);
        \draw[Arete](N3)--(E5);
        \draw[Arete](N3)--(E6);
        \draw[Arete](N4)--(S5);
        \draw[Arete](N4)--(E7);
        \draw[Arete](N4)--(E8);
        \draw[Arete](N4)--(E9);
    \end{tikzpicture}}\end{split}}
    \enspace + \enspace
    \Sbf^\star_{
    \begin{split}\scalebox{.25}{\begin{tikzpicture}[yscale=.65]
        \node[Feuille](S1)at(0,0){};
        \node[Feuille](S2)at(2,0){};
        \node[Operateur](N1)at(1,-2){\begin{math}\La\end{math}};
        \node[Feuille](E1)at(0,-4){};
        \node[Feuille](E2)at(2,-4){};
        \draw[Arete](N1)--(S1);
        \draw[Arete](N1)--(S2);
        \draw[Arete](N1)--(E1);
        \draw[Arete](N1)--(E2);
    \end{tikzpicture}}\end{split}}
    \otimes
    \Sbf^\star_{
    \begin{split}\scalebox{.25}{\begin{tikzpicture}[yscale=.65]
        \node[Feuille](S3)at(4.5,0){};
        \node[Feuille](S4)at(8,0){};
        \node[Feuille](S5)at(10,0){};
        \node[Operateur,Marque1](N2)at(4.5,-2){\begin{math}\Lb\end{math}};
        \node[Operateur](N3)at(7,-5){\begin{math}\La\end{math}};
        \node[Operateur,Marque1](N4)at(10,-3.5){\begin{math}\Lb\end{math}};
        \node[Feuille](E3)at(3,-7){};
        \node[Feuille](E4)at(4.5,-7){};
        \node[Feuille](E5)at(6,-7){};
        \node[Feuille](E6)at(8,-7){};
        \node[Feuille](E7)at(9,-7){};
        \node[Feuille](E8)at(10,-7){};
        \node[Feuille](E9)at(11,-7){};
        \draw[Arete](N2)--(S3);
        \draw[Arete](N2)--(E3);
        \draw[Arete](N2)--(E4);
        \draw[Arete](N2)--(N3);
        \draw[Arete](N3)--(S4);
        \draw[Arete](N3)--(E5);
        \draw[Arete](N3)--(E6);
        \draw[Arete](N4)--(S5);
        \draw[Arete](N4)--(E7);
        \draw[Arete](N4)--(E8);
        \draw[Arete](N4)--(E9);
    \end{tikzpicture}}\end{split}} \\
    \enspace + \enspace
    \Sbf^\star_{
    \begin{split}\scalebox{.25}{\begin{tikzpicture}[yscale=.65]
        \node[Feuille](S1)at(0,0){};
        \node[Feuille](S2)at(2,0){};
        \node[Feuille](S3)at(4.5,0){};
        \node[Feuille](S4)at(8,0){};
        \node[Operateur](N1)at(1,-3.5){\begin{math}\La\end{math}};
        \node[Operateur,Marque1](N2)at(4.5,-2){\begin{math}\Lb\end{math}};
        \node[Operateur](N3)at(7,-5){\begin{math}\La\end{math}};
        \node[Feuille](E1)at(0,-7){};
        \node[Feuille](E2)at(2,-7){};
        \node[Feuille](E3)at(3,-7){};
        \node[Feuille](E4)at(4.5,-7){};
        \node[Feuille](E5)at(6,-7){};
        \node[Feuille](E6)at(8,-7){};
        \draw[Arete](N1)--(S1);
        \draw[Arete](N1)--(S2);
        \draw[Arete](N1)--(E1);
        \draw[Arete](N1)--(E2);
        \draw[Arete](N2)--(S3);
        \draw[Arete](N2)--(E3);
        \draw[Arete](N2)--(E4);
        \draw[Arete](N2)--(N3);
        \draw[Arete](N3)--(S4);
        \draw[Arete](N3)--(E5);
        \draw[Arete](N3)--(E6);
    \end{tikzpicture}}\end{split}}
    \otimes
    \Sbf^\star_{
    \begin{split}\scalebox{.25}{\begin{tikzpicture}[yscale=.65]
        \node[Feuille](S1)at(1,0){};
        \node[Operateur,Marque1](N1)at(1,-2){\begin{math}\Lb\end{math}};
        \node[Feuille](E1)at(0,-4){};
        \node[Feuille](E2)at(1,-4){};
        \node[Feuille](E3)at(2,-4){};
        \draw[Arete](N1)--(S1);
        \draw[Arete](N1)--(E1);
        \draw[Arete](N1)--(E2);
        \draw[Arete](N1)--(E3);
    \end{tikzpicture}}\end{split}}
    \enspace + \enspace
    \Sbf^\star_{
    \begin{split}\scalebox{.25}{\begin{tikzpicture}[yscale=.65]
        \node[Feuille](S1)at(0,0){};
        \node[Feuille](S2)at(2,0){};
        \node[Feuille](S3)at(4.5,0){};
        \node[Feuille](S4)at(8,0){};
        \node[Feuille](S5)at(10,0){};
        \node[Operateur](N1)at(1,-3.5){\begin{math}\La\end{math}};
        \node[Operateur,Marque1](N2)at(4.5,-2){\begin{math}\Lb\end{math}};
        \node[Operateur](N3)at(7,-5){\begin{math}\La\end{math}};
        \node[Operateur,Marque1](N4)at(10,-3.5){\begin{math}\Lb\end{math}};
        \node[Feuille](E1)at(0,-7){};
        \node[Feuille](E2)at(2,-7){};
        \node[Feuille](E3)at(3,-7){};
        \node[Feuille](E4)at(4.5,-7){};
        \node[Feuille](E5)at(6,-7){};
        \node[Feuille](E6)at(8,-7){};
        \node[Feuille](E7)at(9,-7){};
        \node[Feuille](E8)at(10,-7){};
        \node[Feuille](E9)at(11,-7){};
        \draw[Arete](N1)--(S1);
        \draw[Arete](N1)--(S2);
        \draw[Arete](N1)--(E1);
        \draw[Arete](N1)--(E2);
        \draw[Arete](N2)--(S3);
        \draw[Arete](N2)--(E3);
        \draw[Arete](N2)--(E4);
        \draw[Arete](N2)--(N3);
        \draw[Arete](N3)--(S4);
        \draw[Arete](N3)--(E5);
        \draw[Arete](N3)--(E6);
        \draw[Arete](N4)--(S5);
        \draw[Arete](N4)--(E7);
        \draw[Arete](N4)--(E8);
        \draw[Arete](N4)--(E9);
    \end{tikzpicture}}\end{split}}
    \otimes
    \Sbf^\star_{\Unite_0}.
\end{multline}
\medskip

\subsubsection{Quotient bialgebras}

\begin{Proposition} \label{prop:sous_generateurs_libre_donne_quotient}
    Let $G$ and $G'$ be two bigraded sets such that $G' \subseteq G$.
    Then, the map $\phi : \PvH(\Free(G)) \to \PvH(\Free(G'))$ linearly
    defined, for any reduced element $x$ of $\Free(G)$, by
    \begin{equation}
        \phi(\Sbf_x) :=
        \begin{cases}
            \Sbf_x & \mbox{if } x \in \Free(G'), \\
            0 & \mbox{otherwise},
        \end{cases}
    \end{equation}
    is a surjective bialgebra morphism. Moreover, $\PvH(\Free(G'))$ is
    a quotient bialgebra of $\PvH(\Free(G))$.
\end{Proposition}
\begin{proof}
    Let $V$ be the linear span of the $\Sbf_x$ where the $x$ are reduced
    elements of $\Free(G) \setminus \Free(G')$. Immediately from the
    definitions of the product and the coproduct of $\PvH(\Free(G))$,
    we observe that $V$ is a bialgebra ideal of $\PvH(\Free(G))$. The
    map $\phi$ is the canonical projection from $\PvH(\Free(G))$
    to $\PvH(\Free(G))/_V \simeq \PvH(\Free(G'))$, whence the result.
\end{proof}
\medskip

\subsection{The Hopf algebra of a stiff PRO}
We now extend the construction $\PvH$ to a class a non-necessarily free
PROs. Still in this section, $\Pca$ is a free PRO.
\medskip

Let $\equiv$ be a congruence of $\Pca$. For any element $x$ of $\Pca$,
we denote by $[x]_\equiv$ (or by $[x]$ if the context is clear) the
$\equiv$-equivalence class of $x$. We say that $\equiv$ is a
{\em stiff congruence} if the following three properties are satisfied:
\begin{enumerate}[label = ({\it C{\arabic*}})]
    \item \label{item:propriete_bonnes_congruences_0}
    for any reduced element $x$ of $\Pca$, the set $[x]$ is finite;
    \item \label{item:propriete_bonnes_congruences_1}
    for any reduced element $x$ of $\Pca$, $[x]$ contains reduced
    elements only;
    \item \label{item:propriete_bonnes_congruences_2}
    for any two elements $x$ and $x'$ of $\Pca$ such that $x \equiv x'$,
    the maximal decompositions of $x$ and $x'$ are, respectively of the
    form $(x_1, \dots, x_\ell)$ and $(x'_1, \dots, x'_\ell)$ for some
    $\ell \geq 0$, and for any $i \in [\ell]$, $x_i \equiv x'_i$.
\end{enumerate}
We say that a PRO is a {\em stiff PRO} if it is the quotient of a free
PRO by a stiff congruence.
\medskip

For any $\equiv$-equivalence class $[x]$ of reduced elements of $\Pca$,
set
\begin{equation} \label{equ:definition_des_T}
    \Tbf_{[x]} := \sum_{x' \in [x]} \Sbf_{x'}.
\end{equation}
Notice that thanks to \ref{item:propriete_bonnes_congruences_0}
and \ref{item:propriete_bonnes_congruences_1}, $\Tbf_{[x]}$ is a
well-defined element of $\PvH(\Pca)$.
\medskip

For instance, if $\Pca$ is the quotient of the free PRO generated by
$G := G(1, 1) \sqcup G(2, 2)$ where $G(1, 1) := \{\La\}$ and
$G(2,2) := \{\Lb\}$ by the finest congruence $\equiv$ satisfying
\begin{equation}
    \begin{split}
    \scalebox{.25}{\begin{tikzpicture}[yscale=.7]
        \node[Feuille](S1)at(0,0){};
        \node[Feuille](S2)at(2,0){};
        \node[Operateur,Marque1](N1)at(1,-2){\begin{math}\Lb\end{math}};
        \node[Operateur](N2)at(0,-5){\begin{math}\La\end{math}};
        \node[Feuille](E1)at(0,-7){};
        \node[Feuille](E2)at(2,-7){};
        \draw[Arete](N1)--(S1);
        \draw[Arete](N1)--(S2);
        \draw[Arete](N1)--(E2);
        \draw[Arete](N2)--(E1);
        \draw[Arete](N1)--(N2);
    \end{tikzpicture}}
    \end{split}
    \quad \equiv \quad
    \begin{split}
    \scalebox{.25}{\begin{tikzpicture}[yscale=.7]
        \node[Feuille](S1)at(0,0){};
        \node[Feuille](S2)at(2,0){};
        \node[Operateur,Marque1](N1)at(1,-2){\begin{math}\Lb\end{math}};
        \node[Operateur,Marque1](N2)at(1,-5){\begin{math}\Lb\end{math}};
        \node[Feuille](E1)at(0,-7){};
        \node[Feuille](E2)at(2,-7){};
        \draw[Arete](N1)--(S1);
        \draw[Arete](N1)--(S2);
        \draw[Arete](N2)--(E1);
        \draw[Arete](N2)--(E2);
        \draw[Arete,out=-110,in=110](N1)edge node[]{}(N2);
        \draw[Arete,out=-70,in=70](N1)edge node[]{}(N2);
    \end{tikzpicture}}
    \end{split}\,,
\end{equation}
one has
\begin{equation}
    \Tbf_{\left[
    \begin{split}\scalebox{.25}{\begin{tikzpicture}[yscale=.7]
        \node[Feuille](S1)at(0,0){};
        \node[Feuille](S2)at(2,0){};
        \node[Operateur,Marque1](N1)at(1,-2){\begin{math}\Lb\end{math}};
        \node[Operateur](N2)at(0,-5){\begin{math}\La\end{math}};
        \node[Operateur,Marque1](N3)at(1,-8){\begin{math}\Lb\end{math}};
        \node[Feuille](E1)at(0,-10){};
        \node[Feuille](E2)at(2,-10){};
        \draw[Arete](N1)--(S1);
        \draw[Arete](N1)--(S2);
        \draw[Arete](N3)--(E2);
        \draw[Arete](N3)--(E1);
        \draw[Arete](N1)--(N2);
        \draw[Arete](N2)--(N3);
        \draw[Arete,out=-70,in=70](N1)edge node[]{}(N3);
    \end{tikzpicture}}\end{split}\right]}
    \enspace = \enspace
    \Sbf_{
    \begin{split}\scalebox{.25}{\begin{tikzpicture}[yscale=.7]
        \node[Feuille](S1)at(0,0){};
        \node[Feuille](S2)at(2,0){};
        \node[Operateur,Marque1](N1)at(1,-2){\begin{math}\Lb\end{math}};
        \node[Operateur](N2)at(0,-5){\begin{math}\La\end{math}};
        \node[Operateur,Marque1](N3)at(1,-8){\begin{math}\Lb\end{math}};
        \node[Feuille](E1)at(0,-10){};
        \node[Feuille](E2)at(2,-10){};
        \draw[Arete](N1)--(S1);
        \draw[Arete](N1)--(S2);
        \draw[Arete](N3)--(E2);
        \draw[Arete](N3)--(E1);
        \draw[Arete](N1)--(N2);
        \draw[Arete](N2)--(N3);
        \draw[Arete,out=-70,in=70](N1)edge node[]{}(N3);
    \end{tikzpicture}}\end{split}}
    \enspace + \enspace
    \Sbf_{
    \begin{split}\scalebox{.25}{\begin{tikzpicture}[yscale=.7]
        \node[Feuille](S1)at(0,0){};
        \node[Feuille](S2)at(2,0){};
        \node[Operateur,Marque1](N1)at(1,-2){\begin{math}\Lb\end{math}};
        \node[Operateur,Marque1](N2)at(1,-5){\begin{math}\Lb\end{math}};
        \node[Operateur,Marque1](N3)at(1,-8){\begin{math}\Lb\end{math}};
        \node[Feuille](E1)at(0,-10){};
        \node[Feuille](E2)at(2,-10){};
        \draw[Arete](N1)--(S1);
        \draw[Arete](N1)--(S2);
        \draw[Arete](N3)--(E1);
        \draw[Arete](N3)--(E2);
        \draw[Arete,out=-110,in=110](N1)edge node[]{}(N2);
        \draw[Arete,out=-70,in=70](N1)edge node[]{}(N2);
        \draw[Arete,out=-110,in=110](N2)edge node[]{}(N3);
        \draw[Arete,out=-70,in=70](N2)edge node[]{}(N3);
    \end{tikzpicture}}\end{split}}
    \enspace + \enspace
    \Sbf_{
    \begin{split}\scalebox{.25}{\begin{tikzpicture}[yscale=.7]
        \node[Feuille](S1)at(0,0){};
        \node[Feuille](S2)at(2,0){};
        \node[Operateur,Marque1](N1)at(1,-2){\begin{math}\Lb\end{math}};
        \node[Operateur,Marque1](N2)at(1,-5){\begin{math}\Lb\end{math}};
        \node[Operateur](N3)at(0,-8){\begin{math}\La\end{math}};
        \node[Feuille](E1)at(0,-10){};
        \node[Feuille](E2)at(2,-10){};
        \draw[Arete](N1)--(S1);
        \draw[Arete](N1)--(S2);
        \draw[Arete](N3)--(E1);
        \draw[Arete](N2)--(E2);
        \draw[Arete,out=-110,in=110](N1)edge node[]{}(N2);
        \draw[Arete,out=-70,in=70](N1)edge node[]{}(N2);
        \draw[Arete](N2)--(N3);
    \end{tikzpicture}}\end{split}}
    \enspace + \enspace
    \Sbf_{
    \begin{split}\scalebox{.25}{\begin{tikzpicture}[yscale=.7]
        \node[Feuille](S1)at(0,0){};
        \node[Feuille](S2)at(2,0){};
        \node[Operateur,Marque1](N1)at(1,-2){\begin{math}\Lb\end{math}};
        \node[Operateur](N2)at(0,-5){\begin{math}\La\end{math}};
        \node[Operateur](N3)at(0,-8){\begin{math}\La\end{math}};
        \node[Feuille](E1)at(0,-10){};
        \node[Feuille](E2)at(2,-10){};
        \draw[Arete](N1)--(S1);
        \draw[Arete](N1)--(S2);
        \draw[Arete](N3)--(E1);
        \draw[Arete](N1)--(N2);
        \draw[Arete](N2)--(N3);
        \draw[Arete](N1)--(E2);
    \end{tikzpicture}}\end{split}}\,.
\end{equation}
Moreover, we can observe that $\equiv$ is a stiff congruence.
\medskip

If $\equiv$ is a stiff congruence of $\Pca$,
\ref{item:propriete_bonnes_congruences_1} and
\ref{item:propriete_bonnes_congruences_2} imply that all the
elements of a same $\equiv$-equivalence class $[x]$ have the same
number of factors and are all reduced or all nonreduced. Then, by
extension, we shall say that a $\equiv$-equivalence class $[x]$
of $\Pca/_\equiv$ is {\em indecomposable} (resp. {\em reduced}) if all
its elements are indecomposable (resp. reduced) in $\Pca$. In the
same way, the {\em wire} of $\Pca/_\equiv$ is the $\equiv$-equivalence
class of the wire of $\Pca$.
\medskip

We shall now study how the product and the coproduct of $\PvH(\Pca)$
behave on the $\Tbf_{[x]}$.
\medskip

\subsubsection{Product}
Let us show that the linear span of the $\Tbf_{[x]}$, where the $[x]$
are $\equiv$-equivalence classes of reduced elements of $\Pca$, forms a
subalgebra of $\PvH(\Pca)$. The product on the $\Tbf_{[x]}$ is
multiplicative and admits the following simple description.
\begin{Proposition} \label{prop:PRO_vers_AHC_congruence_produit}
    Let $\Pca$ be a free PRO and $\equiv$ be a stiff congruence of
    $\Pca$. Then, for any $\equiv$-equivalence classes $[x]$ and $[y]$,
    \begin{equation}
        \Tbf_{[x]} \cdot \Tbf_{[y]} = \Tbf_{[x * y]},
    \end{equation}
    where $x$ (resp. $y$) is any element of $[x]$ (resp. $[y]$).
\end{Proposition}
\begin{proof}
    We have
    \begin{equation} \label{equ:PRO_vers_AHC_congruence_produit_1}
        \Tbf_{[x]} \cdot \Tbf_{[y]} =
        \sum_{\substack{x' \in [x] \\ y' \in [y]}} \Sbf_{x' * y'}
    \end{equation}
    and
    \begin{equation} \label{equ:PRO_vers_AHC_congruence_produit_2}
        \Tbf_{[x * y]} = \sum_{z \in [x * y]} \Sbf_z.
    \end{equation}
    Let us show that \eqref{equ:PRO_vers_AHC_congruence_produit_1} and
    \eqref{equ:PRO_vers_AHC_congruence_produit_2} are equal. It is enough
    to check that these sums have the same support. Indeed,
    \eqref{equ:PRO_vers_AHC_congruence_produit_2} is by definition
    multiplicity free and \eqref{equ:PRO_vers_AHC_congruence_produit_1}
    is multiplicity free because $\Pca$ is free, and all elements
    of a $\equiv$-equivalence class have the same input arity and the
    same output arity.
    \smallskip

    Assume that there is a reduced element $t$ of $\Pca$ such that
    $\Sbf_t$ appears in \eqref{equ:PRO_vers_AHC_congruence_produit_1}.
    Then, one has $t = x' * y'$ for two reduced elements $x'$ and $y'$
    of $\Pca$ such that $x' \in [x]$ and $y' \in [y]$. Since $\equiv$ is
    a congruence of PROs, we have $[x' * y'] = [x * y]$ and thus,
    $t \in [x * y]$. This shows that $\Sbf_t$ also appears in
    \eqref{equ:PRO_vers_AHC_congruence_produit_2}.
    \smallskip

    Conversely, assume that there is a reduced element $z$ of $\Pca$ such
    that $\Sbf_z$ appears in \eqref{equ:PRO_vers_AHC_congruence_produit_2}.
    Then, one has $z \in [x * y]$. Since $\equiv$ satisfies
    \ref{item:propriete_bonnes_congruences_2}, the maximal decomposition
    of $z$ satisfies $\Dec(z) = (x'_1, \dots, x'_k, y'_1, \dots, y'_\ell)$
    where $\Dec(x) = (x_1, \dots, x_k)$, $\Dec(y) = (y_1, \dots, y_\ell)$,
    $x'_i \equiv x_i$, and $y'_j \equiv y_j$ for all $i \in [k]$ and
    $j \in [\ell]$. Moreover, as $\equiv$ is a congruence of PROs,
    $x' := x'_1 * \dots * x'_k \equiv x$ and
    $y' := y'_1 * \dots * y'_\ell \equiv y$. We then have $z = x' * y'$
    with $x' \in [x]$ and $y' \in [y]$. This shows that $\Sbf_z$ also
    appears in \eqref{equ:PRO_vers_AHC_congruence_produit_1}.
\end{proof}
\medskip

\subsubsection{Coproduct}
To prove that the linear span of the $\Tbf_{[x]}$, where the $[x]$
are $\equiv$-equivalence classes of reduced elements of $\Pca$, forms a
subcoalgebra of $\PvH(\Pca)$ and provides the description of the coproduct
of a $\Tbf_{[x]}$, we need the following notation. For any element $x$
of $\Pca$,
\begin{equation} \label{equ:reduit_classe}
    \Reduit\left([x]\right) :=
    \left\{\Reduit\left(x'\right) : x' \in [x]\right\}.
\end{equation}
\medskip

\begin{Lemme} \label{lem:congruence_rigide_reduit_classe}
    Let $\Pca$ be a free PRO and $\equiv$ be a stiff congruence of
    $\Pca$. For any element $x$ of $\Pca$,
    \begin{equation}
        \Reduit\left([x]\right) = \left[\Reduit(x)\right].
    \end{equation}
\end{Lemme}
\begin{proof}
    Let us denote by $y$ the element $\Reduit(x)$ and let
    $y' \in \Reduit([x])$. Let us show that $y' \in [\Reduit(x)]$. By
    Definition \eqref{equ:reduit_classe}, there is an element $x'$ of
    $\Pca$ such that $x' \in [x]$ and $\Reduit(x') = y'$. Since $\equiv$
    satisfies \ref{item:propriete_bonnes_congruences_2}, $\Dec(x')$ and
    $\Dec(x)$ have the same length $\ell$ and $\Dec(x')_i \equiv \Dec(x)_i$
    for all $i \in [\ell]$. Moreover, since $\equiv$ satisfies
    \ref{item:propriete_bonnes_congruences_1}, for all $i \in [\ell]$,
    $\Dec(x')_i$ and $\Dec(x)_i$ are both reduced elements or are both
    wires. Hence, since $\Dec(y')$ and $\Dec(y)$ are, respectively
    subwords of $\Dec(x')$ and $\Dec(x)$, they have the same length $k$
    and $\Dec(y')_j \equiv \Dec(y)_j$ for all $j \in [k]$. Finally, since
    $\equiv$ is a congruence of PROs, $y' \equiv y$. This shows that
    $y' \in [\Reduit(x)]$ and hence, $\Reduit([x]) \subseteq [\Reduit(x)]$.
    \smallskip

    Again, let us denote by $y$ the element $\Reduit(x)$ and let
    $y' \in [\Reduit(x)]$. Let us show that $y' \in \Reduit([x])$.
    Since $y' \equiv y$ and $\equiv$ satisfies
    \ref{item:propriete_bonnes_congruences_2}, $\Dec(y')$ and $\Dec(y)$
    have the same length $\ell$ and $\Dec(y')_i \equiv \Dec(y)_i$ for all
    $i \in [\ell]$. Moreover, since $y = \Reduit(x)$, by
    Lemma \ref{lem:relation_element_et_son_reduit}, for some
    $p_1, \dots, p_{\ell + 1} \geq 0$, we have
    $x = \Unite_{p_1} * \Dec(y)_1 * \dots *
    \Unite_{p_\ell} * \Dec(y)_\ell * \Unite_{p_{\ell + 1}}$.
    Now, by setting
    $x' := \Unite_{p_1} * \Dec(y')_1 * \dots *
    \Unite_{p_\ell} * \Dec(y')_\ell * \Unite_{p_{\ell + 1}}$, the fact
    that $\equiv$ is a congruence of PROs implies $x' \equiv x$. Since
    $y' = \Reduit(x')$, this shows that $y' \in \Reduit([x])$ and hence,
    $[\Reduit(x)] \subseteq \Reduit([x])$.
\end{proof}
\medskip

\begin{Lemme} \label{lem:congruence_rigide_meme_reduit_meme_element}
    Let $\Pca$ be a free PRO, $\equiv$ be a stiff congruence of $\Pca$,
    and $y$ and $z$ be two elements of $\Pca$ such that $y \equiv z$.
    Then, $\Reduit(y) = \Reduit(z)$ implies $y = z$.
\end{Lemme}
\begin{proof}
    By contraposition, assume that $y \ne z$. Since $y \equiv z$ and
    $\equiv$ satisfies \ref{item:propriete_bonnes_congruences_2},
    $\Dec(y)$ and $\Dec(z)$ have the same length $\ell$, and
    $\Dec(y)_i \equiv \Dec(z)_i$ for all $i \in [\ell]$. Moreover, as
    $y \ne z$, there exists a $j \in [\ell]$ such that $\Dec(y)_j \ne \Dec(z)_j$.
    Now, since $\equiv$ satisfies \ref{item:propriete_bonnes_congruences_1},
    $\Dec(y)_j$ and $\Dec(z)_j$ are both reduced elements. Moreover, for
    all $i \in [\ell]$, $\Dec(y)_i$ and $\Dec(z)_i$ are both reduced
    elements or are both wires. Hence, there is $j' \geq 1$ such that
    $\Dec(\Reduit(y))_{j'} = \Dec(y)_j$ and
    $\Dec(\Reduit(z))_{j'} = \Dec(z)_j$. Since $\Pca$ is free, this
    implies that $\Reduit(y) \ne \Reduit(z)$.
\end{proof}
\medskip

\begin{Proposition} \label{prop:PRO_vers_AHC_congruence_coproduit}
    Let $\Pca$ be a free PRO and $\equiv$ be a stiff congruence of $\Pca$.
    Then, for any $\equiv$-equivalence class $[x]$,
    \begin{equation}
        \Delta\left(\Tbf_{[x]}\right) =
        \sum_{\substack{[y], [z] \in \Pca/_\equiv \\
            [y] \circ [z] = [x]}}
            \Tbf_{\Reduit([y])} \otimes \Tbf_{\Reduit([z])}.
    \end{equation}
\end{Proposition}
\begin{proof}
    We have
    \begin{align}
        \Delta\left(\Tbf_{[x]}\right)
        & = \sum_{\substack{y, z \in \Pca \\ y \circ z \in [x]}}
            \Sbf_{\Reduit(y)} \otimes \Sbf_{\Reduit(z)}
                \label{equ:coproduit_congruence_1} \\
        & = \sum_{\substack{y, z \in \Pca \\ [y] \circ [z] = [x]}}
            \Sbf_{\Reduit(y)} \otimes \Sbf_{\Reduit(z)}
                \label{equ:coproduit_congruence_2} \\
        & = \sum_{\substack{[y], [z] \in \Pca/_\equiv \\
                [y] \circ [z] = [x]}} \;
            \sum_{\substack{y' \in [y] \\ z' \in [z]}}
            \Sbf_{\Reduit(y')} \otimes \Sbf_{\Reduit(z')}
                \label{equ:coproduit_congruence_3} \\
        & = \sum_{\substack{[y], [z] \in \Pca/_\equiv \\
                [y] \circ [z] = [x]}} \;
            \sum_{\substack{y' \in \Reduit([y]) \\ z' \in \Reduit([z])}}
            \Sbf_{y'} \otimes \Sbf_{z'}
                \label{equ:coproduit_congruence_4} \\
        & = \sum_{\substack{[y], [z] \in \Pca/_\equiv \\
            [y] \circ [z] = [x]}}
            \Tbf_{\Reduit([y])} \otimes \Tbf_{\Reduit([z])}.
                \label{equ:coproduit_congruence_5}
    \end{align}
    Let us comment the non-obvious equalities appearing in this computation.
    The equality between \eqref{equ:coproduit_congruence_1} and
    \eqref{equ:coproduit_congruence_2} comes from the fact that $\equiv$
    is a congruence of PROs. The equality between
    \eqref{equ:coproduit_congruence_3} and
    \eqref{equ:coproduit_congruence_4} is a consequence of
    Lemma \ref{lem:congruence_rigide_meme_reduit_meme_element}. Finally,
    \eqref{equ:coproduit_congruence_4} is, thanks to
    Lemma \ref{lem:congruence_rigide_reduit_classe}, equal to
    \eqref{equ:coproduit_congruence_5}.
\end{proof}
\medskip

\subsubsection{Sub-bialgebra}
The description of the product and the coproduct on the $\Tbf_{[x]}$
leads to the following result.
\begin{Theoreme} \label{thm:PRO_vers_AHC_congruence}
    Let $\Pca$ be a free PRO and $\equiv$ be a stiff congruence of $\Pca$.
    Then, the linear span of the $\Tbf_{[x]}$, where the $[x]$ are
    $\equiv$-equivalence classes of reduced elements of $\Pca$, forms
    a sub-bialgebra of $\PvH(\Pca)$.
\end{Theoreme}
\begin{proof}
    By Propositions \ref{prop:PRO_vers_AHC_congruence_produit} and
    \ref{prop:PRO_vers_AHC_congruence_coproduit}, the product and the
    coproduct of $\PvH(\Pca)$ are still well-defined on the $\Tbf_{[x]}$.
    Then, since the $\Tbf_{[x]}$ are by \eqref{equ:definition_des_T}
    sums of some $\Sbf_{x'}$, this implies the statement of the theorem.
\end{proof}
\medskip

We shall denote, by a slight abuse of notation, by $\PvH(\Pca/_\equiv)$
the sub-bialgebra of $\PvH(\Pca)$ spanned by the $\Tbf_{[x]}$, where the
$[x]$ are $\equiv$-equivalence classes of reduced elements of $\Pca$.
Notice that the construction $\PvH$ as it was presented in
Section~\ref{subsec:PRO_libre_vers_AHC} is a special case of this latter
when $\equiv$ is the most refined congruence of PROs.
\medskip

Note that this construction of sub-bialgebras of $\PvH(\Pca)$ by taking
an equivalence relation satisfying some precise properties and by
considering the elements obtained by summing over its equivalence classes
is analog to the construction of certain sub-bialgebras of the
Malvenuto-Reutenauer Hopf algebra \cite{MR95}. Indeed, some famous
Hopf algebras are obtained in this way, as the Loday-Ronco Hopf algebra
\cite{LR98} by using the sylvester monoid congruence \cite{HNT05}, or
the Poirier-Reutenauer Hopf algebra \cite{PR95} by using the plactic
monoid congruence \cite{DHT02, HNT05}.
\medskip

\subsubsection{The importance of the stiff congruence condition}
Let us now explain why the stiff congruence condition required as a
premise of Theorem \ref{thm:PRO_vers_AHC_congruence} is important by
providing an example of a non-stiff congruence of PROs failing to
produce a bialgebra.
\medskip

Consider the PRO $\Pca$ quotient of the free PRO generated by
$G := G(1, 1) \sqcup G(2, 2)$ where $G(1, 1) := \{\La\}$ and
$G(2, 2) := \{\Lb\}$ by the finest congruence $\equiv$ satisfying
\begin{equation}
    \begin{split}
    \scalebox{.25}{\begin{tikzpicture}[yscale=.8]
        \node[Feuille](S1)at(0,0){};
        \node[Feuille](S2)at(2,0){};
        \node[Operateur](N1)at(0,-2){\begin{math}\La\end{math}};
        \node[Operateur](N2)at(2,-2){\begin{math}\La\end{math}};
        \node[Feuille](E1)at(0,-4){};
        \node[Feuille](E2)at(2,-4){};
        \draw[Arete](N1)--(S1);
        \draw[Arete](N2)--(S2);
        \draw[Arete](N1)--(E1);
        \draw[Arete](N2)--(E2);
    \end{tikzpicture}}
    \end{split}
    \begin{split}\quad \equiv \quad\end{split}
    \begin{split}
    \scalebox{.25}{\begin{tikzpicture}[yscale=.8]
        \node[Feuille](S1)at(0,0){};
        \node[Feuille](S2)at(2,0){};
        \node[Operateur,Marque1](N1)at(1,-2){\begin{math}\Lb\end{math}};
        \node[Feuille](E1)at(0,-4){};
        \node[Feuille](E2)at(2,-4){};
        \draw[Arete](N1)--(S1);
        \draw[Arete](N1)--(S2);
        \draw[Arete](N1)--(E1);
        \draw[Arete](N1)--(E2);
    \end{tikzpicture}}
    \end{split}
\end{equation}
Here, $\equiv$ is not a stiff congruence since it satisfies
\ref{item:propriete_bonnes_congruences_1} but not
\ref{item:propriete_bonnes_congruences_2}.
\medskip

We have
\begin{equation}
    \Tbf_{
    \left[
    \begin{split}
    \scalebox{.25}{\begin{tikzpicture}[yscale=.8]
        \node[Feuille](S1)at(0,0){};
        \node[Operateur](N1)at(0,-2){\begin{math}\La\end{math}};
        \node[Feuille](E1)at(0,-4){};
        \draw[Arete](N1)--(S1);
        \draw[Arete](N1)--(E1);
    \end{tikzpicture}}\end{split}
    \right]}
    \cdot
    \Tbf_{
    \left[
    \begin{split}
    \scalebox{.25}{\begin{tikzpicture}[yscale=.8]
        \node[Feuille](S1)at(0,0){};
        \node[Operateur](N1)at(0,-2){\begin{math}\La\end{math}};
        \node[Feuille](E1)at(0,-4){};
        \draw[Arete](N1)--(S1);
        \draw[Arete](N1)--(E1);
    \end{tikzpicture}}\end{split}
    \right]}
    \enspace = \enspace
    \Sbf_{
    \begin{split}
    \scalebox{.25}{\begin{tikzpicture}[yscale=.8]
        \node[Feuille](S1)at(0,0){};
        \node[Operateur](N1)at(0,-2){\begin{math}\La\end{math}};
        \node[Feuille](E1)at(0,-4){};
        \draw[Arete](N1)--(S1);
        \draw[Arete](N1)--(E1);
    \end{tikzpicture}}\end{split}}
    \cdot
    \Sbf_{
    \begin{split}
    \scalebox{.25}{\begin{tikzpicture}[yscale=.8]
        \node[Feuille](S1)at(0,0){};
        \node[Operateur](N1)at(0,-2){\begin{math}\La\end{math}};
        \node[Feuille](E1)at(0,-4){};
        \draw[Arete](N1)--(S1);
        \draw[Arete](N1)--(E1);
    \end{tikzpicture}}\end{split}}
    \enspace = \enspace
    \Sbf_{
    \begin{split}
    \scalebox{.25}{\begin{tikzpicture}[yscale=.8]
        \node[Feuille](S1)at(0,0){};
        \node[Feuille](S2)at(2,0){};
        \node[Operateur](N1)at(0,-2){\begin{math}\La\end{math}};
        \node[Operateur](N2)at(2,-2){\begin{math}\La\end{math}};
        \node[Feuille](E1)at(0,-4){};
        \node[Feuille](E2)at(2,-4){};
        \draw[Arete](N1)--(S1);
        \draw[Arete](N2)--(S2);
        \draw[Arete](N1)--(E1);
        \draw[Arete](N2)--(E2);
    \end{tikzpicture}}\end{split}}
\end{equation}
but this last element cannot be expressed on the $\Tbf_{[x]}$.
\medskip

Besides, by a straightforward computation, we have
\begin{multline}
    \Delta \Tbf_{\left[
    \begin{split}
    \scalebox{.25}{\begin{tikzpicture}[yscale=.8]
        \node[Feuille](S1)at(0,0){};
        \node[Feuille](S2)at(2,0){};
        \node[Feuille](S3)at(4,0){};
        \node[Operateur](N1)at(0,-2){\begin{math}\La\end{math}};
        \node[Operateur](N2)at(2,-2){\begin{math}\La\end{math}};
        \node[Operateur](N3)at(4,-2){\begin{math}\La\end{math}};
        \node[Feuille](E1)at(0,-4){};
        \node[Feuille](E2)at(2,-4){};
        \node[Feuille](E3)at(4,-4){};
        \draw[Arete](N1)--(S1);
        \draw[Arete](N2)--(S2);
        \draw[Arete](N3)--(S3);
        \draw[Arete](N1)--(E1);
        \draw[Arete](N2)--(E2);
        \draw[Arete](N3)--(E3);
    \end{tikzpicture}}\end{split}\right]}
    \enspace = \enspace
    \Delta \Sbf_{
    \begin{split}
    \scalebox{.25}{\begin{tikzpicture}[yscale=.8]
        \node[Feuille](S1)at(0,0){};
        \node[Feuille](S2)at(2,0){};
        \node[Feuille](S3)at(4,0){};
        \node[Operateur](N1)at(0,-2){\begin{math}\La\end{math}};
        \node[Operateur](N2)at(2,-2){\begin{math}\La\end{math}};
        \node[Operateur](N3)at(4,-2){\begin{math}\La\end{math}};
        \node[Feuille](E1)at(0,-4){};
        \node[Feuille](E2)at(2,-4){};
        \node[Feuille](E3)at(4,-4){};
        \draw[Arete](N1)--(S1);
        \draw[Arete](N2)--(S2);
        \draw[Arete](N3)--(S3);
        \draw[Arete](N1)--(E1);
        \draw[Arete](N2)--(E2);
        \draw[Arete](N3)--(E3);
    \end{tikzpicture}}\end{split}}
    \enspace + \enspace
    \Delta \Sbf_{
    \begin{split}
    \scalebox{.25}{\begin{tikzpicture}[yscale=.8]
        \node[Feuille](S1)at(0,0){};
        \node[Feuille](S2)at(2,0){};
        \node[Feuille](S3)at(4,0){};
        \node[Operateur](N1)at(0,-2){\begin{math}\La\end{math}};
        \node[Operateur,Marque1](N2)at(3,-2){\begin{math}\Lb\end{math}};
        \node[Feuille](E1)at(0,-4){};
        \node[Feuille](E2)at(2,-4){};
        \node[Feuille](E3)at(4,-4){};
        \draw[Arete](N1)--(S1);
        \draw[Arete](N2)--(S2);
        \draw[Arete](N2)--(S3);
        \draw[Arete](N1)--(E1);
        \draw[Arete](N2)--(E2);
        \draw[Arete](N2)--(E3);
    \end{tikzpicture}}\end{split}}
    \enspace + \enspace
    \Delta \Sbf_{
    \begin{split}
    \scalebox{.25}{\begin{tikzpicture}[yscale=.8]
        \node[Feuille](S1)at(-1,0){};
        \node[Feuille](S2)at(1,0){};
        \node[Feuille](S3)at(3,0){};
        \node[Operateur,Marque1](N1)at(0,-2){\begin{math}\Lb\end{math}};
        \node[Operateur](N2)at(3,-2){\begin{math}\La\end{math}};
        \node[Feuille](E1)at(-1,-4){};
        \node[Feuille](E2)at(1,-4){};
        \node[Feuille](E3)at(3,-4){};
        \draw[Arete](N1)--(S1);
        \draw[Arete](N1)--(S2);
        \draw[Arete](N2)--(S3);
        \draw[Arete](N1)--(E1);
        \draw[Arete](N1)--(E2);
        \draw[Arete](N2)--(E3);
    \end{tikzpicture}}\end{split}} \\
    \enspace = \enspace
    \Tbf_{\left[\Unite_0\right]}
    \otimes
    \Tbf_{\left[
    \begin{split}
    \scalebox{.25}{\begin{tikzpicture}[yscale=.8]
        \node[Feuille](S1)at(0,0){};
        \node[Feuille](S2)at(2,0){};
        \node[Feuille](S3)at(4,0){};
        \node[Operateur](N1)at(0,-2){\begin{math}\La\end{math}};
        \node[Operateur](N2)at(2,-2){\begin{math}\La\end{math}};
        \node[Operateur](N3)at(4,-2){\begin{math}\La\end{math}};
        \node[Feuille](E1)at(0,-4){};
        \node[Feuille](E2)at(2,-4){};
        \node[Feuille](E3)at(4,-4){};
        \draw[Arete](N1)--(S1);
        \draw[Arete](N2)--(S2);
        \draw[Arete](N3)--(S3);
        \draw[Arete](N1)--(E1);
        \draw[Arete](N2)--(E2);
        \draw[Arete](N3)--(E3);
    \end{tikzpicture}}\end{split}\right]}
    \enspace + \enspace
    \Tbf_{\left[
    \begin{split}
    \scalebox{.25}{\begin{tikzpicture}[yscale=.8]
        \node[Feuille](S1)at(0,0){};
        \node[Feuille](S2)at(2,0){};
        \node[Feuille](S3)at(4,0){};
        \node[Operateur](N1)at(0,-2){\begin{math}\La\end{math}};
        \node[Operateur](N2)at(2,-2){\begin{math}\La\end{math}};
        \node[Operateur](N3)at(4,-2){\begin{math}\La\end{math}};
        \node[Feuille](E1)at(0,-4){};
        \node[Feuille](E2)at(2,-4){};
        \node[Feuille](E3)at(4,-4){};
        \draw[Arete](N1)--(S1);
        \draw[Arete](N2)--(S2);
        \draw[Arete](N3)--(S3);
        \draw[Arete](N1)--(E1);
        \draw[Arete](N2)--(E2);
        \draw[Arete](N3)--(E3);
    \end{tikzpicture}}\end{split}\right]}
    \otimes
    \Tbf_{\left[\Unite_0\right]} \\
    \enspace + \enspace
    2\,
    \Tbf_{\left[
    \begin{split}
    \scalebox{.25}{\begin{tikzpicture}[yscale=.8]
        \node[Feuille](S1)at(0,0){};
        \node[Operateur](N1)at(0,-2){\begin{math}\La\end{math}};
        \node[Feuille](E1)at(0,-4){};
        \draw[Arete](N1)--(S1);
        \draw[Arete](N1)--(E1);
    \end{tikzpicture}}\end{split}\right]}
    \otimes
    \Tbf_{\left[
    \begin{split}
    \scalebox{.25}{\begin{tikzpicture}[yscale=.8]
        \node[Feuille](S1)at(2,0){};
        \node[Feuille](S2)at(4,0){};
        \node[Operateur,Marque1](N1)at(3,-2){\begin{math}\Lb\end{math}};
        \node[Feuille](E1)at(2,-4){};
        \node[Feuille](E2)at(4,-4){};
        \draw[Arete](N1)--(S1);
        \draw[Arete](N1)--(S2);
        \draw[Arete](N1)--(E1);
        \draw[Arete](N1)--(E2);
    \end{tikzpicture}}\end{split}\right]}
    \enspace + \enspace
    2\,
    \Tbf_{\left[
    \begin{split}
    \scalebox{.25}{\begin{tikzpicture}[yscale=.8]
        \node[Feuille](S1)at(2,0){};
        \node[Feuille](S2)at(4,0){};
        \node[Operateur,Marque1](N1)at(3,-2){\begin{math}\Lb\end{math}};
        \node[Feuille](E1)at(2,-4){};
        \node[Feuille](E2)at(4,-4){};
        \draw[Arete](N1)--(S1);
        \draw[Arete](N1)--(S2);
        \draw[Arete](N1)--(E1);
        \draw[Arete](N1)--(E2);
    \end{tikzpicture}}\end{split}\right]}
    \otimes
    \Tbf_{\left[
    \begin{split}
    \scalebox{.25}{\begin{tikzpicture}[yscale=.8]
        \node[Feuille](S1)at(0,0){};
        \node[Operateur](N1)at(0,-2){\begin{math}\La\end{math}};
        \node[Feuille](E1)at(0,-4){};
        \draw[Arete](N1)--(S1);
        \draw[Arete](N1)--(E1);
    \end{tikzpicture}}\end{split}\right]} \\
    \enspace + \enspace
    \Sbf_{
    \begin{split}
    \scalebox{.25}{\begin{tikzpicture}[yscale=.8]
        \node[Feuille](S1)at(0,0){};
        \node[Operateur](N1)at(0,-2){\begin{math}\La\end{math}};
        \node[Feuille](E1)at(0,-4){};
        \draw[Arete](N1)--(S1);
        \draw[Arete](N1)--(E1);
    \end{tikzpicture}}\end{split}}
    \otimes
    \Sbf_{
    \begin{split}
    \scalebox{.25}{\begin{tikzpicture}[yscale=.8]
        \node[Feuille](S1)at(0,0){};
        \node[Feuille](S2)at(2,0){};
        \node[Operateur](N1)at(0,-2){\begin{math}\La\end{math}};
        \node[Operateur](N2)at(2,-2){\begin{math}\La\end{math}};
        \node[Feuille](E1)at(0,-4){};
        \node[Feuille](E2)at(2,-4){};
        \draw[Arete](N1)--(S1);
        \draw[Arete](N2)--(S2);
        \draw[Arete](N1)--(E1);
        \draw[Arete](N2)--(E2);
    \end{tikzpicture}}\end{split}}
    \enspace + \enspace
    \Sbf_{
    \begin{split}
    \scalebox{.25}{\begin{tikzpicture}[yscale=.8]
        \node[Feuille](S1)at(0,0){};
        \node[Feuille](S2)at(2,0){};
        \node[Operateur](N1)at(0,-2){\begin{math}\La\end{math}};
        \node[Operateur](N2)at(2,-2){\begin{math}\La\end{math}};
        \node[Feuille](E1)at(0,-4){};
        \node[Feuille](E2)at(2,-4){};
        \draw[Arete](N1)--(S1);
        \draw[Arete](N2)--(S2);
        \draw[Arete](N1)--(E1);
        \draw[Arete](N2)--(E2);
    \end{tikzpicture}}\end{split}}
    \otimes
    \Sbf_{
    \begin{split}
    \scalebox{.25}{\begin{tikzpicture}[yscale=.8]
        \node[Feuille](S1)at(0,0){};
        \node[Operateur](N1)at(0,-2){\begin{math}\La\end{math}};
        \node[Feuille](E1)at(0,-4){};
        \draw[Arete](N1)--(S1);
        \draw[Arete](N1)--(E1);
    \end{tikzpicture}}\end{split}}\,,
\end{multline}
showing that the coproduct is neither well-defined on the $\Tbf_{[x]}$.
\medskip

\subsubsection{Properties}
By using similar arguments as those used to establish Proposition
\ref{prop:PRO_vers_AHC_generation_liberte} together with the fact that
$\equiv$ satisfies \ref{item:propriete_bonnes_congruences_2} and the
product formula of Proposition \ref{prop:PRO_vers_AHC_congruence_produit},
we obtain that $\PvH(\Pca/_\equiv)$ is freely generated as an algebra by
the $\Tbf_{[x]}$ where the $[x]$ are $\equiv$-equivalence classes of
indecomposable and reduced elements of $\Pca$. Moreover, when $\omega$
is a grading of $\Pca$ so that all elements of a same $\equiv$-equivalence
class have the same degree, the bialgebra $\PvH(\Pca/_\equiv)$ is graded
by the grading inherited from the one of $\PvH(\Pca)$ and forms hence a
combinatorial Hopf algebra.
\medskip

\begin{Proposition} \label{prop:congruence_moins_fine_donne_sous_ahc}
    Let $\Pca$ be a free PRO and $\equiv_1$ and $\equiv_2$ be two
    stiff congruences of $\Pca$ such that $\equiv_1$ is finer than
    $\equiv_2$. Then, $\PvH\left(\Pca/_{\equiv_2}\right)$ is a
    sub-bialgebra of $\PvH\left(\Pca/_{\equiv_1}\right)$.
\end{Proposition}
\begin{proof}
    Since $\equiv_1$
    is finer that $\equiv_2$, any $\Tbf_{[x]_{\equiv_2}}$, where
    $[x]_{\equiv_2}$ is a $\equiv_2$-equivalence class of reduced
    elements of $\Pca$, is a sum of some $\Tbf_{[x']_{\equiv_1}}$.
    More precisely, we have
    \begin{equation}
        \Tbf_{[x]_{\equiv_2}} =
        \sum_{[x']_{\equiv_1} \subseteq [x]_{\equiv_2}}
        \Tbf_{[x']_{\equiv_1}},
    \end{equation}
    implying the result.
\end{proof}
\medskip

\subsection{Related constructions}
In this section, we first describe two constructions allowing to build
stiff PROs. The main interest of these constructions is that the obtained
stiff PROs can be placed at the input of the construction $\PvH$. We next
present a way to recover the natural Hopf algebra of an operad through
the construction $\PvH$ and the previous constructions of stiff PROs.
\medskip

\subsubsection{From operads to stiff PROs}
\label{subsubsec:operades_vers_PROs_rigides}
Any operad $\Oca$ gives naturally rise to a PRO $\OvP(\Oca)$ whose
elements are sequences of elements of $\Oca$ (see \cite{Mar08}).
\medskip

We recall here this construction. Let us set
$\OvP(\Oca) := \sqcup_{p \geq 0} \sqcup_{q \geq 0} \OvP(\Oca)(p, q)$
where
\begin{equation}
    \OvP(\Oca)(p, q) :=
    \{x_1 \dots x_q : x_i \in \Oca(p_i) \mbox{ for all }  i \in [q]
    \mbox{ and } p_1 + \dots + p_q = p\}.
\end{equation}
The horizontal composition of $\OvP(\Oca)$ is the concatenation of
sequences, and the vertical composition of $\OvP(\Oca)$ comes directly
from the composition map of $\Oca$. More precisely, for any
$x_1 \dots x_r \in \OvP(\Oca)(q, r)$ and
$y_{11} \dots y_{1q_1} \dots y_{r1} \dots y_{rq_r} \in \OvP(\Oca)(p, q)$,
we have
\begin{equation} \label{equ:definition_compo_v_construction_r}
    x_1 \dots x_r \circ
    y_{11} \dots y_{1q_1} \dots y_{r1} \dots y_{rq_r}
    := x_1 \circ [y_{11}, \dots, y_{1q_1}] \dots
        x_r \circ [y_{r1}, \dots, y_{rq_r}],
\end{equation}
where for any $i \in [r]$, $x_i \in \Oca(q_i)$ and the occurrences of
$\circ$ in the right-member of \eqref{equ:definition_compo_v_construction_r}
refer to the total composition map of $\Oca$.
\medskip

For instance, if $\Oca$ is the free operad generated by a generator of
arity $2$, $\Oca$ is an operad involving binary trees. Then, the
elements of the PRO $\OvP(\Oca)$ are forests of binary trees. The horizontal
composition of $\OvP(\Oca)$ is the concatenation of forests, and the
vertical composition $F_1 \circ F_2$ in $\OvP(\Oca)$, defined only between
two forests $F_1$ and $F_2$ such that the number of leaves of $F_1$ is
the same as the number of trees in $F_2$, consists in the forest obtained
by grafting, from left to right, the roots of the trees of $F_2$ on the
leaves of $F_1$.
\medskip

\begin{Proposition} \label{prop:operade_vers_bon_PRO}
    Let $\Oca$ be an operad such that the monoid $(\Oca(1), \circ_1)$
    does not contain any nontrivial subgroup. Then, $\OvP(\Oca)$ is a
    stiff PRO.
\end{Proposition}
\begin{proof}
    As any operad, $\Oca$ is the quotient by a certain operadic
    congruence $\equiv$ of the free operad generated by a certain set of
    generators $G$. It follows directly from the definition of the
    construction $\OvP$ that the PRO $\OvP(\Oca)$ is the quotient by the
    congruence of PROs $\equiv'$ of the free PRO generated by $G'$  where
    \begin{equation}
        G'(p, q) :=
        \begin{cases}
            G(p) & \mbox{if } p \geq 1 \mbox{ and } q = 1, \\
            \emptyset & \mbox{otherwise},
        \end{cases}
    \end{equation}
    and $\equiv'$ is the finest congruence of PROs satisfying $x \equiv' y$
    for any relation $x \equiv y$ between elements $x$ and $y$ of the
    free operad generated by $G$. Since by hypothesis $(\Oca(1), \circ_1)$
    does not contain any nontrivial subgroup, for all elements $x$ and
    $y$ of $\Oca(1) \setminus \{\Unite\}$, $x \circ_1 y \ne \Unite$. Then,
    $\equiv'$ satisfies \ref{item:propriete_bonnes_congruences_1}.
    Moreover, by  definition of $\OvP$, $\equiv'$ satisfies
    \ref{item:propriete_bonnes_congruences_2}. Hence, $\OvP(\Oca)$ is a
    stiff PRO.
\end{proof}
\medskip

\subsubsection{From monoids to stiff PROs}
\label{subsubsec:monoides_vers_PROs}
Any monoid $\Mca$ can be seen as an operad concentrated in arity one.
Then, starting from a monoid $\Mca$, one can construct a PRO $\MvP(\Mca)$
by applying the construction $\OvP$ on $\Mca$ seen as an operad.
\medskip

This construction can be rephrased as follows. We have
$\MvP(\Mca) = \sqcup_{p \geq 0} \sqcup_{q \geq 0} \MvP(\Mca)(p, q)$ where
\begin{equation}
    \MvP(\Mca)(p, q) =
    \begin{cases}
        \left\{x_1 \dots x_p : x_i \in M \mbox{ for all } i \in [p] \right\} &
            \mbox{if } p = q, \\
        \emptyset & \mbox{otherwise}. \\
    \end{cases}
\end{equation}
The horizontal composition of $\MvP(\Mca)$ is the concatenation of
sequences and the vertical composition
$\circ : \MvP(\Mca)(p, p) \times \MvP(\Mca)(p, p) \to \MvP(\Mca)(p, p)$
of $\MvP(\Mca)$ satisfies, for any $x_1 \dots x_p \in \MvP(\Mca)(p, p)$
and $y_1 \dots y_q \in \MvP(\Mca)(q, q)$,
\begin{equation}
    x_1 \dots x_p \circ y_1 \dots y_p =
    (x_1 \bullet y_1) \dots (x_p \bullet y_p),
\end{equation}
where $\bullet$ is the product of $\Mca$.
\medskip

For instance, if $\Mca$ is the additive monoid of natural numbers,
the PRO $\MvP(\Mca)$ contains all words over $\EnsNat$. The horizontal
composition of $\MvP(\Mca)$ is the concatenation of words, and the
vertical composition of $\MvP(\Mca)$, defined only on words with a same
length, is the componentwise addition of their letters.
\medskip

\begin{Proposition} \label{prop:monoide_vers_bon_PRO}
    Let $\Mca$ be a monoid that does not contain any nontrivial subgroup.
    Then, $\MvP(\Mca)$ is a stiff PRO.
\end{Proposition}
\begin{proof}
    Since $\Mca$ does not contain any nontrivial subgroup, seen as an
    operad, the elements of arity one of $\Mca$ do not contain any
    nontrivial subgroup. Hence, by definition of the construction $\MvP$
    passing by $\OvP$ and by Proposition \ref{prop:operade_vers_bon_PRO},
    $\MvP(\Mca)$ is a stiff PRO.
\end{proof}
\medskip

\subsubsection{The natural Hopf algebra of an operad}
We call {\em abelianization} of a bialgebra $\Hca$ the quotient of $\Hca$
by the bialgebra ideal spanned by the $x \cdot y - y \cdot x$ for all
$x, y \in \Hca$.
\medskip

Here is the link between our construction $\PvH$ and the construction
$\OvH$.
\medskip

\begin{Proposition} \label{prop:ahc_naturelle_par_h}
    Let $\Oca$ be an operad such that the monoid $(\Oca(1), \circ_1)$
    does not contain any nontrivial subgroup. Then, the bialgebra
    $\OvH(\Oca)$ is the abelianization of $\PvH(\OvP(\Oca))$.
\end{Proposition}
\begin{proof}
    By Proposition \ref{prop:operade_vers_bon_PRO}, $\OvP(\Oca)$ is a
    stiff PRO, and then, by Theorem \ref{thm:PRO_vers_AHC_congruence},
    $\PvH(\OvP(\Oca))$ is a bialgebra. By construction, this bialgebra
    is freely generated by the $\Tbf_x$ where $x \in \Oca$. Hence, the
    map $\phi : \PvH(\OvP(\Oca)) \to \OvH(\Oca)$ defined for any
    $x \in \Oca$ by $\phi(\Tbf_x) := \Tit_x$ can be uniquely extended
    into a bialgebra morphism, which we denote also by $\phi$. Since
    $\OvH(\Oca)$ is generated by the $\Tit_x$ where $x \in \Oca$, $\phi$
    is surjective. Directly from the definition of the construction $\OvH$,
    we observe that the kernel of $\phi$ is the bialgebra ideal $I$
    spanned by the $\Tbf_{x * y} -  \Tbf_{y * x}$ for all
    $x, y \in \OvP(\Oca)$. Then, the associated map
    $\phi_I : \PvH(\OvP(\Oca))/_I \to \OvH(\Oca)$ is a bialgebra
    isomorphism.
\end{proof}
\medskip

\section{Examples of application of the construction}
\label{sec:exemples}
We conclude this paper by presenting examples of application of the
construction $\PvH$. The PROs considered in this section fit into the
diagram represented by Figure \ref{fig:diagramme_PROs} and the obtained
Hopf algebras fit into the diagram represented by Figure \ref{fig:diagramme_AHC}.
\begin{figure}[ht]
    \centering
    \begin{tikzpicture}[scale=.55]
        \node(PRF)at(0,0){$\PRF_\gamma$};
        \node(FBT)at(3,2){$\FBT_\gamma$};
        \node(As)at(3,-2){$\As_\gamma$};
        \node(BAs)at(6,0){$\BAs_\gamma$};
        \node(Heap)at(10,2){$\Heap_\gamma$};
        \node(FHeap)at(10,-2){$\FHeap_\gamma$};
        \draw[Surjection](PRF)--(As);
        \draw[Surjection](FBT)--(BAs);
        \draw[Surjection](Heap)--(FHeap);
        \draw[Injection](PRF)--(FBT);
        \draw[Injection](As)--(BAs);
    \end{tikzpicture}
    \caption{Diagram of PROs where arrows $\rightarrowtail$
    (resp. $\twoheadrightarrow$) are injective (resp. surjective)
    PRO morphisms. The parameter $\gamma$ is a positive integer.
    When $\gamma = 0$, $\PRF_0=\As_0=\Heap_0=\FHeap_0$ and
    $\FBT_0=\BAs_0$.}
    \label{fig:diagramme_PROs}
\end{figure}
\begin{figure}[ht]
    \centering
    \begin{tikzpicture}[scale=.55]
        \node(PRF)at(0,0){$\PvH(\PRF_\gamma)$};
        \node(FBT)at(3,2){$\PvH(\FBT_\gamma)$};
        \node(As)at(3,-2){$\PvH(\As_\gamma)$};
        \node(BAs)at(6,0){$\PvH(\BAs_\gamma)$};
        \node(Heap)at(10,2){$\PvH(\Heap_\gamma)$};
        \node(FHeap)at(10,-2){$\PvH(\FHeap_\gamma)$};
        \draw[Injection](As)--(PRF);
        \draw[Injection](BAs)--(FBT);
        \draw[Injection](FHeap)--(Heap);
        \draw[Surjection](FBT)--(PRF);
        \draw[Surjection](BAs)--(As);
    \end{tikzpicture}
    \caption{Diagram of combinatorial Hopf algebras where arrows
    $\rightarrowtail$ (resp. $\twoheadrightarrow$) are injective
    (resp. surjective) PRO morphisms. The parameter $\gamma$ is
    a positive integer. When $\gamma = 0$, $\PvH(\PRF_0)=\PvH(\As_0)=\PvH(\Heap_0)=\PvH(\FHeap_0)$ and $\PvH(\FBT_0)=\PvH(\BAs_0)$.}
    \label{fig:diagramme_AHC}
\end{figure}
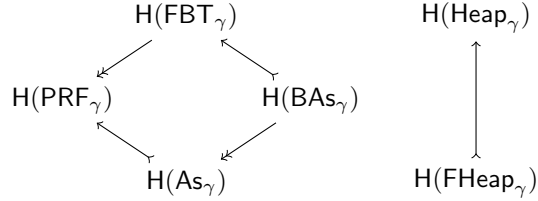
\medskip

\subsection{Hopf algebras of forests}
We present here the construction of two Hopf algebras of forests, one
depending on a nonnegative integer $\gamma$, and with different
gradings. The PRO we shall define in this section will intervene in
the next examples.
\medskip

\subsubsection{PRO of forests with a fixed arity}
Let $\gamma$ be a nonnegative integer and $\PRF_\gamma$ be the free PRO
generated by $G := G(\gamma + 1, 1) := \{\La\}$, with the grading
$\omega$ defined by $\omega(\La) := 1$. Any prograph $x$ of $\PRF_\gamma$
can be seen as a planar forest of planar rooted trees with only internal
nodes of arity $\gamma + 1$. Since the reduced elements of $\PRF_\gamma$
have no wire, they are encoded by forests of nonempty trees.
\medskip

\subsubsection{Hopf algebra}
By Theorem \ref{thm:PRO_vers_AHC_bigebre} and
Proposition \ref{prop:PRO_vers_AHC_graduation}, $\PvH(\PRF_\gamma)$ is a
combinatorial Hopf algebra. By
Proposition \ref{prop:PRO_vers_AHC_generation_liberte}, as an algebra,
$\PvH(\PRF_\gamma)$ is freely generated by the $\Sbf_T$, where the $T$ are
nonempty planar rooted trees with only internal nodes of arity
$\gamma + 1$. Its bases are indexed by planar forests of such trees
where the degree of a basis element $\Sbf_F$ is the number of internal
nodes of $F$.
\medskip

Notice that the bases of $\PvH(\PRF_0)$ are indexed by forests of linear
trees and that $\PvH(\PRF_0)$ and $\SymNC$ are trivially isomorphic as
combinatorial Hopf algebras.
\medskip

\subsubsection{Coproduct}
By definition of the construction $\PvH$, the coproduct of $\PvH(\PRF_\gamma)$
is given on a generator $\Sbf_T$ by
\begin{equation} \label{equ:coproduit_forets}
    \Delta(\Sbf_T) =
    \sum_{T' \in \Adm(T)} \Sbf_{T'} \otimes
    \Sbf_{T/_{T'}},
\end{equation}
where $\Adm(T)$ is the set of {\em admissible cuts} of $T$, that is, the
empty tree or the subtrees of $T$ containing the root of $T$ and where
$T/_{T'}$ denotes the forest consisting in the maximal subtrees of $T$
whose roots are leaves of $T'$, by respecting the order of these leaves
in $T'$ and by removing the empty trees. For instance, we have
\begin{multline}
    \Delta \Sbf_{
    \scalebox{.25}{\begin{tikzpicture}[yscale=.5,xscale=.9]
        \node[Feuille](0)at(0.00,-6.50){};
        \node[Feuille](10)at(6.00,-9.75){};
        \node[Feuille](11)at(7.00,-9.75){};
        \node[Feuille](12)at(8.00,-3.25){};
        \node[Feuille](2)at(1.00,-6.50){};
        \node[Feuille](3)at(2.00,-6.50){};
        \node[Feuille](5)at(3.00,-6.50){};
        \node[Feuille](7)at(4.00,-6.50){};
        \node[Feuille](8)at(5.00,-9.75){};
        \node[Noeud](1)at(1.00,-3.25){};
        \node[Noeud](4)at(5.00,0.00){};
        \node[Noeud](6)at(4.00,-3.25){};
        \node[Noeud](9)at(6.00,-6.50){};
        \draw[Arete](0)--(1);
        \draw[Arete](1)--(4);
        \draw[Arete](10)--(9);
        \draw[Arete](11)--(9);
        \draw[Arete](12)--(4);
        \draw[Arete](2)--(1);
        \draw[Arete](3)--(1);
        \draw[Arete](5)--(6);
        \draw[Arete](6)--(4);
        \draw[Arete](7)--(6);
        \draw[Arete](8)--(9);
        \draw[Arete](9)--(6);
    \end{tikzpicture}}}
    \enspace = \enspace
    \Sbf_{\emptyset} \otimes
    \Sbf_{
    \scalebox{.25}{\begin{tikzpicture}[yscale=.5,xscale=.9]
        \node[Feuille](0)at(0.00,-6.50){};
        \node[Feuille](10)at(6.00,-9.75){};
        \node[Feuille](11)at(7.00,-9.75){};
        \node[Feuille](12)at(8.00,-3.25){};
        \node[Feuille](2)at(1.00,-6.50){};
        \node[Feuille](3)at(2.00,-6.50){};
        \node[Feuille](5)at(3.00,-6.50){};
        \node[Feuille](7)at(4.00,-6.50){};
        \node[Feuille](8)at(5.00,-9.75){};
        \node[Noeud](1)at(1.00,-3.25){};
        \node[Noeud](4)at(5.00,0.00){};
        \node[Noeud](6)at(4.00,-3.25){};
        \node[Noeud](9)at(6.00,-6.50){};
        \draw[Arete](0)--(1);
        \draw[Arete](1)--(4);
        \draw[Arete](10)--(9);
        \draw[Arete](11)--(9);
        \draw[Arete](12)--(4);
        \draw[Arete](2)--(1);
        \draw[Arete](3)--(1);
        \draw[Arete](5)--(6);
        \draw[Arete](6)--(4);
        \draw[Arete](7)--(6);
        \draw[Arete](8)--(9);
        \draw[Arete](9)--(6);
    \end{tikzpicture}}}
    \enspace + \enspace
    \Sbf_{
    \scalebox{.25}{\begin{tikzpicture}[yscale=.7]
        \node[Feuille](0)at(0.00,-2.00){};
        \node[Feuille](2)at(1.00,-2.00){};
        \node[Feuille](3)at(2.00,-2.00){};
        \node[Noeud](1)at(1.00,0.00){};
        \draw[Arete](0)--(1);
        \draw[Arete](2)--(1);
        \draw[Arete](3)--(1);
    \end{tikzpicture}}}
    \otimes \Sbf_{
    \scalebox{.25}{\begin{tikzpicture}[yscale=.7]
        \node[Feuille](0)at(0.00,-2.00){};
        \node[Feuille](2)at(1.00,-2.00){};
        \node[Feuille](3)at(2.00,-2.00){};
        \node[Noeud](1)at(1.00,0.00){};
        \draw[Arete](0)--(1);
        \draw[Arete](2)--(1);
        \draw[Arete](3)--(1);
        \node[Feuille](b0)at(3.00,-2){};
        \node[Feuille](b2)at(4.00,-2){};
        \node[Feuille](b3)at(5.00,-4){};
        \node[Feuille](b5)at(6.00,-4){};
        \node[Feuille](b6)at(7.00,-4){};
        \node[Noeud](b1)at(4.00,0.00){};
        \node[Noeud](b4)at(6.00,-2){};
        \draw[Arete](b0)--(b1);
        \draw[Arete](b2)--(b1);
        \draw[Arete](b3)--(b4);
        \draw[Arete](b4)--(b1);
        \draw[Arete](b5)--(b4);
        \draw[Arete](b6)--(b4);
    \end{tikzpicture}}}
    \enspace + \enspace
    \Sbf_{
    \scalebox{.25}{\begin{tikzpicture}[yscale=.7]
        \node[Feuille](0)at(0.00,-4){};
        \node[Feuille](2)at(1.00,-4){};
        \node[Feuille](3)at(2.00,-4){};
        \node[Feuille](5)at(3.00,-2){};
        \node[Feuille](6)at(4.00,-2){};
        \node[Noeud](1)at(1.00,-2){};
        \node[Noeud](4)at(3.00,0.00){};
        \draw[Arete](0)--(1);
        \draw[Arete](1)--(4);
        \draw[Arete](2)--(1);
        \draw[Arete](3)--(1);
        \draw[Arete](5)--(4);
        \draw[Arete](6)--(4);
    \end{tikzpicture}}}
    \otimes
    \Sbf_{
    \scalebox{.25}{\begin{tikzpicture}[yscale=.7]
        \node[Feuille](b0)at(3.00,-2){};
        \node[Feuille](b2)at(4.00,-2){};
        \node[Feuille](b3)at(5.00,-4){};
        \node[Feuille](b5)at(6.00,-4){};
        \node[Feuille](b6)at(7.00,-4){};
        \node[Noeud](b1)at(4.00,0.00){};
        \node[Noeud](b4)at(6.00,-2){};
        \draw[Arete](b0)--(b1);
        \draw[Arete](b2)--(b1);
        \draw[Arete](b3)--(b4);
        \draw[Arete](b4)--(b1);
        \draw[Arete](b5)--(b4);
        \draw[Arete](b6)--(b4);
    \end{tikzpicture}}} \\
    \enspace + \enspace
    \Sbf_{
    \scalebox{.25}{\begin{tikzpicture}[yscale=.7]
        \node[Feuille](0)at(0.00,-2){};
        \node[Feuille](2)at(1.00,-4){};
        \node[Feuille](4)at(2.00,-4){};
        \node[Feuille](5)at(3.00,-4){};
        \node[Feuille](6)at(4.00,-2){};
        \node[Noeud](1)at(2.00,0.00){};
        \node[Noeud](3)at(2.00,-2){};
        \draw[Arete](0)--(1);
        \draw[Arete](2)--(3);
        \draw[Arete](3)--(1);
        \draw[Arete](4)--(3);
        \draw[Arete](5)--(3);
        \draw[Arete](6)--(1);
    \end{tikzpicture}}}
    \otimes \Sbf_{
    \scalebox{.25}{\begin{tikzpicture}[yscale=.7]
        \node[Feuille](0)at(0.00,-2.00){};
        \node[Feuille](2)at(1.00,-2.00){};
        \node[Feuille](3)at(2.00,-2.00){};
        \node[Noeud](1)at(1.00,0.00){};
        \draw[Arete](0)--(1);
        \draw[Arete](2)--(1);
        \draw[Arete](3)--(1);
        \node[Feuille](0)at(3.00,-2.00){};
        \node[Feuille](2)at(4.00,-2.00){};
        \node[Feuille](3)at(5.00,-2.00){};
        \node[Noeud](1)at(4.00,0.00){};
        \draw[Arete](0)--(1);
        \draw[Arete](2)--(1);
        \draw[Arete](3)--(1);
    \end{tikzpicture}}}
    \enspace + \enspace
    \Sbf_{
    \scalebox{.25}{\begin{tikzpicture}[yscale=.5]
        \node[Feuille](0)at(0.00,-6.67){};
        \node[Feuille](2)at(1.00,-6.67){};
        \node[Feuille](3)at(2.00,-6.67){};
        \node[Feuille](5)at(3.00,-6.67){};
        \node[Feuille](7)at(4.00,-6.67){};
        \node[Feuille](8)at(5.00,-6.67){};
        \node[Feuille](9)at(6.00,-3.33){};
        \node[Noeud](1)at(1.00,-3.33){};
        \node[Noeud](4)at(4.00,0.00){};
        \node[Noeud](6)at(4.00,-3.33){};
        \draw[Arete](0)--(1);
        \draw[Arete](1)--(4);
        \draw[Arete](2)--(1);
        \draw[Arete](3)--(1);
        \draw[Arete](5)--(6);
        \draw[Arete](6)--(4);
        \draw[Arete](7)--(6);
        \draw[Arete](8)--(6);
        \draw[Arete](9)--(4);
    \end{tikzpicture}}} \otimes
    \Sbf_{
    \scalebox{.25}{\begin{tikzpicture}[yscale=.7]
        \node[Feuille](0)at(0.00,-2.00){};
        \node[Feuille](2)at(1.00,-2.00){};
        \node[Feuille](3)at(2.00,-2.00){};
        \node[Noeud](1)at(1.00,0.00){};
        \draw[Arete](0)--(1);
        \draw[Arete](2)--(1);
        \draw[Arete](3)--(1);
    \end{tikzpicture}}}
    \enspace + \enspace
    \Sbf_{
    \scalebox{.25}{\begin{tikzpicture}[yscale=.5]
        \node[Feuille](0)at(0.00,-2.50){};
        \node[Feuille](2)at(1.00,-5.00){};
        \node[Feuille](4)at(2.00,-5.00){};
        \node[Feuille](5)at(3.00,-7.50){};
        \node[Feuille](7)at(4.00,-7.50){};
        \node[Feuille](8)at(5.00,-7.50){};
        \node[Feuille](9)at(6.00,-2.50){};
        \node[Noeud](1)at(3.00,0.00){};
        \node[Noeud](3)at(2.00,-2.50){};
        \node[Noeud](6)at(4.00,-5.00){};
        \draw[Arete](0)--(1);
        \draw[Arete](2)--(3);
        \draw[Arete](3)--(1);
        \draw[Arete](4)--(3);
        \draw[Arete](5)--(6);
        \draw[Arete](6)--(3);
        \draw[Arete](7)--(6);
        \draw[Arete](8)--(6);
        \draw[Arete](9)--(1);
    \end{tikzpicture}}}
    \otimes
    \Sbf_{
    \scalebox{.25}{\begin{tikzpicture}[yscale=.7]
        \node[Feuille](0)at(0.00,-2.00){};
        \node[Feuille](2)at(1.00,-2.00){};
        \node[Feuille](3)at(2.00,-2.00){};
        \node[Noeud](1)at(1.00,0.00){};
        \draw[Arete](0)--(1);
        \draw[Arete](2)--(1);
        \draw[Arete](3)--(1);
    \end{tikzpicture}}}
    \enspace + \enspace
    \Sbf_{
    \scalebox{.25}{\begin{tikzpicture}[yscale=.5,xscale=.9]
        \node[Feuille](0)at(0.00,-6.50){};
        \node[Feuille](10)at(6.00,-9.75){};
        \node[Feuille](11)at(7.00,-9.75){};
        \node[Feuille](12)at(8.00,-3.25){};
        \node[Feuille](2)at(1.00,-6.50){};
        \node[Feuille](3)at(2.00,-6.50){};
        \node[Feuille](5)at(3.00,-6.50){};
        \node[Feuille](7)at(4.00,-6.50){};
        \node[Feuille](8)at(5.00,-9.75){};
        \node[Noeud](1)at(1.00,-3.25){};
        \node[Noeud](4)at(5.00,0.00){};
        \node[Noeud](6)at(4.00,-3.25){};
        \node[Noeud](9)at(6.00,-6.50){};
        \draw[Arete](0)--(1);
        \draw[Arete](1)--(4);
        \draw[Arete](10)--(9);
        \draw[Arete](11)--(9);
        \draw[Arete](12)--(4);
        \draw[Arete](2)--(1);
        \draw[Arete](3)--(1);
        \draw[Arete](5)--(6);
        \draw[Arete](6)--(4);
        \draw[Arete](7)--(6);
        \draw[Arete](8)--(9);
        \draw[Arete](9)--(6);
    \end{tikzpicture}}}
    \otimes
    \Sbf_{\emptyset}.
\end{multline}
\medskip

This coproduct is similar to the one of the noncommutative Connes-Kreimer
Hopf algebra $\CK$ \cite{CK98}. The main difference between $\PvH(\PRF_\gamma)$
and $\CK$ lies in the fact that in a coproduct of $\CK$, the admissible
cuts can change the arity of some internal nodes; it is not the case in
$\PvH(\PRF_\gamma)$ because for any $T' \in \Adm(T)$, any internal node
$x$ of $T'$ has the same arity as it has in~$T$.
\medskip

\subsubsection{Dimensions}
The series of the algebraic generators of
$\PvH(\PRF_\gamma)$ is
\begin{equation}
    G(t) :=
    \sum_{n \geq 1} \frac{1}{n\gamma+1}\binom{n(\gamma + 1)}{n} t^n
\end{equation}
since its coefficients are the Fuss-Catalan numbers, counting planar
rooted trees with $n$ internal nodes of arity $\gamma + 1$. Since
$\PvH(\PRF_\gamma)$ is free as an algebra, its Hilbert series is
$H(t) := \frac{1}{1 - G(t)}$.
\medskip

The first dimensions of $\PvH(\PRF_1)$ are
\begin{equation} \label{equ:dim_AHC_PRF_1}
    1, 1, 3, 10, 35, 126, 462, 1716, 6435, 24310, 92378,
\end{equation}
and those of $\PvH(\PRF_2)$ are
\begin{equation}
    1, 4, 19, 98, 531, 2974, 17060, 99658, 590563, 3540464, 21430267.
\end{equation}
The first sequence is listed in \cite{Slo} as Sequence \Sloane{A001700}
and the second as Sequence \Sloane{A047099}.
\medskip

\subsubsection{PRO of general forests}
We denote by $\PRF_\infty$ the free PRO generated by
$G := \sqcup_{n \geq 1} G(n, 1) := \sqcup_{n \geq 1} \{\La_n\}$. Any
prograph $x$ of $\PRF_\infty$ can be seen as a planar forest of planar rooted
trees. Since the reduced elements of $\PRF_\infty$ have no wire, they are
encoded by forests of nonempty trees. Observe that for any nonnegative
integer $\gamma$, $\PRF_\gamma$ is a sub-PRO of $\PRF_\infty$.
\medskip

\subsubsection{Hopf algebra}
By Theorem \ref{thm:PRO_vers_AHC_bigebre}, $\PvH(\PRF_\infty)$ is a
bialgebra. By Proposition \ref{prop:PRO_vers_AHC_generation_liberte}, as
an algebra, $\PvH(\PRF_\infty)$ is freely generated by the $\Sbf_T$,
where the $T$ are nonempty planar rooted trees. Its bases are indexed by
planar forests of such trees. Besides, by
Proposition \ref{prop:sous_generateurs_libre_donne_quotient},
since $\PRF_\gamma$ is generated by a subset of the generators of
$\PRF_\infty$, $\PvH(\PRF_\gamma)$ is a bialgebra quotient of
$\PvH(\PRF_\infty)$. Moreover, the coproduct of $\PvH(\PRF_\infty)$
satisfies \eqref{equ:coproduit_forets}.
\medskip

To turn $\PvH(\PRF_\infty)$ into a combinatorial Hopf algebra, we cannot
consider the grading $\omega$ defined by $\omega(\La_n) := 1$ because
there would be infinitely many elements of degree $1$. Therefore, we
consider on $\PvH(\PRF_\infty)$ the grading $\omega$ defined by
$\omega(\La_n) := n$. In this way, the degree of a basis element
$\Sbf_F$ is the number of edges of the forest $F$. By
Proposition \ref{prop:PRO_vers_AHC_graduation}, $\PvH(\PRF_\infty)$
is a combinatorial Hopf algebra.
\medskip

\subsubsection{Dimensions}
The series of the algebraic generators of $\PvH(\PRF_\infty)$ is
\begin{equation}
    G(t) :=
    \sum_{n \geq 1} \frac{1}{n+1}\binom{2n}{n} t^n
\end{equation}
since its coefficients are the Catalan numbers, counting planar rooted
trees with $n$ edges. As $\PvH(\PRF_\infty)$ is free as an algebra,
its Hilbert series is
\begin{equation}
    H(t) := \frac{1}{1 - G(t)} =
    1 + \sum_{n \geq 1} \frac{1}{2}\binom{2n}{n} t^n.
\end{equation}
The dimensions of $\PvH(\PRF_\infty)$ are then the same as the
dimensions of $\PvH(\PRF_1)$ (see \eqref{equ:dim_AHC_PRF_1}).
\medskip

\subsection{The Faà di Bruno algebra and its deformations}
We shall give here a method to construct the Hopf algebras $\FdBNC_\gamma$
of Foissy \cite{Foi08} from our construction $\PvH$ in the case where
$\gamma$ is a nonnegative integer.
\medskip

\subsubsection{Associative PRO}
Let $\gamma$ be a nonnegative integer and $\As_\gamma$ be the quotient
of $\PRF_\gamma$ by the finest congruence $\equiv$ satisfying
\begin{equation}
    \begin{split}\scalebox{.25}{\begin{tikzpicture}
        \node[Feuille](S1)at(0,0){};
        \node[Operateur](N1)at(0,-2){\begin{math}\La\end{math}};
        \node[Operateur](N2)at(0,-5){\begin{math}\La\end{math}};
        \node[Feuille](E1)at(-5,-7){};
        \node[Feuille](E2)at(-3,-7){};
        \node[Feuille](E3)at(-1,-7){};
        \node[Feuille](E4)at(1,-7){};
        \node[Feuille](E5)at(3,-7){};
        \node[Feuille](E6)at(5,-7){};
        \draw[Arete](N1)--(S1);
        \draw[Arete](N1)--(N2);
        \draw[Arete](N1)--(E1);
        \draw[Arete](N1)--(E2);
        \draw[Arete](N2)--(E3);
        \draw[Arete](N2)--(E4);
        \draw[Arete](N1)--(E5);
        \draw[Arete](N1)--(E6);
        \node[right of=E3,node distance=1cm]
            {\scalebox{2}{\begin{math}\Huge \dots\end{math}}};
        \node(k1)[left of=E2,node distance=1cm]
            {\scalebox{2}{\begin{math}\Huge \dots\end{math}}};
        \node[below of=k1,node distance=.75cm]
            {\scalebox{3}{\begin{math}\Huge k_1\end{math}}};
        \node(k2)[right of=E5,node distance=1cm]
            {\scalebox{2}{\begin{math}\Huge \dots\end{math}}};
        \node[below of=k2,node distance=.75cm]
            {\scalebox{3}{\begin{math}\Huge k_2\end{math}}};
    \end{tikzpicture}}\end{split}
    \quad \equiv \quad
    \begin{split}\scalebox{.25}{\begin{tikzpicture}
        \node[Feuille](S1)at(0,0){};
        \node[Operateur](N1)at(0,-2){\begin{math}\La\end{math}};
        \node[Operateur](N2)at(0,-5){\begin{math}\La\end{math}};
        \node[Feuille](E1)at(-5,-7){};
        \node[Feuille](E2)at(-3,-7){};
        \node[Feuille](E3)at(-1,-7){};
        \node[Feuille](E4)at(1,-7){};
        \node[Feuille](E5)at(3,-7){};
        \node[Feuille](E6)at(5,-7){};
        \draw[Arete](N1)--(S1);
        \draw[Arete](N1)--(N2);
        \draw[Arete](N1)--(E1);
        \draw[Arete](N1)--(E2);
        \draw[Arete](N2)--(E3);
        \draw[Arete](N2)--(E4);
        \draw[Arete](N1)--(E5);
        \draw[Arete](N1)--(E6);
        \node[right of=E3,node distance=1cm]
            {\scalebox{2}{\begin{math}\Huge \dots\end{math}}};
        \node(k1)[left of=E2,node distance=1cm]
            {\scalebox{2}{\begin{math}\Huge \dots\end{math}}};
        \node[below of=k1,node distance=.75cm]
            {\scalebox{3}{\begin{math}\Huge \ell_1\end{math}}};
        \node(k2)[right of=E5,node distance=1cm]
            {\scalebox{2}{\begin{math}\Huge \dots\end{math}}};
        \node[below of=k2,node distance=.75cm]
            {\scalebox{3}{\begin{math}\Huge \ell_2\end{math}}};
    \end{tikzpicture}}\end{split}\,,
    \qquad k_1 + k_2 = \gamma, \ell_1 + \ell_2 = \gamma.
\end{equation}
\medskip

We can observe that $\As_\gamma$ is a stiff PRO because $\equiv$
satisfies \ref{item:propriete_bonnes_congruences_1} and
\ref{item:propriete_bonnes_congruences_2} and that $\As_0 = \PRF_0$.
Moreover, observe that, when $\gamma \geq 1$, there is in $\As_\gamma$
exactly one indecomposable element of arity $n\gamma + 1$ for any
$n \geq 0$. We denote by $\alpha_n$ this element. We consider on
$\As_\gamma$ the grading $\omega$ inherited from the one of
$\PRT_\gamma$. This grading is still well-defined in $\As_\gamma$
since any $\equiv$-equivalence class contains prographs of a same
degree and satisfies, for all $n \geq 0$, $\omega(\alpha_n) = n$.
Any element of $\As_\gamma$ is then a word
$\alpha_{k_1} \dots \alpha_{k_\ell}$ and can be encoded by a word of
nonnegative integers $k_1 \dots k_{\ell}$. Since the reduced elements
of $\As_\gamma$ have no wire, they are encoded by words of positive
integers.
\medskip

\subsubsection{Hopf algebra}
By Theorem \ref{thm:PRO_vers_AHC_congruence} and
Proposition \ref{prop:PRO_vers_AHC_graduation}, $\PvH(\As_\gamma)$
is a combinatorial Hopf algebra. As an algebra, $\PvH(\As_\gamma)$
is freely generated by the $\Tbf_n$, $n \geq 1$, and its bases are
indexed by words of positive integers where the degree of a basis
element $\Tbf_{k_1 \dots k_\ell}$ is $k_1 + \dots + k_\ell$.
\medskip

\subsubsection{Coproduct}
Since any element $\alpha_n$ of $\As_\gamma$ decomposes into
$\alpha_n = x \circ y$ if and only if $x = \alpha_k$ and
$y = \alpha_{i_1} \dots \alpha_{i_{k\gamma + 1}}$ with
$i_1 + \dots + i_{k\gamma + 1} = n - k$, by
Proposition \ref{prop:PRO_vers_AHC_congruence_coproduit}, for
any $n \geq 1$, the coproduct of $\PvH(\As_\gamma)$ expresses as
\begin{equation} \label{equ:coproduit_analogue_FdBNC_gamma}
    \Delta(\Tbf_n) =
        \sum_{k = 0}^n \Tbf_k
        \otimes
        \left(\sum_{i_1+\dots+ i_{k\gamma+1} = n - k}
        \Tbf_{i_1}\dots \Tbf_{i_{k\gamma+1}}\right),
\end{equation}
where $\Tbf_0$ is identified with the unity $\Tbf_\epsilon$ of $\PvH(\As_\gamma)$.
For instance, in $\PvH(\As_1)$, we have
\begin{multline}
    \Delta(\Tbf_3) =
        \Tbf_0 \otimes \Tbf_3 +
        \Tbf_1 \otimes (\Tbf_0\Tbf_2 + \Tbf_1\Tbf_1 + \Tbf_2\Tbf_0) \\
    + \Tbf_2 \otimes
        (\Tbf_0\Tbf_0\Tbf_1+\Tbf_0\Tbf_1\Tbf_0+\Tbf_1\Tbf_0\Tbf_0)
    + \Tbf_3 \otimes (\Tbf_0\Tbf_0\Tbf_0\Tbf_0) \\
    = \Tbf_\epsilon \otimes \Tbf_3 + 2\,\Tbf_1\otimes \Tbf_2 +
    \Tbf_1 \otimes \Tbf_{11} + 3\,\Tbf_2\otimes \Tbf_1 +
    \Tbf_3 \otimes \Tbf_\epsilon,
\end{multline}
and in $\PvH(\As_2)$, we have
\begin{multline}
    \Delta(\Tbf_3) =
    \Tbf_0 \otimes \Tbf_3
    + \Tbf_1 \otimes (\Tbf_0\Tbf_0\Tbf_2 + \Tbf_0\Tbf_2\Tbf_0
        + \Tbf_2\Tbf_0\Tbf_0 + \Tbf_0\Tbf_1\Tbf_1
        + \Tbf_1\Tbf_0\Tbf_1 + \Tbf_1\Tbf_1\Tbf_0) \\
    + \Tbf_2 \otimes (\Tbf_0\Tbf_0\Tbf_0\Tbf_0\Tbf_1
        + \Tbf_0\Tbf_0\Tbf_0\Tbf_1\Tbf_0
        + \Tbf_0\Tbf_0\Tbf_1\Tbf_0\Tbf_0
        + \Tbf_0\Tbf_1\Tbf_0\Tbf_0\Tbf_0
        + \Tbf_1\Tbf_0\Tbf_0\Tbf_0\Tbf_0) \\
    + \Tbf_3 \otimes \Tbf_0\Tbf_0\Tbf_0\Tbf_0\Tbf_0\Tbf_0\Tbf_0 \\
    =
    \Tbf_\epsilon \otimes \Tbf_3 + 3\,\Tbf_1 \otimes \Tbf_2
    + 3\,\Tbf_1 \otimes \Tbf_{11} + 5\,\Tbf_2 \otimes \Tbf_1
    + \Tbf_3 \otimes \Tbf_\epsilon.
\end{multline}
\medskip

\subsubsection{Isomorphism with the deformation of the noncommutative
    Faà di Bruno Hopf algebra}

\begin{Theoreme} \label{thm:construction_FdBNC_gamma}
    For any nonnegative integer $\gamma$, the Hopf algebra
    $\PvH(\As_\gamma)$ is the deformation of the noncommutative
    Faà di Bruno Hopf algebra $\FdBNC_\gamma$.
\end{Theoreme}
\begin{proof}
    Let us set $\sigma_1 := \sum_{n\geq 0} \Tbf_n$. By
    \eqref{equ:coproduit_analogue_FdBNC_gamma}, we have
    \begin{align}
        \Delta(\sigma_1) & =
            \sum_{n\geq 0}\sum_{k=0}^n \Tbf_k \otimes
            \left(\sum_{i_1+\dots+ i_{k\gamma+1} = n-k}
            \Tbf_{i_1}\dots \Tbf_{i_{k\gamma+1}}\right) \\
        & = \sum_{k\geq 0}\Tbf_k\otimes \left(\sum_{n\geq k} \enspace
            \sum_{i_1+\dots+ i_{k\gamma+1} = n-k}
            \Tbf_{i_1}\dots \Tbf_{i_{k\gamma+1}}\right)\\
        & = \sum_{k\geq 0}\Tbf_k\otimes
        \left(\sum_{i_1+\dots+ i_{k\gamma+1}\geq 0}
        \Tbf_{i_1}\dots \Tbf_{i_{k\gamma+1}}\right)\\
        & = \sum_{k\geq 0}
            \Tbf_k\otimes
            \left(\sum_{i\geq 0}\Tbf_i\right)^{k\gamma+1}\\
        & = \sum_{k\geq 0} \Tbf_k\otimes \sigma_1^{k\gamma+1},
    \end{align}
    showing that for any $n \geq 0$, $\Delta(\Tbf_n)$ is the
    homogeneous component of degree $n$ in
    $\sum_{k\geq 0} \Tbf_k\otimes \sigma_1^{k\gamma+1}$.
    \smallskip

    Therefore, by the definition of the coproduct of $\FdBNC_\gamma$
    (see \eqref{equ:coproduit_delta_gamma_non_commutatif}),
    the map $\Tbf_n\mapsto \Sbf_n$ from $\PvH(\As_\gamma)$ to
    $\FdBNC_\gamma$ is an isomorphism of Hopf algebras.
\end{proof}
\medskip

\subsection{Hopf algebra of forests of bitrees}
Let us describe here a general construction on PROs. If $G$ is a bigraded
set of the form $G = \sqcup_{p \geq 1} \sqcup_{q \geq 1} G(p, q)$, we
denote by $G^\Op$ the bigraded set defined by
\begin{equation}
    G^\Op(p, q) := G(q, p), \qquad p, q \geq 1.
\end{equation}
From a geometrical point of view, any elementary prograph over $G^\Op$
is obtained by reversing from bottom to top an elementary prograph over
$G$. We moreover denote by $\Rev : \Free(G^\Op) \to \Free(G)$ the bijection
sending any prograph $x$ of $\Free(G^\Op)$ to the prograph $\Rev(x)$
of $\Free(G)$ obtained by reversing $x$ from bottom to top.
\medskip

Now, given a PRO $\Pca := \Free(G)/_\equiv$, we define $\Sat(\Pca)$ as
the PRO
\begin{equation}
    \Sat(\Pca) := \Free\left(G \sqcup G^\Op\right)/_{\cong},
\end{equation}
where $\cong$ is the finest congruence of
$\Free(G \sqcup G^\Op)$ satisfying
\begin{equation}
    x \cong y
    \quad \mbox{if }
    (x, y \in \Free(G) \mbox{ and } x \equiv y)
    \enspace \mbox{ or } \enspace
    (x, y \in \Free(G^\Op) \mbox{ and } \Rev(x) \equiv \Rev(y)).
\end{equation}
Notice that in this definition, we consider $\Free(G)$ and $\Free(G^\Op)$
as sub-PROs of $\Free(G\sqcup G^\Op)$ in an obvious way.
Notice also that if $\Pca$ is a free PRO $\Free(G)$, then the congruence
$\equiv$ is trivial, so that $\cong$ is also trivial, and
$\Sat(\Pca) = \Free(G\sqcup G^\Op)$. Besides, as an other immediate
property of this construction, remark that when $\Pca$ is a stiff PRO,
the congruence $\cong$ satisfies~\ref{item:propriete_bonnes_congruences_1}
and~\ref{item:propriete_bonnes_congruences_2}, and then, $\Sat(\Pca)$
is a stiff PRO.
\medskip

We shall present here two Hopf algebras coming from the construction
$\Sat$ applied on $\PRF_\gamma$ and~$\As_\gamma$.
\medskip

\subsubsection{PRO of forests of bitrees}
Let $\gamma$ be a nonnegative integer and $\FBT_\gamma$ be the free PRO
generated by
$G := G(\gamma+1, 1) \sqcup G(1, \gamma+1)$ where $G(\gamma+1, 1) := \{\La\}$
and $G(1, \gamma+1) := \{\Lb\}$, with the grading $\omega$ defined by
$\omega(\La) := \omega(\Lb) := 1$. One has $\Sat(\PRF_\gamma) = \FBT_\gamma$.
Any prograph $x$ of $\FBT_\gamma$ can be seen as a forest of
{\em $\gamma$-bitrees}, that are labeled planar trees where internal
nodes labeled by $\La$ have $\gamma+1$ children and one parent, and the
internal nodes labeled by $\Lb$ have one child and $\gamma+1$ parents.
Since the reduced elements of $\FBT_\gamma$ have no wire, they are encoded
by forests of nonempty $\gamma$-bitrees.
\medskip

\subsubsection{Hopf algebra}
By Theorem \ref{thm:PRO_vers_AHC_bigebre} and
Proposition \ref{prop:PRO_vers_AHC_graduation}, $\PvH(\FBT_\gamma)$ is
a combinatorial Hopf algebra. By
Proposition \ref{prop:PRO_vers_AHC_generation_liberte}, as an algebra,
$\PvH(\FBT_\gamma)$ is freely generated by the $\Sbf_T$, where the $T$
are nonempty $\gamma$-bitrees. Its bases are indexed by planar forests
of such bitrees where the degree of a basis element $\Sbf_F$ is the
total number of internal nodes in the bitrees of $F$. Moreover, by
Proposition \ref{prop:sous_generateurs_libre_donne_quotient}, since
$\PRF_\gamma$ is generated by a subset of the generators of $\FBT_\gamma$,
$\PvH(\PRF_\gamma)$ is a quotient bialgebra of $\PvH(\FBT_\gamma)$.
\medskip

\subsubsection{Coproduct}
The coproduct of $\PvH(\FBT_\gamma)$ can be described, like the one
of $\CK$ on forests, by means of admissible cuts on forests of
$\gamma$-bitrees. We have for instance
\begin{multline}
    \Delta \Sbf_{
    \scalebox{.25}{\begin{tikzpicture}[yscale=.7]
        \node[Feuille](S1)at(0,1){};
        \node[Feuille](E1)at(-2,-7){};
        \node[Feuille](E2)at(1.5,-7){};
        \node[Feuille](E3)at(2.5,-7){};
        \node[Noeud](N1)at(0,-1){};
        \node[Noeud](N2)at(1,-3){};
        \node[Noeud,Marque1](N3)at(-2,-5){};
        \node[Noeud](N4)at(2,-5){};
        \draw[Arete](N1)--(S1);
        \draw[Arete](N3)--(E1);
        \draw[Arete](N4)--(E2);
        \draw[Arete](N4)--(E3);
        \draw[Arete](N1)--(N3);
        \draw[Arete](N1)--(N2);
        \draw[Arete](N2)--(N3);
        \draw[Arete](N2)--(N4);
    \end{tikzpicture}}}
    \enspace = \enspace
    \Sbf_{\emptyset} \otimes
    \Sbf_{
    \scalebox{.25}{\begin{tikzpicture}[yscale=.7]
        \node[Feuille](S1)at(0,1){};
        \node[Feuille](E1)at(-2,-7){};
        \node[Feuille](E2)at(1.5,-7){};
        \node[Feuille](E3)at(2.5,-7){};
        \node[Noeud](N1)at(0,-1){};
        \node[Noeud](N2)at(1,-3){};
        \node[Noeud,Marque1](N3)at(-2,-5){};
        \node[Noeud](N4)at(2,-5){};
        \draw[Arete](N1)--(S1);
        \draw[Arete](N3)--(E1);
        \draw[Arete](N4)--(E2);
        \draw[Arete](N4)--(E3);
        \draw[Arete](N1)--(N3);
        \draw[Arete](N1)--(N2);
        \draw[Arete](N2)--(N3);
        \draw[Arete](N2)--(N4);
    \end{tikzpicture}}}
    \enspace + \enspace
    \Sbf_{
    \scalebox{.25}{\begin{tikzpicture}[yscale=.7]
        \node[Feuille](S1)at(0,1){};
        \node[Feuille](E1)at(-.5,-3){};
        \node[Feuille](E2)at(.5,-3){};
        \node[Noeud](N1)at(0,-1){};
        \draw[Arete](N1)--(S1);
        \draw[Arete](N1)--(E1);
        \draw[Arete](N1)--(E2);
    \end{tikzpicture}}}
    \otimes \Sbf_{
    \scalebox{.25}{\begin{tikzpicture}[yscale=.7]
        \node[Feuille](S1)at(-2,-3){};
        \node[Feuille](S2)at(1,-1){};
        \node[Feuille](E1)at(-1,-7){};
        \node[Feuille](E2)at(1.5,-7){};
        \node[Feuille](E3)at(2.5,-7){};
        \node[Noeud](N2)at(1,-3){};
        \node[Noeud,Marque1](N3)at(-1,-5){};
        \node[Noeud](N4)at(2,-5){};
        \draw[Arete](N3)--(S1);
        \draw[Arete](N2)--(S2);
        \draw[Arete](N3)--(E1);
        \draw[Arete](N4)--(E2);
        \draw[Arete](N4)--(E3);
        \draw[Arete](N2)--(N3);
        \draw[Arete](N2)--(N4);
    \end{tikzpicture}}}
    \enspace + \enspace
    \Sbf_{
    \scalebox{.25}{\begin{tikzpicture}[yscale=.7]
        \node[Feuille](S2)at(1,-1){};
        \node[Feuille](E1)at(0,-5){};
        \node[Feuille](E2)at(1.5,-7){};
        \node[Feuille](E3)at(2.5,-7){};
        \node[Noeud](N2)at(1,-3){};
        \node[Noeud](N4)at(2,-5){};
        \draw[Arete](N2)--(S2);
        \draw[Arete](N2)--(E1);
        \draw[Arete](N4)--(E2);
        \draw[Arete](N4)--(E3);
        \draw[Arete](N2)--(N4);
    \end{tikzpicture}}}
    \otimes
    \Sbf_{
    \scalebox{.25}{\begin{tikzpicture}[yscale=.7]
        \node[Feuille](SS1)at(.5,4){};
        \node[Feuille](SS2)at(-.5,4){};
        \node[Feuille](EE1)at(0,0){};
        \node[Noeud,Marque1](NN1)at(0,2){};
        \draw[Arete](NN1)--(SS1);
        \draw[Arete](NN1)--(SS2);
        \draw[Arete](NN1)--(EE1);
        \node[Feuille](S1)at(2,4){};
        \node[Feuille](E1)at(1.5,0){};
        \node[Feuille](E2)at(2.5,0){};
        \node[Noeud](N1)at(2,2){};
        \draw[Arete](N1)--(S1);
        \draw[Arete](N1)--(E1);
        \draw[Arete](N1)--(E2);
    \end{tikzpicture}}} \\
    \enspace + \enspace
    \Sbf_{
    \scalebox{.25}{\begin{tikzpicture}[yscale=.7]
        \node[Feuille](S1)at(0,1){};
        \node[Feuille](E1)at(-2,-7){};
        \node[Feuille](E2)at(1.5,-5){};
        \node[Noeud](N1)at(0,-1){};
        \node[Noeud](N2)at(1,-3){};
        \node[Noeud,Marque1](N3)at(-2,-5){};
        \draw[Arete](N1)--(S1);
        \draw[Arete](N3)--(E1);
        \draw[Arete](N2)--(E2);
        \draw[Arete](N1)--(N3);
        \draw[Arete](N1)--(N2);
        \draw[Arete](N2)--(N3);
    \end{tikzpicture}}}
    \otimes \Sbf_{
    \scalebox{.25}{\begin{tikzpicture}[yscale=.7]
        \node[Feuille](S1)at(0,1){};
        \node[Feuille](E1)at(-.5,-3){};
        \node[Feuille](E2)at(.5,-3){};
        \node[Noeud](N1)at(0,-1){};
        \draw[Arete](N1)--(S1);
        \draw[Arete](N1)--(E1);
        \draw[Arete](N1)--(E2);
    \end{tikzpicture}}}
    \enspace + \enspace
    \Sbf_{
    \scalebox{.25}{\begin{tikzpicture}[yscale=.7]
        \node[Feuille](S1)at(0,1){};
        \node[Feuille](E1)at(-1,-3){};
        \node[Feuille](E2)at(0,-5){};
        \node[Feuille](E3)at(1.5,-7){};
        \node[Feuille](E4)at(2.5,-7){};
        \node[Noeud](N1)at(0,-1){};
        \node[Noeud](N2)at(1,-3){};
        \node[Noeud](N4)at(2,-5){};
        \draw[Arete](N1)--(S1);
        \draw[Arete](N1)--(E1);
        \draw[Arete](N2)--(E2);
        \draw[Arete](N4)--(E3);
        \draw[Arete](N4)--(E4);
        \draw[Arete](N1)--(N2);
        \draw[Arete](N2)--(N4);
    \end{tikzpicture}}}
    \otimes \Sbf_{
    \scalebox{.25}{\begin{tikzpicture}[yscale=.7]
        \node[Feuille](SS1)at(.5,4){};
        \node[Feuille](SS2)at(-.5,4){};
        \node[Feuille](EE1)at(0,0){};
        \node[Noeud,Marque1](NN1)at(0,2){};
        \draw[Arete](NN1)--(SS1);
        \draw[Arete](NN1)--(SS2);
        \draw[Arete](NN1)--(EE1);
    \end{tikzpicture}}}
    \enspace + \enspace
    \Sbf_{
    \scalebox{.25}{\begin{tikzpicture}[yscale=.7]
        \node[Feuille](S1)at(0,1){};
        \node[Feuille](E1)at(-2,-7){};
        \node[Feuille](E2)at(1.5,-7){};
        \node[Feuille](E3)at(2.5,-7){};
        \node[Noeud](N1)at(0,-1){};
        \node[Noeud](N2)at(1,-3){};
        \node[Noeud,Marque1](N3)at(-2,-5){};
        \node[Noeud](N4)at(2,-5){};
        \draw[Arete](N1)--(S1);
        \draw[Arete](N3)--(E1);
        \draw[Arete](N4)--(E2);
        \draw[Arete](N4)--(E3);
        \draw[Arete](N1)--(N3);
        \draw[Arete](N1)--(N2);
        \draw[Arete](N2)--(N3);
        \draw[Arete](N2)--(N4);
    \end{tikzpicture}}} \otimes \Sbf_{\emptyset}.
\end{multline}

\subsubsection{Dimensions}
We only know the dimensions of $\PvH(\FBT_\gamma)$ when $\gamma = 0$.
In this case, $0$-bitrees of size $n$ are linear trees and can hence be seen
as words of length $n$ on the alphabet $\{\La, \Lb\}$. Therefore, as
$\PvH(\FBT_\gamma)$ is free as an algebra, the bases of $\PvH(\FBT_0)$
are indexed by multiwords on $\{\La, \Lb\}$ and its Hilbert series is
\begin{equation}
    H(t) := 1+\sum_{n\geq 1}2^{2n-1}t^n.
\end{equation}
\medskip

\subsubsection{PRO of biassociative operators and its Hopf algebra}
Let $\BAs_\gamma$ be the quotient of $\FBT_\gamma$ by the finest congruence
$\equiv$ satisfying
\begin{equation}
    \begin{split}\scalebox{.25}{\begin{tikzpicture}
        \node[Feuille](S1)at(0,0){};
        \node[Operateur](N1)at(0,-2){\begin{math}\La\end{math}};
        \node[Operateur](N2)at(0,-5){\begin{math}\La\end{math}};
        \node[Feuille](E1)at(-5,-7){};
        \node[Feuille](E2)at(-3,-7){};
        \node[Feuille](E3)at(-1,-7){};
        \node[Feuille](E4)at(1,-7){};
        \node[Feuille](E5)at(3,-7){};
        \node[Feuille](E6)at(5,-7){};
        \draw[Arete](N1)--(S1);
        \draw[Arete](N1)--(N2);
        \draw[Arete](N1)--(E1);
        \draw[Arete](N1)--(E2);
        \draw[Arete](N2)--(E3);
        \draw[Arete](N2)--(E4);
        \draw[Arete](N1)--(E5);
        \draw[Arete](N1)--(E6);
        \node[right of=E3,node distance=1cm]
            {\scalebox{2}{\begin{math}\Huge \dots\end{math}}};
        \node(k1)[left of=E2,node distance=1cm]
            {\scalebox{2}{\begin{math}\Huge \dots\end{math}}};
        \node[below of=k1,node distance=.75cm]
            {\scalebox{3}{\begin{math}\Huge k_1\end{math}}};
        \node(k2)[right of=E5,node distance=1cm]
            {\scalebox{2}{\begin{math}\Huge \dots\end{math}}};
        \node[below of=k2,node distance=.75cm]
            {\scalebox{3}{\begin{math}\Huge k_2\end{math}}};
    \end{tikzpicture}}\end{split}
    \quad \equiv \quad
    \begin{split}\scalebox{.25}{\begin{tikzpicture}
        \node[Feuille](S1)at(0,0){};
        \node[Operateur](N1)at(0,-2){\begin{math}\La\end{math}};
        \node[Operateur](N2)at(0,-5){\begin{math}\La\end{math}};
        \node[Feuille](E1)at(-5,-7){};
        \node[Feuille](E2)at(-3,-7){};
        \node[Feuille](E3)at(-1,-7){};
        \node[Feuille](E4)at(1,-7){};
        \node[Feuille](E5)at(3,-7){};
        \node[Feuille](E6)at(5,-7){};
        \draw[Arete](N1)--(S1);
        \draw[Arete](N1)--(N2);
        \draw[Arete](N1)--(E1);
        \draw[Arete](N1)--(E2);
        \draw[Arete](N2)--(E3);
        \draw[Arete](N2)--(E4);
        \draw[Arete](N1)--(E5);
        \draw[Arete](N1)--(E6);
        \node[right of=E3,node distance=1cm]
            {\scalebox{2}{\begin{math}\Huge \dots\end{math}}};
        \node(k1)[left of=E2,node distance=1cm]
            {\scalebox{2}{\begin{math}\Huge \dots\end{math}}};
        \node[below of=k1,node distance=.75cm]
            {\scalebox{3}{\begin{math}\Huge \ell_1\end{math}}};
        \node(k2)[right of=E5,node distance=1cm]
            {\scalebox{2}{\begin{math}\Huge \dots\end{math}}};
        \node[below of=k2,node distance=.75cm]
            {\scalebox{3}{\begin{math}\Huge \ell_2\end{math}}};
    \end{tikzpicture}}\end{split}\,,
    \qquad k_1 + k_2 = \gamma, \ell_1 + \ell_2 = \gamma,
\end{equation}
and
\begin{equation}
    \begin{split}\scalebox{.25}{\begin{tikzpicture}
        \node[Feuille](S1)at(0,0){};
        \node[Operateur,Marque1](N1)at(0,2){\begin{math}\Lb\end{math}};
        \node[Operateur,Marque1](N2)at(0,5){\begin{math}\Lb\end{math}};
        \node[Feuille](E1)at(-5,7){};
        \node[Feuille](E2)at(-3,7){};
        \node[Feuille](E3)at(-1,7){};
        \node[Feuille](E4)at(1,7){};
        \node[Feuille](E5)at(3,7){};
        \node[Feuille](E6)at(5,7){};
        \draw[Arete](N1)--(S1);
        \draw[Arete](N1)--(N2);
        \draw[Arete](N1)--(E1);
        \draw[Arete](N1)--(E2);
        \draw[Arete](N2)--(E3);
        \draw[Arete](N2)--(E4);
        \draw[Arete](N1)--(E5);
        \draw[Arete](N1)--(E6);
        \node[right of=E3,node distance=1cm]
            {\scalebox{2}{\begin{math}\Huge \dots\end{math}}};
        \node(k1)[left of=E2,node distance=1cm]
            {\scalebox{2}{\begin{math}\Huge \dots\end{math}}};
        \node[above of=k1,node distance=.75cm]
            {\scalebox{3}{\begin{math}\Huge k_1\end{math}}};
        \node(k2)[right of=E5,node distance=1cm]
            {\scalebox{2}{\begin{math}\Huge \dots\end{math}}};
        \node[above of=k2,node distance=.75cm]
            {\scalebox{3}{\begin{math}\Huge k_2\end{math}}};
    \end{tikzpicture}}\end{split}
    \quad \equiv \quad
    \begin{split}\scalebox{.25}{\begin{tikzpicture}
        \node[Feuille](S1)at(0,0){};
        \node[Operateur,Marque1](N1)at(0,2){\begin{math}\Lb\end{math}};
        \node[Operateur,Marque1](N2)at(0,5){\begin{math}\Lb\end{math}};
        \node[Feuille](E1)at(-5,7){};
        \node[Feuille](E2)at(-3,7){};
        \node[Feuille](E3)at(-1,7){};
        \node[Feuille](E4)at(1,7){};
        \node[Feuille](E5)at(3,7){};
        \node[Feuille](E6)at(5,7){};
        \draw[Arete](N1)--(S1);
        \draw[Arete](N1)--(N2);
        \draw[Arete](N1)--(E1);
        \draw[Arete](N1)--(E2);
        \draw[Arete](N2)--(E3);
        \draw[Arete](N2)--(E4);
        \draw[Arete](N1)--(E5);
        \draw[Arete](N1)--(E6);
        \node[right of=E3,node distance=1cm]
            {\scalebox{2}{\begin{math}\Huge \dots\end{math}}};
        \node(k1)[left of=E2,node distance=1cm]
            {\scalebox{2}{\begin{math}\Huge \dots\end{math}}};
        \node[above of=k1,node distance=.75cm]
            {\scalebox{3}{\begin{math} \Huge \ell_1\end{math}}};
        \node(k2)[right of=E5,node distance=1cm]
            {\scalebox{2}{\begin{math}\Huge \dots\end{math}}};
        \node[above of=k2,node distance=.75cm]
            {\scalebox{3}{\begin{math}\Huge \ell_2\end{math}}};
    \end{tikzpicture}}\end{split}\,,
    \qquad k_1 + k_2 = \gamma, \ell_1 + \ell_2 = \gamma.
\end{equation}
\medskip

We can observe that $\BAs_\gamma$ is a stiff PRO because $\equiv$
satisfies \ref{item:propriete_bonnes_congruences_1} and
\ref{item:propriete_bonnes_congruences_2}. Notice that
$\Sat(\As_\gamma) = \BAs_\gamma$ and $\BAs_0 = \FBT_0$.
We consider on $\BAs_\gamma$ the grading $\omega$ inherited from
the one of $\FBT_\gamma$. This grading is still well-defined in
$\BAs_\gamma$ since any $\equiv$-equivalence class contains prographs
of a same degree. Notice that $\BAs_1$ is very similar to the PRO
governing bialgebras (see \cite{Mar08}). Indeed, it only lacks in
$\BAs_1$ the usual compatibility relation between its two generators.
Notice also that the PRO governing bialgebras is not a stiff PRO.
\medskip

By Theorem \ref{thm:PRO_vers_AHC_congruence} and
Proposition \ref{prop:PRO_vers_AHC_graduation}, $\PvH(\BAs_\gamma)$
is then a combinatorial Hopf algebra. Moreover, we can observe that
$\PvH(\As_\gamma)$ is a quotient bialgebra of $\PvH(\BAs_\gamma)$.
\medskip

\subsection{Hopf algebra of heaps of pieces}
We present here the construction of a Hopf algebra depending on
a nonnegative integer $\gamma$, whose bases are indexed by heaps
of pieces.
\medskip

\subsubsection{PRO of heaps of pieces}
Let $\gamma$ be a nonnegative integer and $\Heap_\gamma$ be the free PRO
generated by $G := G(\gamma + 1, \gamma + 1) := \{\La\}$, with the
grading $\omega$ defined by $\omega(\La) := 1$. Any prograph $x$ of
$\Heap_\gamma$ can be seen as a heap of pieces of width $\gamma + 1$
(see \cite{Vie86} for some theory about these objects). For instance,
the prograph
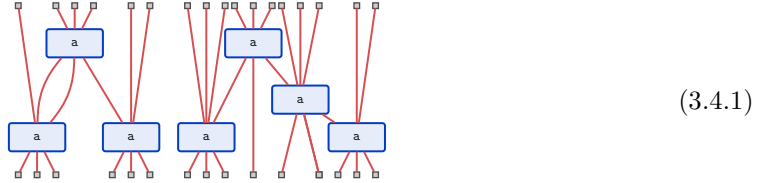
\begin{equation}
    \begin{split}\scalebox{.25}{\begin{tikzpicture}
        \node[Feuille](S1)at(0,0){};
        \node[Feuille](S2)at(2,0){};
        \node[Feuille](S3)at(3,0){};
        \node[Feuille](S4)at(4,0){};
        \node[Feuille](S5)at(6,0){};
        \node[Feuille](S6)at(7,0){};
        \node[Feuille](S7)at(9,0){};
        \node[Feuille](S8)at(10,0){};
        \node[Feuille](S9)at(11,0){};
        \node[Feuille](S10)at(14,0){};
        \node[Feuille](S11)at(15,0){};
        \node[Feuille](S12)at(16,0){};
        \node[Feuille](S13)at(18,0){};
        \node[Feuille](S14)at(19,0){};
        \node[Feuille](S10a)at(11.5,0){};
        \node[Feuille](S10b)at(12.5,0){};
        \node[Feuille](S10c)at(13.5,0){};
        \node[Operateur,minimum width=3cm](N1)at(1,-7){\begin{math}\La\end{math}};
        \node[Operateur,minimum width=3cm](N2)at(3,-2){\begin{math}\La\end{math}};
        \node[Operateur,minimum width=3cm](N3)at(6,-7){\begin{math}\La\end{math}};
        \node[Operateur,minimum width=3cm](N4)at(10,-7){\begin{math}\La\end{math}};
        \node[Operateur,minimum width=3cm](N5)at(15,-5){\begin{math}\La\end{math}};
        \node[Operateur,minimum width=3cm](N6)at(18,-7){\begin{math}\La\end{math}};
        \node[Operateur,minimum width=3cm](N7)at(12.5,-2){\begin{math}\La\end{math}};
        \node[Feuille](E1)at(0,-9){};
        \node[Feuille](E2)at(1,-9){};
        \node[Feuille](E3)at(2,-9){};
        \node[Feuille](E4)at(5,-9){};
        \node[Feuille](E5)at(6,-9){};
        \node[Feuille](E6)at(7,-9){};
        \node[Feuille](E7)at(9,-9){};
        \node[Feuille](E8)at(10,-9){};
        \node[Feuille](E9)at(11,-9){};
        \node[Feuille](E10)at(14,-9){};
        \node[Feuille](E11)at(16,-9){};
        \node[Feuille](E12)at(16,-9){};
        \node[Feuille](E13)at(17,-9){};
        \node[Feuille](E14)at(18,-9){};
        \node[Feuille](E15)at(19,-9){};
        \node[Feuille](E10a)at(12.5,-9){};
        \draw[Arete](N1)--(S1);
        \draw[Arete](N2)--(S2);
        \draw[Arete](N2)--(S3);
        \draw[Arete](N2)--(S4);
        \draw[Arete](N3)--(S5);
        \draw[Arete](N3)--(S6);
        \draw[Arete](N4)--(S7);
        \draw[Arete](N4)--(S8);
        \draw[Arete](N4)--(S9);
        \draw[Arete](N5)--(S10);
        \draw[Arete](N5)--(S11);
        \draw[Arete](N5)--(S12);
        \draw[Arete](N6)--(S13);
        \draw[Arete](N6)--(S14);
        \draw[Arete](N7)--(S10a);
        \draw[Arete](N7)--(S10b);
        \draw[Arete](N7)--(S10c);
        \draw[Arete](N1)--(E1);
        \draw[Arete](N1)--(E2);
        \draw[Arete](N1)--(E3);
        \draw[Arete](N3)--(E4);
        \draw[Arete](N3)--(E5);
        \draw[Arete](N3)--(E6);
        \draw[Arete](N4)--(E7);
        \draw[Arete](N4)--(E8);
        \draw[Arete](N4)--(E9);
        \draw[Arete](N5)--(E10);
        \draw[Arete](N5)--(E11);
        \draw[Arete](N5)--(E12);
        \draw[Arete](N6)--(E13);
        \draw[Arete](N6)--(E14);
        \draw[Arete](N6)--(E15);
        \draw[Arete](N7)--(E10a);
        \draw[Arete](N2)--(N3);
        \draw[Arete](N5)--(N6);
        \draw[Arete](N7)--(N4);
        \draw[Arete](N7)--(N5);
        \draw[Arete](N1)edge[bend left=20] node[]{}(N2);
        \draw[Arete](N1)edge[bend right=20] node[]{}(N2);
    \end{tikzpicture}}\end{split}
\end{equation}
of $\Heap_2$ is encoded by the heap of pieces of width $3$ depicted by
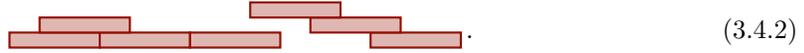
\begin{equation}
    \begin{split}\scalebox{.4}{\begin{tikzpicture}
        \node[Domino3](0)at(0,0){};
        \node[Domino3](1)at(3,0){};
        \node[Domino3](2)at(1,.5){};
        \node[Domino3](3)at(6,0){};
        \node[Domino3](4)at(12,0){};
        \node[Domino3](5)at(10,.5){};
        \node[Domino3](6)at(8,1){};
    \end{tikzpicture}}\end{split}\,.
\end{equation}
Notice that $\Heap_0 = \PRF_0$. Besides, since the reduced elements of
$\Heap_\gamma$ have no wire, they are encoded by horizontally connected
heaps of pieces of width $\gamma + 1$.
\medskip

\subsubsection{Hopf algebra}
By Theorem \ref{thm:PRO_vers_AHC_bigebre} and
Proposition \ref{prop:PRO_vers_AHC_graduation}, $\PvH(\Heap_\gamma)$
is a combinatorial Hopf algebra. By
Proposition \ref{prop:PRO_vers_AHC_generation_liberte}, as an algebra,
$\PvH(\Heap_\gamma)$ is freely generated by the $\Sbf_\lambda$ where
the $\lambda$ are heaps of pieces that cannot be obtained by juxtaposing
two heaps of pieces. Its bases are indexed by horizontally connected
heaps of pieces of width $\gamma + 1$ where the degree of a basis
element $\Sbf_\Lambda$ is the number of pieces of $\Lambda$.
\medskip

\subsubsection{Coproduct}
The coproduct of $\PvH(\Heap_\gamma)$ can be described, like the one of
$\CK$ on forests, by means of admissible cuts on heaps of pieces. Indeed,
if $\Lambda$ is a horizontally connected heap of pieces, by definition
of the construction $\PvH$,
\begin{equation}
    \Delta(\Sbf_\Lambda) =
    \sum_{\Lambda' \in \Adm(\Lambda)}
    \Sbf_{\Lambda'} \otimes \Sbf_{\Lambda/_{\Lambda'}},
\end{equation}
where $\Adm(\Lambda)$ is the set of {\em admissible cuts} of $\Lambda$,
that is, the set of heaps of pieces obtained by keeping an upper part
of $\Lambda$ and by readjusting it so that it becomes horizontally
connected and where $\Lambda/_{\Lambda'}$ denotes the heap of pieces
obtained by removing from $\Lambda$ the pieces of $\Lambda'$ and by
readjusting the remaining pieces so that they form an horizontally
connected heap of pieces. For instance, in $\PvH(\Heap_1)$, we have
\begin{multline}
    \Delta
    \Sbf_{\scalebox{.4}{\begin{tikzpicture}
        \node[Domino2](0)at(0,0){};
        \node[Domino2](1)at(2,0){};
        \node[Domino2](2)at(1,-.5){};
        \node[Domino2](3)at(2,-1){};
    \end{tikzpicture}}}
    \enspace = \enspace
    \Sbf_{\emptyset} \otimes
    \Sbf_{\scalebox{.4}{\begin{tikzpicture}
        \node[Domino2](0)at(0,0){};
        \node[Domino2](1)at(2,0){};
        \node[Domino2](2)at(1,-.5){};
        \node[Domino2](3)at(2,-1){};
    \end{tikzpicture}}}
    \enspace + \enspace
    \Sbf_{\scalebox{.4}{\begin{tikzpicture}
        \node[Domino2](0)at(0,0){};
    \end{tikzpicture}}}
    \otimes
    \Sbf_{\scalebox{.4}{\begin{tikzpicture}
        \node[Domino2](1)at(2,0){};
        \node[Domino2](2)at(1,-.5){};
        \node[Domino2](3)at(2,-1){};
    \end{tikzpicture}}}
    \enspace + \enspace
    \Sbf_{\scalebox{.4}{\begin{tikzpicture}
        \node[Domino2](0)at(0,0){};
    \end{tikzpicture}}}
    \otimes
    \Sbf_{\scalebox{.4}{\begin{tikzpicture}
        \node[Domino2](0)at(0,0){};
        \node[Domino2](2)at(1,-.5){};
        \node[Domino2](3)at(2,-1){};
    \end{tikzpicture}}} \\
    \enspace + \enspace
    \Sbf_{\scalebox{.4}{\begin{tikzpicture}
        \node[Domino2](0)at(0,0){};
        \node[Domino2](1)at(2,0){};
    \end{tikzpicture}}}
    \otimes
    \Sbf_{\scalebox{.4}{\begin{tikzpicture}
        \node[Domino2](2)at(1,-.5){};
        \node[Domino2](3)at(2,-1){};
    \end{tikzpicture}}}
    \enspace + \enspace
    \Sbf_{\scalebox{.4}{\begin{tikzpicture}
        \node[Domino2](0)at(0,0){};
        \node[Domino2](1)at(2,0){};
        \node[Domino2](2)at(1,-.5){};
    \end{tikzpicture}}}
    \otimes
    \Sbf_{\scalebox{.4}{\begin{tikzpicture}
        \node[Domino2](3)at(2,-1){};
    \end{tikzpicture}}}
    \enspace + \enspace
    \Sbf_{\scalebox{.4}{\begin{tikzpicture}
        \node[Domino2](0)at(0,0){};
        \node[Domino2](1)at(2,0){};
        \node[Domino2](2)at(1,-.5){};
        \node[Domino2](3)at(2,-1){};
    \end{tikzpicture}}}
    \otimes
    \Sbf_{\emptyset}\,.
\end{multline}
\medskip

\subsubsection{Dimensions}
\begin{Proposition} \label{prop:dimensions_PvH_Heap_gamma}
    For any nonnegative integer $\gamma$, the Hilbert series
    $C_{\gamma}(t)$ of $\PvH(\Heap_\gamma)$ satisfies
    $C_{\gamma}(t) = \sum_{n \geq 0} C_{\gamma, n}(t)$, where
    \begin{equation}
        C_{\gamma, n}(t) := P_{\gamma, n}(t)
        - \sum_{k = 0}^{n - 1}
            C_{\gamma, k}(t) P_{\gamma, n - k - 1}(t),
    \end{equation}
    \begin{equation}
        P_{\gamma, n}(t) := \frac{1}{F_{\gamma, n}(t)},
    \end{equation}
    and
    \begin{equation}
        F_{\gamma, n}(t) :=
        \begin{cases}
            1 & \mbox{if } n \leq \gamma, \\
            F_{\gamma, n - 1}(t) - t F_{\gamma, n - \gamma - 1}(t)
                & \mbox{otherwise}.
        \end{cases}
    \end{equation}
\end{Proposition}
\begin{proof}
    The proof of this statement uses the {\em Inversion Lemma} of
    Viennot \cite{Vie86} and some ideas employed in \cite{BMR02} for
    the enumeration of the so-called {\em connected heaps}.
    \smallskip

    The first ingredient consists in the alternating generating series
    \begin{equation}
        \sum_{x = \Unite_{p_1} * \La * \Unite_{p_2} * \dots * \Unite_{p_\ell} * \La * \Unite_{p_{\ell + 1}} \in \Heap_\gamma}
        (-1)^{\deg(x)} t^{\deg(x)}
    \end{equation}
    of the heaps of pieces of $\Heap_\gamma$ of height no greater than
    $1$ (called {\em trivial heaps} in \cite{Vie86} and \cite{BMR02}).
    This series is obviously $F_{\gamma, n}(t)$. By the Inversion Lemma,
    we have that $P_{\gamma, n}(t)$ is the generating series of the
    elements of $\Heap_\gamma$ with exactly $n$ inputs (and thus, also
    $n$ outputs).
    \smallskip

    Now, to count only the reduced elements of $\Heap_\gamma$, observe
    that any element $x$ of $\Heap_\gamma$ with $n$ inputs is either
    reduced or is of the form $x = y * \Unite_1 * z$ where $y$ is a
    reduced element of $\Heap_\gamma$ with $k$ inputs and $z$ is an
    element of $\Heap_\gamma$ with $n - k - 1$ inputs. Then, we have
    \begin{equation}
        P_{\gamma, n}(t) =
        C_{\gamma, n}(t)
        + \sum_{k = 0}^{n - 1} C_{\gamma, k}(t) P_{\gamma, n - k - 1}(t),
    \end{equation}
    so that $C_{\gamma, n}(t)$ counts the reduced elements of $\Heap_\gamma$
    with $n$ inputs. Whence the result.
\end{proof}
\medskip

By using Proposition \ref{prop:dimensions_PvH_Heap_gamma}, one can
compute the first dimensions of $\PvH(\Heap_\gamma)$. The first
dimensions of $\PvH(\Heap_1)$ are
\begin{equation} \label{equ:dim_Heap_1}
    1, 1, 4, 18, 85, 411, 2014, 9950, 49417, 246302, 1230623,
\end{equation}
and those of $\PvH(\Heap_2)$ are
\begin{equation} \label{equ:dim_Heap_2}
    1, 1, 6, 42, 313, 2407, 18848, 149271, 1191092, 9553551, 76910632.
\end{equation}
Since by Proposition \ref{prop:PRO_vers_AHC_generation_liberte},
$\PvH(\Heap_\gamma)$ is free as an algebra, the series $G_\gamma(t)$ of
its algebraic generators satisfies $G_\gamma(t) = 1 - \frac{1}{C_\gamma(t)}$.
The first dimensions of the algebraic generators of $\PvH(\Heap_1)$ are
\begin{equation} \label{equ:dim_generateurs_Heap_1}
    1, 3, 11, 44, 184, 790, 3450, 15242, 67895, 304267, 1369761,
\end{equation}
and those of $\PvH(\Heap_2)$ are
\begin{equation}
    1, 5, 31, 210, 1488, 10826, 80111, 599671, 4525573, 34357725,
    262011295.
\end{equation}
Among these four integer sequences, only \eqref{equ:dim_generateurs_Heap_1}
is listed in \cite{Slo} as Sequence \Sloane{A059715}.
\medskip

\subsection{Hopf algebra of heaps of friable pieces}
From a PRO being a special quotient of $\Heap_\gamma$, we construct
a Hopf algebra structure on the $(\gamma+1)$-st tensor power of the vector
space $\SymNC$.
\medskip

\subsubsection{PRO of heaps of friable pieces}
Let $\gamma$ be a nonnegative integer and $\FHeap_\gamma$ be the quotient
of $\Heap_\gamma$ by the finest congruence $\equiv$ satisfying
\begin{equation}
\begin{split}\scalebox{.25}{\begin{tikzpicture}[yscale=1.1]
    \node[Feuille](S1)at(0,0){};
    \node[Feuille](S2)at(2,0){};
    \node[Feuille](S3)at(4,0){};
    \node[Feuille](S4)at(6,0){};
    \node[Operateur,minimum width=3cm](N1)at(1,-2){\begin{math}\La\end{math}};
    \node[Operateur,minimum width=3cm](N2)at(3,-5){\begin{math}\La\end{math}};
    \node[Feuille](E1)at(-2,-7){};
    \node[Feuille](E2)at(0,-7){};
    \node[Feuille](E3)at(2,-7){};
    \node[Feuille](E4)at(4,-7){};
    \draw[Arete](N1)--(S1);
    \draw[Arete](N1)--(S2);
    \draw[Arete](N1)--(E1);
    \draw[Arete](N1)--(E2);
    \draw[Arete](N2)--(S3);
    \draw[Arete](N2)--(S4);
    \draw[Arete](N2)--(E3);
    \draw[Arete](N2)--(E4);
    \node[above of=N1,node distance=2cm]
        {\scalebox{2}{\begin{math}\Huge \dots\end{math}}};
    \node[below of=N2,node distance=2cm]
        {\scalebox{2}{\begin{math}\Huge \dots\end{math}}};
    \node[right of=S3,node distance=1cm]
        {\scalebox{2}{\begin{math}\Huge \dots\end{math}}};
    \node[left of=E2,node distance=1cm]
        {\scalebox{2}{\begin{math}\Huge \dots\end{math}}};
    \draw[Arete](N1)edge[bend left=27] node[]{}(N2);
    \draw[Arete](N1)edge[bend right=27] node[]{}(N2);
    \node[]at(2.0,-3.5)
        {\scalebox{2}{\begin{math}\Huge \dots\end{math}}};
    \node[]at(2,-3.9){\scalebox{3}{\begin{math}\Huge \ell\end{math}}};
\end{tikzpicture}}\end{split}
\quad \equiv \quad
\begin{split}\scalebox{.25}{\begin{tikzpicture}[yscale=1.1]
    \node[Feuille](S1)at(6,0){};
    \node[Feuille](S2)at(4,0){};
    \node[Feuille](S3)at(2,0){};
    \node[Feuille](S4)at(0,0){};
    \node[Operateur,minimum width=3cm](N1)at(5,-2){\begin{math}\La\end{math}};
    \node[Operateur,minimum width=3cm](N2)at(3,-5){\begin{math}\La\end{math}};
    \node[Feuille](E1)at(8,-7){};
    \node[Feuille](E2)at(6,-7){};
    \node[Feuille](E3)at(4,-7){};
    \node[Feuille](E4)at(2,-7){};
    \draw[Arete](N1)--(S1);
    \draw[Arete](N1)--(S2);
    \draw[Arete](N1)--(E1);
    \draw[Arete](N1)--(E2);
    \draw[Arete](N2)--(S3);
    \draw[Arete](N2)--(S4);
    \draw[Arete](N2)--(E3);
    \draw[Arete](N2)--(E4);
    \node[above of=N1,node distance=2cm]
        {\scalebox{2}{\begin{math}\Huge \dots\end{math}}};
    \node[below of=N2,node distance=2cm]
        {\scalebox{2}{\begin{math}\Huge \dots\end{math}}};
    \node[right of=S4,node distance=1cm]
        {\scalebox{2}{\begin{math}\Huge \dots\end{math}}};
    \node[right of=E2,node distance=1cm]
        {\scalebox{2}{\begin{math}\Huge \dots\end{math}}};
    \draw[Arete](N1)edge[bend left=27] node[]{}(N2);
    \draw[Arete](N1)edge[bend right=27] node[]{}(N2);
    \node[]at(4.0,-3.5)
        {\scalebox{2}{\begin{math}\Huge \dots\end{math}}};
    \node[]at(4,-3.9){\scalebox{3}{\begin{math}\Huge \ell\end{math}}};
\end{tikzpicture}}\end{split}\,,
\qquad \ell \in [\gamma].
\end{equation}
For instance, for $\gamma = 2$, the $\equiv$-equivalence class of
\begin{equation}
\begin{split}\scalebox{.25}{\begin{tikzpicture}
    \node[Feuille](S1)at(0,0){};
    \node[Feuille](S2)at(1,0){};
    \node[Feuille](S3)at(2,0){};
    \node[Feuille](S4)at(5,0){};
    \node[Feuille](S5)at(6,0){};
    \node[Feuille](S6)at(7,0){};
    \node[Operateur,minimum width=3cm](N1)at(1,-2){\begin{math}\La\end{math}};
    \node[Operateur,minimum width=3cm](N2)at(3,-5){\begin{math}\La\end{math}};
    \node[Operateur,minimum width=3cm](N3)at(6,-2){\begin{math}\La\end{math}};
    \node[Feuille](E1)at(0,-7){};
    \node[Feuille](E2)at(2,-7){};
    \node[Feuille](E3)at(3,-7){};
    \node[Feuille](E4)at(4,-7){};
    \node[Feuille](E5)at(6,-7){};
    \node[Feuille](E6)at(7,-7){};
    \draw[Arete](N1)--(S1);
    \draw[Arete](N1)--(S2);
    \draw[Arete](N1)--(S3);
    \draw[Arete](N3)--(S4);
    \draw[Arete](N3)--(S5);
    \draw[Arete](N3)--(S6);
    \draw[Arete](N1)--(E1);
    \draw[Arete](N2)--(E2);
    \draw[Arete](N2)--(E3);
    \draw[Arete](N2)--(E4);
    \draw[Arete](N3)--(E5);
    \draw[Arete](N3)--(E6);
    \draw[Arete](N2)--(N3);
    \draw[Arete](N1)edge[bend left=20] node[]{}(N2);
    \draw[Arete](N1)edge[bend right=20] node[]{}(N2);
\end{tikzpicture}}\end{split}
\end{equation}
contains exactly the prographs
\begin{equation} \label{equ:exemple_classe_FHeap_3}
    \begin{split}\scalebox{.25}{\begin{tikzpicture}
        \node[Feuille](S1)at(0,0){};
        \node[Feuille](S2)at(2,0){};
        \node[Feuille](S3)at(3,0){};
        \node[Feuille](S4)at(5,0){};
        \node[Feuille](S5)at(6,0){};
        \node[Feuille](S6)at(7,0){};
        \node[Operateur,minimum width=3cm](N1)at(1,-8){\begin{math}\La\end{math}};
        \node[Operateur,minimum width=3cm](N2)at(3,-5){\begin{math}\La\end{math}};
        \node[Operateur,minimum width=3cm](N3)at(6,-2){\begin{math}\La\end{math}};
        \node[Feuille](E1)at(0,-10){};
        \node[Feuille](E2)at(1,-10){};
        \node[Feuille](E3)at(2,-10){};
        \node[Feuille](E4)at(4,-10){};
        \node[Feuille](E5)at(6,-10){};
        \node[Feuille](E6)at(7,-10){};
        \draw[Arete](N1)--(S1);
        \draw[Arete](N2)--(S2);
        \draw[Arete](N2)--(S3);
        \draw[Arete](N3)--(S4);
        \draw[Arete](N3)--(S5);
        \draw[Arete](N3)--(S6);
        \draw[Arete](N1)--(E1);
        \draw[Arete](N1)--(E2);
        \draw[Arete](N1)--(E3);
        \draw[Arete](N2)--(E4);
        \draw[Arete](N3)--(E5);
        \draw[Arete](N3)--(E6);
        \draw[Arete](N2)--(N3);
        \draw[Arete](N1)edge[bend left=20] node[]{}(N2);
        \draw[Arete](N1)edge[bend right=20] node[]{}(N2);
    \end{tikzpicture}}\end{split}\,,\quad
    \begin{split}\scalebox{.25}{\begin{tikzpicture}
        \node[Feuille](S1)at(0,0){};
        \node[Feuille](S2)at(2,0){};
        \node[Feuille](S3)at(3,0){};
        \node[Feuille](S4)at(4,0){};
        \node[Feuille](S5)at(6,0){};
        \node[Feuille](S6)at(7,0){};
        \node[Operateur,minimum width=3cm](N1)at(1,-5){\begin{math}\La\end{math}};
        \node[Operateur,minimum width=3cm](N2)at(3,-2){\begin{math}\La\end{math}};
        \node[Operateur,minimum width=3cm](N3)at(6,-5){\begin{math}\La\end{math}};
        \node[Feuille](E1)at(0,-7){};
        \node[Feuille](E2)at(1,-7){};
        \node[Feuille](E3)at(2,-7){};
        \node[Feuille](E4)at(5,-7){};
        \node[Feuille](E5)at(6,-7){};
        \node[Feuille](E6)at(7,-7){};
        \draw[Arete](N1)--(S1);
        \draw[Arete](N2)--(S2);
        \draw[Arete](N2)--(S3);
        \draw[Arete](N2)--(S4);
        \draw[Arete](N3)--(S5);
        \draw[Arete](N3)--(S6);
        \draw[Arete](N1)--(E1);
        \draw[Arete](N1)--(E2);
        \draw[Arete](N1)--(E3);
        \draw[Arete](N3)--(E4);
        \draw[Arete](N3)--(E5);
        \draw[Arete](N3)--(E6);
        \draw[Arete](N2)--(N3);
        \draw[Arete](N1)edge[bend left=20] node[]{}(N2);
        \draw[Arete](N1)edge[bend right=20] node[]{}(N2);
    \end{tikzpicture}}\end{split}\,,\quad
    \begin{split}\scalebox{.25}{\begin{tikzpicture}
        \node[Feuille](S1)at(0,0){};
        \node[Feuille](S2)at(1,0){};
        \node[Feuille](S3)at(2,0){};
        \node[Feuille](S4)at(5,0){};
        \node[Feuille](S5)at(6,0){};
        \node[Feuille](S6)at(7,0){};
        \node[Operateur,minimum width=3cm](N1)at(1,-2){\begin{math}\La\end{math}};
        \node[Operateur,minimum width=3cm](N2)at(3,-5){\begin{math}\La\end{math}};
        \node[Operateur,minimum width=3cm](N3)at(6,-2){\begin{math}\La\end{math}};
        \node[Feuille](E1)at(0,-7){};
        \node[Feuille](E2)at(2,-7){};
        \node[Feuille](E3)at(3,-7){};
        \node[Feuille](E4)at(4,-7){};
        \node[Feuille](E5)at(6,-7){};
        \node[Feuille](E6)at(7,-7){};
        \draw[Arete](N1)--(S1);
        \draw[Arete](N1)--(S2);
        \draw[Arete](N1)--(S3);
        \draw[Arete](N3)--(S4);
        \draw[Arete](N3)--(S5);
        \draw[Arete](N3)--(S6);
        \draw[Arete](N1)--(E1);
        \draw[Arete](N2)--(E2);
        \draw[Arete](N2)--(E3);
        \draw[Arete](N2)--(E4);
        \draw[Arete](N3)--(E5);
        \draw[Arete](N3)--(E6);
        \draw[Arete](N2)--(N3);
        \draw[Arete](N1)edge[bend left=20] node[]{}(N2);
        \draw[Arete](N1)edge[bend right=20] node[]{}(N2);
    \end{tikzpicture}}\end{split}\,,\quad
    \begin{split}\scalebox{.25}{\begin{tikzpicture}
        \node[Feuille](S1)at(0,0){};
        \node[Feuille](S2)at(1,0){};
        \node[Feuille](S3)at(2,0){};
        \node[Feuille](S4)at(4,0){};
        \node[Feuille](S5)at(6,0){};
        \node[Feuille](S6)at(7,0){};
        \node[Operateur,minimum width=3cm](N1)at(1,-2){\begin{math}\La\end{math}};
        \node[Operateur,minimum width=3cm](N2)at(3,-5){\begin{math}\La\end{math}};
        \node[Operateur,minimum width=3cm](N3)at(6,-8){\begin{math}\La\end{math}};
        \node[Feuille](E1)at(0,-10){};
        \node[Feuille](E2)at(2,-10){};
        \node[Feuille](E3)at(3,-10){};
        \node[Feuille](E4)at(5,-10){};
        \node[Feuille](E5)at(6,-10){};
        \node[Feuille](E6)at(7,-10){};
        \draw[Arete](N1)--(S1);
        \draw[Arete](N1)--(S2);
        \draw[Arete](N1)--(S3);
        \draw[Arete](N2)--(S4);
        \draw[Arete](N3)--(S5);
        \draw[Arete](N3)--(S6);
        \draw[Arete](N1)--(E1);
        \draw[Arete](N2)--(E2);
        \draw[Arete](N2)--(E3);
        \draw[Arete](N3)--(E4);
        \draw[Arete](N3)--(E5);
        \draw[Arete](N3)--(E6);
        \draw[Arete](N2)--(N3);
        \draw[Arete](N1)edge[bend left=20] node[]{}(N2);
        \draw[Arete](N1)edge[bend right=20] node[]{}(N2);
    \end{tikzpicture}}\end{split}\,.
\end{equation}
\medskip

We can observe that $\FHeap_\gamma$ is a stiff PRO because $\equiv$
satisfies \ref{item:propriete_bonnes_congruences_1} and
\ref{item:propriete_bonnes_congruences_2} and
$\FHeap_0 = \Heap_0$. We call $\FHeap_\gamma$
the {\em PRO of heaps of friable pieces} of width $\gamma + 1$. This
terminology is justified by the following observation. Any piece of
width $\gamma + 1$ (depicted by
$\scalebox{.4}{\tikz \node[Domino2](0)at(0,0){};}$) consists in
$\gamma + 1$ small pieces, called {\em bursts}, glued together. This forms
a {\em friable piece} (depicted, for $\gamma = 2$ for instance, by
$\scalebox{.4}{\begin{tikzpicture}
    \node[Domino1](0)at(0,0){};
    \node[Domino1](1)at(1,0){};
    \node[Domino1](2)at(2,0){};
\end{tikzpicture}}$).
The congruence $\equiv$ of $\Heap_\gamma$ can be interpreted by letting
all pieces break under gravity, separating the bursts constituting these.
For instance, the prographs of \eqref{equ:exemple_classe_FHeap_3},
respectively, encoded by the heaps of pieces
\begin{equation} \label{equ:exemple_classe_Heap_3}
    \begin{split}\scalebox{.4}{\begin{tikzpicture}
        \node[Domino3](0)at(0,0){};
        \node[Domino3](1)at(1,.5){};
        \node[Domino3](2)at(3,1){};
    \end{tikzpicture}}\end{split}\,,\quad
    \begin{split}\scalebox{.4}{\begin{tikzpicture}
        \node[Domino3](0)at(0,0){};
        \node[Domino3](1)at(3,0){};
        \node[Domino3](2)at(1,.5){};
    \end{tikzpicture}}\end{split}\,,\quad
    \begin{split}\scalebox{.4}{\begin{tikzpicture}
        \node[Domino3](0)at(0,0){};
        \node[Domino3](1)at(-1,.5){};
        \node[Domino3](2)at(2,.5){};
    \end{tikzpicture}}\end{split}\,,\quad
    \begin{split}\scalebox{.4}{\begin{tikzpicture}
        \node[Domino3](0)at(0,0){};
        \node[Domino3](1)at(1,-.5){};
        \node[Domino3](2)at(3,-1){};
    \end{tikzpicture}}\end{split}\,,
\end{equation}
all become the heap of friable pieces
\begin{equation} \label{equ:exemple_empilement_friable}
    \begin{split}\scalebox{.4}{\begin{tikzpicture}
    \node[Domino1](0)at(0,0){};
    \node[Domino1](1)at(1,0){};
    \node[Domino1](2)at(2,0){};
    \node[Domino1](3)at(3,0){};
    \node[Domino1](4)at(4,0){};
    \node[Domino1](5)at(5,0){};
    \node[Domino1](6)at(1,.5){};
    \node[Domino1](7)at(2,.5){};
    \node[Domino1](8)at(3,.5){};
    \end{tikzpicture}}\end{split}
\end{equation}
obtained by replacing each piece of any heap of pieces of
\eqref{equ:exemple_classe_Heap_3} by friable pieces.
\medskip

The grading $\omega$ of $\FHeap_\gamma$ is the one inherited from
the one of $\Heap_\gamma$. This grading is still well-defined in
$\Heap_\gamma$ since any $\equiv$-equivalence class contains prographs
of a same degree. Since the reduced elements of $\FHeap_\gamma$ have
no wire, they are encoded by horizontally connected heaps of friable
pieces.
\medskip

Besides, $\FHeap_\gamma$ admits the following alternative description
using the $\MvP$ construction (see Section \ref{subsubsec:monoides_vers_PROs}).
Indeed, $\FHeap_\gamma$ is the sub-PRO of $\MvP(\EnsNat)$ generated
by $1^{\gamma+1}$, where $\EnsNat$ denotes here the additive monoid
of nonnegative integers and $1^{\gamma+1}$ denotes the sequence of $\gamma+1$
occurrences of $1 \in \EnsNat$. The correspondence between heaps of
friable pieces and words of integers of this second description is clear
since any element $x$ of the sub-PRO of $\MvP(\EnsNat)$ generated by
$1^{\gamma+1}$ encodes a heap of friable pieces consisting, from left
to right, in columns of $x_i$ bursts for $i \in [n]$, where $n$ is the
length of $x$. For instance, the word $122211$ encodes the heap of
friable pieces of \eqref{equ:exemple_empilement_friable}.
\medskip

\subsubsection{Hopf algebra}
By Theorem \ref{thm:PRO_vers_AHC_congruence} and
Proposition \ref{prop:PRO_vers_AHC_graduation}, $\PvH(\FHeap_\gamma)$ is
a combinatorial Hopf subalgebra of $\PvH(\Heap_\gamma)$. The bases
of $\PvH(\FHeap_\gamma)$ are indexed by horizontally connected heaps
of friable pieces of width $\gamma + 1$ where the degree of a basis
element $\Tbf_\Lambda$ is the number of pieces of $\Lambda$.
\medskip

\subsubsection{Coproduct}
The coproduct of $\PvH(\FHeap_\gamma)$ can be described with the aid of
the interpretation of $\FHeap_\gamma$ as a sub-PRO of $\MvP(\N)$.
Indeed, if $\Lambda$ is an horizontally connected heap of friable pieces,
by Proposition \ref{prop:PRO_vers_AHC_congruence_coproduit},
\begin{equation}
    \Delta(\Tbf_\Lambda) =
    \sum_{\substack{\Lambda_1, \Lambda_2 \in \FHeap_\gamma \\
    \Lambda = \Lambda_1 + \Lambda_2}}
    \Tbf_{\Lambda'_1} \otimes \Tbf_{\Lambda'_2},
\end{equation}
where $\Lambda_1 + \Lambda_2$ is the heap of friable pieces obtained
by stacking $\Lambda_2$ onto $\Lambda_1$ and where $\Lambda'_1$
(resp. $\Lambda'_2$) is the readjustment of $\Lambda_1$ (resp. $\Lambda_2$)
so that it is horizontally connected. For instance, we have
in $\PvH(\FHeap_1)$
\begin{equation}
    \Tbf_{\scalebox{.4}{\begin{tikzpicture}
        \node[Domino1](0)at(0,0){};
        \node[Domino1](1)at(-1,-.5){};
        \node[Domino1](2)at(0,-.5){};
        \node[Domino1](3)at(-1,-1){};
        \node[Domino1](4)at(0,-1){};
        \node[Domino1](5)at(1,-1){};
    \end{tikzpicture}}}
    \enspace = \enspace
    \Sbf_{\scalebox{.4}{\begin{tikzpicture}
        \node[Domino2](0)at(0,0){};
        \node[Domino2](1)at(1,-.5){};
        \node[Domino2](2)at(0,-1){};
    \end{tikzpicture}}}
    \enspace + \enspace
    \Sbf_{\scalebox{.4}{\begin{tikzpicture}
        \node[Domino2](0)at(0,0){};
        \node[Domino2](1)at(-1,-.5){};
        \node[Domino2](2)at(-1,-1){};
    \end{tikzpicture}}}
    \enspace + \enspace
    \Sbf_{\scalebox{.4}{\begin{tikzpicture}
        \node[Domino2](0)at(0,0){};
        \node[Domino2](1)at(0,-.5){};
        \node[Domino2](2)at(1,-1){};
    \end{tikzpicture}}}\,,
\end{equation}
\begin{multline}
    \Delta
    \Tbf_{\scalebox{.4}{\begin{tikzpicture}
        \node[Domino1](0)at(0,0){};
        \node[Domino1](1)at(-1,-.5){};
        \node[Domino1](2)at(0,-.5){};
        \node[Domino1](3)at(-1,-1){};
        \node[Domino1](4)at(0,-1){};
        \node[Domino1](5)at(1,-1){};
    \end{tikzpicture}}}
    \enspace = \enspace
    \Tbf_{\emptyset} \otimes
    \Tbf_{\scalebox{.4}{\begin{tikzpicture}
        \node[Domino1](0)at(0,0){};
        \node[Domino1](1)at(-1,-.5){};
        \node[Domino1](2)at(0,-.5){};
        \node[Domino1](3)at(-1,-1){};
        \node[Domino1](4)at(0,-1){};
        \node[Domino1](5)at(1,-1){};
    \end{tikzpicture}}}
    \enspace + \enspace
    \Tbf_{\scalebox{.4}{\begin{tikzpicture}
        \node[Domino1](1)at(-1,-.5){};
        \node[Domino1](2)at(0,-.5){};
    \end{tikzpicture}}}
    \otimes
    \Tbf_{\scalebox{.4}{\begin{tikzpicture}
        \node[Domino1](2)at(0,-.5){};
        \node[Domino1](3)at(-1,-1){};
        \node[Domino1](4)at(0,-1){};
        \node[Domino1](5)at(1,-1){};
    \end{tikzpicture}}}
    \enspace + \enspace
    \Tbf_{\scalebox{.4}{\begin{tikzpicture}
        \node[Domino1](2)at(0,-.5){};
        \node[Domino1](3)at(-1,-1){};
        \node[Domino1](4)at(0,-1){};
        \node[Domino1](5)at(1,-1){};
    \end{tikzpicture}}}
    \otimes
    \Tbf_{\scalebox{.4}{\begin{tikzpicture}
        \node[Domino1](1)at(-1,-.5){};
        \node[Domino1](2)at(0,-.5){};
    \end{tikzpicture}}} \\
    \enspace + \enspace
    \Tbf_{\scalebox{.4}{\begin{tikzpicture}
        \node[Domino1](1)at(-1,-.5){};
        \node[Domino1](2)at(0,-.5){};
    \end{tikzpicture}}}
    \otimes
    \Tbf_{\scalebox{.4}{\begin{tikzpicture}
        \node[Domino1](1)at(-1,-.5){};
        \node[Domino1](2)at(0,-.5){};
        \node[Domino1](3)at(-1,-1){};
        \node[Domino1](4)at(0,-1){};
    \end{tikzpicture}}}
    \enspace + \enspace
    \Tbf_{\scalebox{.4}{\begin{tikzpicture}
        \node[Domino1](1)at(-1,-.5){};
        \node[Domino1](2)at(0,-.5){};
        \node[Domino1](3)at(-1,-1){};
        \node[Domino1](4)at(0,-1){};
    \end{tikzpicture}}}
    \otimes
    \Tbf_{\scalebox{.4}{\begin{tikzpicture}
        \node[Domino1](1)at(-1,-.5){};
        \node[Domino1](2)at(0,-.5){};
    \end{tikzpicture}}}
    \enspace + \enspace
    \Tbf_{\scalebox{.4}{\begin{tikzpicture}
        \node[Domino1](0)at(0,0){};
        \node[Domino1](1)at(-1,-.5){};
        \node[Domino1](2)at(0,-.5){};
        \node[Domino1](3)at(-1,-1){};
        \node[Domino1](4)at(0,-1){};
        \node[Domino1](5)at(1,-1){};
    \end{tikzpicture}}}
    \otimes \Tbf_{\emptyset}.
\end{multline}
\medskip

\subsubsection{Dimensions}
\begin{Proposition} \label{prop:dimensions_PvH_FHeap_gamma}
    For any nonnegative integer $\gamma$, the $n$-th homogeneous
    component of $\PvH(\FHeap_\gamma)$ has dimension $(\gamma + 2)^{n - 1}$.
\end{Proposition}
\begin{proof}
    Let us show that there are $(\gamma + 2)^{n - 1}$ reduced
    elements in $\FHeap_\gamma$.
    \smallskip

    Let $x$ be a prograph of $\Heap_\gamma$. We say that $x$ is
    a {\em $\gamma$-falling staircase} if $x$ is of the form
    \begin{equation} \label{equ:decomposition_escalier_FHeap}
        x = (\Unite_{p_1} * \La * \Unite_{q_1}) \circ
        (\Unite_{p_1 + p_2} * \La * \Unite_{q_2}) \circ \dots
        \circ
        (\Unite_{p_1 + p_2 + \dots + p_\ell} * \La * \Unite_{q_\ell}),
    \end{equation}
    where $p_1 = 0$, $q_\ell = 0$, and $0 \leq p_i \leq \gamma$
    for all $2 \leq i \leq \ell$. Any $\gamma$-falling staircase
    can be encoded by the sequence $(p_2, \dots, p_\ell)$ involved
    in its decomposition \eqref{equ:decomposition_escalier_FHeap}.
    For instance, for
    \begin{equation}
        \begin{split}x := \enspace\end{split}
        \begin{split}\scalebox{.25}{\begin{tikzpicture}
            \node[Feuille](S1)at(0,0){};
            \node[Feuille](S2)at(1,0){};
            \node[Feuille](S3)at(2,0){};
            \node[Feuille](S4)at(4,0){};
            \node[Feuille](S5)at(6,0){};
            \node[Feuille](S6)at(7,0){};
            \node[Operateur,minimum width=3cm](N1)at(1,-2)
                    {\begin{math}\La\end{math}};
            \node[Operateur,minimum width=3cm](N2)at(3,-5)
                    {\begin{math}\La\end{math}};
            \node[Operateur,minimum width=3cm](N3)at(6,-8)
                    {\begin{math}\La\end{math}};
            \node[Feuille](E1)at(0,-10){};
            \node[Feuille](E2)at(2,-10){};
            \node[Feuille](E3)at(3,-10){};
            \node[Feuille](E4)at(5,-10){};
            \node[Feuille](E5)at(6,-10){};
            \node[Feuille](E6)at(7,-10){};
            \draw[Arete](N1)--(S1);
            \draw[Arete](N1)--(S2);
            \draw[Arete](N1)--(S3);
            \draw[Arete](N2)--(S4);
            \draw[Arete](N3)--(S5);
            \draw[Arete](N3)--(S6);
            \draw[Arete](N1)--(E1);
            \draw[Arete](N2)--(E2);
            \draw[Arete](N2)--(E3);
            \draw[Arete](N3)--(E4);
            \draw[Arete](N3)--(E5);
            \draw[Arete](N3)--(E6);
            \draw[Arete](N2)--(N3);
            \draw[Arete](N1)edge[bend left=20] node[]{}(N2);
            \draw[Arete](N1)edge[bend right=20] node[]{}(N2);
        \end{tikzpicture}}\end{split}\,,
    \end{equation}
    we have
    $x = (\Unite_0 * \La * \Unite_3) \circ
    (\Unite_{0 + 1} * \La * \Unite_2) \circ
    (\Unite_{0 + 1 + 2} * \La * \Unite_0)$ so that $x$
    is a $2$-falling staircase encoded by the sequence $(1, 2)$.
    \smallskip

    Moreover, we say that $x$ is a {\em $\gamma$-standard form} if
    $x$ is an horizontal composition of $\gamma$-falling staircases.
    Any $\gamma$-standard form can be encoded by the sequence
    of the sequences encoding, from left to right, its falling
    staircases.
    \smallskip

    By definition of $\FHeap_\gamma$ as a quotient of $\Heap_\gamma$,
    one can observe that two different $\gamma$-standard forms are sent
    to two different heaps of friable pieces by the canonical
    surjection $\Heap_\gamma \twoheadrightarrow \FHeap_\gamma$.
    Besides, one can straightforwardly prove by induction on the
    degree that any reduced element of $\Heap_\gamma$ is
    $\equiv$-equivalent to a $\gamma$-falling staircase.
    \smallskip

    This shows that the reduced elements of $\FHeap_\gamma$ are
    in bijection with the standard forms of $\Heap_\gamma$ of a
    same degree. Hence, the reduced elements of $\FHeap_\gamma$ of
    degree $n$ can be encoded by sequences of $k$ words on the
    alphabet $\{0\} \cup [\gamma]$, having a total of $n - k$
    letters. Whence the result.
\end{proof}
\medskip

\subsubsection{Miscellaneous properties}
By the dimensions of $\PvH(\FHeap_\gamma)$ provided by
Proposition \ref{prop:dimensions_PvH_FHeap_gamma}, as a graded vector
space, $\PvH(\FHeap_\gamma)$ is the $\gamma + 1$-st tensor power of the
underlying vector space of $\SymNC$. Indeed, the $n$-th homogeneous
components of these two spaces have the same dimension. Besides, notice
that since $\FHeap_\gamma$ is by definition a sub-PRO of the PRO obtained
by applying the construction $\MvP$ on a commutative monoid,
$\PvH(\FHeap_\gamma)$ is cocommutative.
\medskip

\section*{Concluding remarks and perspectives}
We have defined a construction $\PvH$ establishing a new link between the
theory of PROs and the theory of combinatorial Hopf algebras, by
generalizing a former construction from operads to bialgebras. By
the way, we have exhibited the so-called stiff PROs which is the most
general class of PROs for which our construction works.
\medskip

By using $\PvH$, we have defined some new and recovered some
already known combinatorial Hopf algebras by starting with very
simple PROs. Nevertheless, we are very far from having exhausted
the possibilities, and it would not be surprising that $\PvH$ could
reconstruct some other known Hopf algebras, maybe in an unexpected basis.
\medskip

Computing the Hilbert series of a combinatorial Hopf algebra is,
usually, a routine work. Nevertheless, in the general case, it is
very difficult to compute the Hilbert series of $\PvH \Pca$ when
$\Pca$ is a free PRO. Indeed, this computation requires to know, given
a free PRO $\Pca$, the series
\begin{equation}
    F_\Pca(t) :=
    \sum_{x \in \Reduit(\Pca)} t^{\deg(x)},
\end{equation}
which seems difficult to explicitly describe in general and is
not known to the knowledge of the authors.
\medskip

As an other perspective, it is conceivable to go further in the study of
the algebraic structure of the bialgebras obtained by $\PvH$. The question
of the potential autoduality of $\PvH \Pca$ depending on some conditions on
the PRO $\Pca$ is noteworthy. A way to solve this problem is to provide
enough conditions on $\Pca$ to endow $\PvH \Pca$ with a bidendriform
bialgebra structure \cite{Foi07}. In such algebraic structures, there are
two products $\prec$ and $\succ$ and two coproducts $\Delta_\prec$ and
$\Delta_\succ$ satisfying some precise axioms. This way to solve this
perspective is based upon the fact that any bidendriform bialgebra is free
and self-dual as a bialgebra \cite{Foi07}.
\medskip

\bibliographystyle{alpha}
\bibliography{Bibliographie}

\end{document}